\documentclass[8pt]{article}

\usepackage[utf8]{inputenc}
\usepackage[T2A]{fontenc}
\usepackage[english]{babel}
\usepackage[margin=1.2cm]{geometry}
\usepackage{mathtools}

\usepackage{pgf,tikz,pgfplots}
\usetikzlibrary{arrows.meta}
\pgfplotsset{compat=1.15}
\usetikzlibrary{arrows,patterns}
\usepackage{float}

\usepackage{amsmath,amsthm,amssymb}
\usepackage{hyperref}

\newtheorem{theorem}{Theorem}[section]
\newtheorem{definition}[theorem]{Definition}

\newtheorem{lemma}[theorem]{Lemma}
\newtheorem{proposition}[theorem]{Proposition}
\newtheorem{corollary}[theorem]{Corollary}

\newtheorem{problem}[theorem]{Problem}

\theoremstyle{remark}

\DeclareMathOperator{\Int}{Int}
\DeclareMathOperator{\conv}{conv}

\newcommand\R{\mathbb{R}}

\newcommand{\St}{{\mathcal St}}
\newcommand{\Err}{{\mathcal Err}}

        \newcommand{\eps}{\varepsilon}
        \newcommand{\HH}{{\mathcal H}}
        \renewcommand{\H}{\HH^1}
        
        \newcommand{\defeq}{:=}
        \newcommand{\forget}[1]{}

        \def\dist{\mathrm{dist}\,}

        \def\turn{\mathrm{turn}\,}

        \def\Int{\mathrm{Int}\,}

\title{Maximal distance minimizers for a rectangle}

\author{D.D. Cherkashin, A.S. Gordeev, G.A. Strukov, Y.I. Teplitskaya}

\begin{document}

\maketitle

\begin{abstract}
\emph{A maximal distance minimizer} for a given compact set $M \subset \mathbb{R}^2$ and some given $r > 0$ is a set having the minimal length (one-dimensional Hausdorff measure) over the class	of closed connected sets $\Sigma \subset \mathbb{R}^2$ satisfying the inequality
\[
\max_{y\in M} \dist (y, \Sigma) \leq r.
\]
This paper deals with the set of maximal distance minimizers for a rectangle $M$ and small enough $r$.

\end{abstract}

\section{Introduction}
Let $M$ be a compact planar set. For every compact set $\Sigma \subset \mathbb{R}^2$ define the maximal distance functional
\[
F_M(\Sigma) := \max_{y\in M} \dist (y, \Sigma),
\]
further the value $F_M(\Sigma)$ is called \textit{energy} of $\Sigma$.

The following problem appeared in~\cite{buttazzo2002optimal} and later has been studied in~\cite{miranda2006one,paolini2004qualitative}.
\begin{problem}
For a given compact set $M \subset \R^2$ and $r > 0$ find a connected set $\Sigma$ of minimal length (one-dimensional Hausdorff measure $\H$) such that
\[
F_M(\Sigma) \leq r.
\]
\label{TheProblem}
\end{problem}

We use the following general results (see~\cite{paolini2004qualitative}).
%; %they also hold in $n$-dimensional case, see~\cite{paolini2004qualitative}.~\cite{teplitskaya2018regularity,teplitskaya2019regularity}
\begin{itemize}
    \item A solution of Problem~\ref{TheProblem} exists and has a finite length. We further call it a \textit{minimizer}.
    \item A minimizer has no cycles (homeomorphic images of $S^1$).
    \item The set of minimizers coincides with the set of solutions of the dual problem: minimize $F_M$ over all compact connected sets $\Sigma$ with
prescribed bound on the total length $\H(\Sigma) \leq l$. This item explains the name of the problem.
    %\item $\Sigma$ consist of a finite number of curves (injective images of the segment $[0,1]$).
\end{itemize}
%The third item explains the name of the problem.

The result of our paper is the following
\begin{theorem}
Let $M = A_1A_2A_3A_4$ be a rectangle, $0 < r < r_0(M)$. Then a maximal distance minimizer has the following topology with 21 segment, depicted in the left part of Fig.~\ref{rectangle}. The middle part of the picture contains enlarged fragment of the minimizer near $A_1$; the labeled angles are equal to $\frac{2\pi}{3}$. The rightmost part contains much more enlarged fragment of minimizer near $A_1$.

All maximal distance minimizers have length approximately $Per - 8.473981r$, where $Per$ is the perimeter of the rectangle.
\label{rectangleT}
\end{theorem}

In fact, every maximal distance minimizer is very close (in the sense of Hausdorff distance) to the one depicted in the picture.

\begin{figure}[H]
    \centering
    \hfill
    \begin{tikzpicture}[scale=0.4]
        
        \def\w{8}
        \def\h{4.5}
        
        \def\t{.85}
        \def\c{}
        
        \foreach\x in {-1,1} {
            \foreach\y in {-1,1} {
                \draw[blue, thick]  ({\x*(-\w+1 - .292893218813455)},
                        {\y*(-\h+1 - .2928932188134548}) -- 
                       ({\x*(-\w+2 - \t)},
                        {\y*(-\h+2 - \t)});
                \draw[blue, thick]  ({\x*(-\w+2 - \t)},
                        {\y*(-\h+2 - \t)}) -- 
                       ({\x*(-\w+1)}, 
                        {\y*(-\h+2 -.5857864376269108)});
                \draw[blue, thick]  ({\x*(-\w+2 - \t)},
                        {\y*(-\h+2 - \t)}) -- 
                       ({\x*(-\w+2 - .5857864376269095)},
                        {\y*(-\h+1)});
            }
        }
        \draw[blue, thick]  (-\w+2 - .5857864376269095,-\h+1) -- 
               ( 0,-\h);
        \draw[blue, thick]  ( 0,-\h) -- 
               ( \w-2 + .5857864376269095,-\h+1);
        \draw[blue, thick]  (-\w+1, -\h+2 - .5857864376269108) -- 
               (-\w+1,  \h-2 + .5857864376269108);
        \draw[blue, thick]  (-\w+2 - .5857864376269095, \h-1) -- 
               ( \w-2 + .5857864376269095, \h-1);
        \draw[blue, thick]  ( \w-1, -\h+2 - .5857864376269108) -- 
               ( \w-1,  \h-2 + .5857864376269108);
               
        \fill[white] (1, -\h) arc (0:180:1);
        
        \draw [very thick] (-\w,-\h) rectangle (\w,\h);
        
        \draw (-\w, -\h) node[below left]{$A_1$};
        \draw (-\w, \h) node[above left]{$A_2$};
        \draw (\w, \h) node[above right]{$A_3$};
        \draw (\w, -\h) node[below right]{$A_4$};
        
    \end{tikzpicture}
\hfill
    \begin{tikzpicture}[scale=2]
        
        \coordinate (V) at (1.10837, 1.10837);
        \coordinate (Q) at (.7071, .7071);
        \coordinate (C) at (.73, .73);
        
        \draw[dashed] (0,0) -- (Q);
        \draw[thick, dotted] (0, 1) arc (90:0:1);
        \draw[thick, dotted] (.414, 0) arc (180:55:1);
        \draw[dashed] (1.414, 0) -- (1.544739754593147, 0.9914448613738104) node[pos=0.5, left]{$r$};
        \draw[thick, dotted] (0, .414) arc (-90:35:1);
        \draw[dashed] (0, 1.414) -- (0.9914448613738104,1.544739754593147) node[pos=0.5,below]{$r$};
        
        %\draw[very thick] (-5, 1) -- (-5,3) -- (-3, 3);
        \draw [very thick] (2,0) -- (0,0) -- (0,2);
        \draw (0, 0) node[below left]{$A_1$};
        \draw [very thick, blue] (0.9914448613738104,2)-- (0.9914448613738104,1.544739754593147);
        \draw [very thick, blue] (V)-- (0.9914448613738104,1.544739754593147);
        \draw [very thick, blue] (V)-- (1.5447397545931463,0.9914448613738106);
        \draw [very thick, blue] (1.5447397545931463,0.9914448613738106)-- (2,0.9914448613738102);
        \draw [very thick, blue] (Q)-- (V);    
        % \draw[very thick, blue] (-4.292893218813455, 2.2928932188134548)-- (-3.912466374682504, 1.9124663746825037);
        % \draw[very thick, blue] (-3.912466374682504, 1.9124663746825037)-- 
        % (-4, 1.5857864376269108);
        % \draw[very thick, blue] (-3.912466374682504, 1.9124663746825037)-- (-3.5857864376269095, 2);
        % \draw[very thick, blue] (-3.5857864376269095, 2)-- (-3, 2);
        % \draw[very thick, blue] (-4, 1.5857864376269108)-- (-4,1);

        \def\rr{0.12}
        \draw[shift={(V)}] 
            (0.966*\rr ,-0.2588*\rr) arc(-15:105:\rr);
        \def\rr{0.09}
        \draw[shift={(V)}] 
            (0.966*\rr ,-0.2588*\rr) arc(-15:105:\rr);
        \def\rr{0.096}
        \draw[shift={(V)}] 
            (0.966*\rr ,-0.2588*\rr) arc(-15:-135:\rr);
        \def\rr{0.072}
        \draw[shift={(V)}] 
            (0.966*\rr ,-0.2588*\rr) arc(-15:-135:\rr);
        \def\rr{0.112}
        \draw[shift={(V)}] 
            (-0.2588*\rr ,0.966*\rr) arc(105:225:\rr);
        \def\rr{0.082}
        \draw[shift={(V)}] 
            (-0.2588*\rr ,0.966*\rr) arc(105:225:\rr);
        
        \node[right] at (0.9914448613738104,1.544739754593147) {\small $\approx \frac{11\pi}{12}$};
        \node[above] at (1.644739754593147, 0.9914448613738104) {\small $\approx \frac{11\pi}{12}$};
        
        \def\ll{0.051}
        \def\llll{0.2}
        \draw[line width = .7pt] (C) circle (\ll);
        \draw[shift={(C)}, line width = 2pt] (\ll * .866, \ll * -.5) -- (\llll * 0.866, \llll * -.5);
    \end{tikzpicture}
    \hfill
    \begin{tikzpicture}[scale=48]
        \coordinate (V) at (.7645532062, .76593);
        \coordinate (Q2) at (.723714, .725155);
        \coordinate (Q1) at (.707224, .706989);
        \coordinate (C) at (.73, .73);

        \draw[very thick, blue] (V) -- (Q2) -- (Q1);
        \fill[blue] (Q2) circle (0.0015);
        
        \draw[shift={(Q2)}, rotate=45] (0.0075, 0) arc (0 : -177 : 0.0075) node[right, pos=0.3]{$\approx 0.98\, \pi$};
        
        \fill[black, shift={(C)}, rotate=60] (-0.01, -0.048) -- (-0.013, -0.0715) -- (0.01, -0.06) -- (0.01, -0.048) -- cycle;
        \draw[black, line width=3] (C) circle (0.05);
        
        \fill[black, shift={(C)}, rotate=60] (0.0171, -0.046) arc (-70 : -110 : 0.05) -- (-0.0171, -0.049) arc (-120 : -60 : 0.0342) -- cycle;
    \end{tikzpicture}
    \hfill\ 
\caption{The minimizer for rectangle $M$ and $r < r_0(M)$.}
%\caption{Вид минимайзера для прямоугольного $M$ и $r < r_0(M)$.}

    \label{rectangle}
\end{figure}
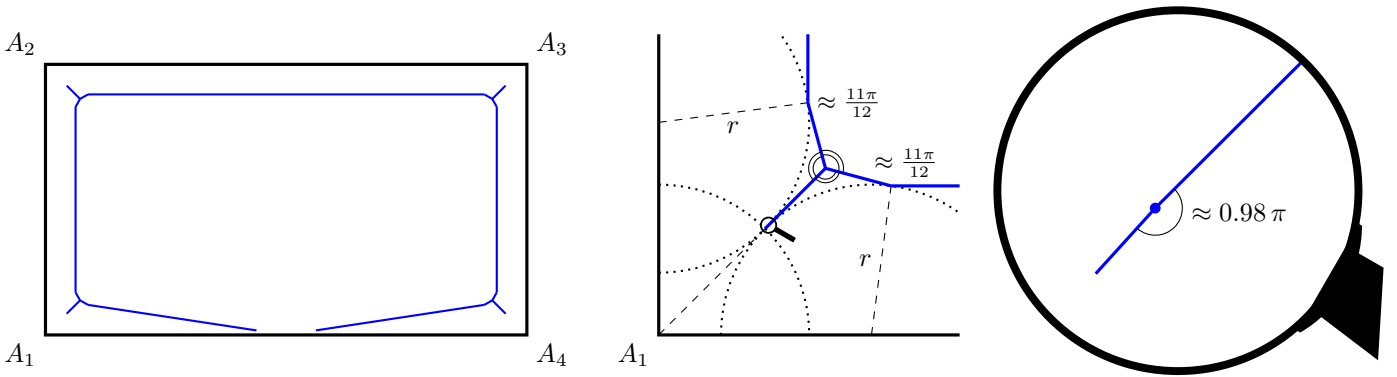

\paragraph{Structure of the paper.} Subsections~\ref{sect:energeticpoints}--\ref{sect:Steiner} contain notation and very basic facts on Problem~\ref{TheProblem} and Steiner problem. The rest of the paper is devoted to the proof of Theorem~\ref{rectangleT}.

\paragraph{Outline of the proof.} In Section~\ref{sect:adapting} we adapt several lemmas from~\cite{cherkashin2018horseshoe} to $M$ being a polygon. It turns out that $\Sigma$ lies in the $const \cdot r$-neighborhood of $M$ (see Corollary~\ref{CorKar}). 

Section~\ref{diff} outlines auxiliary results from~\cite{cherkashin2020minimizers} concerning differentiation in the picture. We derive the conditions on optimality of 
$\Sigma$ under motion of proper points $y(x)$ for energetic $x$.

Section~\ref{sect:nearangles} shows that the main difficulties appear in a neighborhood of the angles of $M$. Lemma~\ref{hujugl} shows that the intersection of $\Sigma$ with the $const \cdot r$-neighborhood of $A_i$ is connected. Then we enclose $\Sigma$ by adding a segment of length at most $2r$. The resulting set $\Sigma'$ has a large cycle $\mathcal{C}$. Subsection~\ref{Sect:studyofC} shows that $\mathcal{C}$ has prescribed combinatorial structure.

At this point we know that $\mathcal{C}$ covers all $M$ except $const\cdot r$-neighborhoods of angles. In Section~\ref{Application} we describe the computer search we use to find an optimal way to cover those neighborhoods. It is enough to consider a union of $\mathcal{C}$ with a Steiner tree on at most 4 special points.
Subsection~\ref{ErrorEstimation} deals with error terms. As a result of the search we conclude that every minimizer lies in a small box in configurational space around a certain point $p_0$.

In Section~\ref{sect:thefinal} we consider all Steiner topologies on at most 4 points one by one. It turns out that all the topologies except the answer provide longer configurations than $p_0$; the answer provides a shorter one.    

Finally, in Section~\ref{discussion} we discuss the uniqueness of a solution for a rectangle and refer to some related open questions.

\subsection{Energetic points}
\label{sect:energeticpoints}

Let $B_\rho(x)$ be the open ball of radius $\rho$ centered at a point $x$, and let $B_\rho(T)$
be the open $\rho$-neighborhood of a set $T$ i.e.
\[
B_\rho(T) := \bigcup_{x \in T} B_\rho(x).
\]
Note that the condition $F_M(\Sigma) \leq r$ in Problem~\ref{TheProblem} is equivalent to $M \subset \overline{B_r(\Sigma)}$.

The following definition plays central role in the study of minimizers, in particular in the proof of Theorem~\ref{rectangleT}.
\begin{definition}
Let $\Sigma$ be a minimizer for $M$ and $r$.
A point $x \in \Sigma$ is called \textit{energetic}, if for all $\rho > 0$ one has 
\[
F_M(\Sigma \setminus B_\rho(x)) > r.
\]
\end{definition}

The following properties of energetic points have been proved in~\cite{miranda2006one} and~\cite{paolini2004qualitative}.
\begin{itemize}
    \item[(a)] For every energetic point $x$ there is a point $y \in M$, such that $\dist(x,y) = r$ and $B_{r}(y)\cap \Sigma=\emptyset$. 
    We say that point $y$ is a \textit{corresponding point} to $x$ and denote it $y(x)$. Note that such $y$ may be not unique.
    \item[(b)] For every non-energetic point $x$ there is an $\varepsilon > 0$, such that $\Sigma \cap B_{\varepsilon}(x)$ is either a line segment or a \textit{regular tripod}, i.e. the union of three line segments with an endpoint in $x$ and relative angles of $2\pi/3$. If a point
$x$ is a center of a regular tripod, then it is called a \textit{Steiner point} (or a \textit{branching point}) of $\Sigma$.
    \label{Sbasis}
\end{itemize}

\begin{theorem}[Teplitskaya,~\cite{teplitskaya2018regularity,teplitskaya2019regularity}]
%Произвольный минимайзер максимального расстояния является объединением конечного числа кривых, имеющих одностороннюю касательную в каждой точке.
%Every maximal distance minimizer is a union of a finite number of curves, having one-sided tangent lines at each point.
Let $\Sigma$ be a maximal distance minimizer for a compact set $M \subset \mathbb{R}^2$, $r > 0$. We say that the ray $ (ax] $ is a \textit{tangent ray} of the set $ \Sigma $ at the point $ x\in \Sigma $ if there exists a non  stabilized sequence of points $ x_k \in \Sigma $ such that $ x_k \rightarrow x $ and $ \angle x_kxa \rightarrow 0 $. Then \begin{itemize}
\item[(i)] $\Sigma$ is a union of a finite number of injective images of the segment $[0,1]$;
\item[(ii)] the angle between each pair of tangent rays at every point of $\Sigma$ is greater or equal to $2\pi/3$;
\item[(iii)] the number of tangent rays at every point of $\Sigma$ is not greater than $3$. If it is equal to $3$, then there exists such a neighbourhood of $x$ that the arcs in it coincide with line segments and the pairwise angles between them are equal to $2\pi/3$. 
\end{itemize}
\end{theorem}

\subsection{Notation}

For a given set $X \subset \mathbb{R}^2$ we denote by $\overline X$ its closure, by $\Int(X)$ its interior and by $\partial X$ its topological
boundary.

For given points $B$, $C \in \mathbb{R}^2$ we use the notation $[BC]$, $[BC)$ and $(BC)$ for the corresponding closed line
segment, ray and line respectively. We denote by $]BC]$ and $]BC[$ the corresponding semiopen and
open segments, and by $|BC|$ the length of these segments.

For rays $[BC)$, $[CD)$ let $\angle \left([BC), [CD)\right)$ stand for the directed angle from $[BC)$ to $[CD)$ with respect
to the clockwise orientation. By definition $\angle \left([BC), [CD)\right) \in [0,2\pi)$.

For a polygonal chain $B_1,\dots, B_n$ define its \textit{turning} as follows
\[
\turn B_1,\dots, B_n := \sum_{i=1}^{n-2} \angle \left([B_iB_{i+1}),[B_{i+1},B_{i+2})\right).
\]
Note that for a closed (i.e. $B_1 = B_n$) and non-self-intersecting chain its turning is equal to $\pm 2\pi$.

For a $\sigma \subset \Sigma$ define the part of $M$, which is covered by $\sigma$:
\[
\gamma (\sigma) := M \cap \overline{B_r(\sigma)}.
\]
If $y\in M$ ($Y \subset M$) is contained in $\gamma (\sigma)$ (is a subset of $\gamma (\sigma)$) we say that $\sigma$ \textit{covers} $y$ ($Y$).

We use denotation $o_x(1)$ for a quantity, which tends to zero with $x$. Note that $o(x) = x o_x(1)$.

\subsection{Steiner problem}
\label{sect:Steiner}

The Steiner problem has several different but more or less equivalent formulations; we need the following one.
Let $S$ be a finite planar set of points, the problem is to find a connected set $\St$ of minimal length such that $S \subset \St$. Vertices from $S$ are called \textit{terminals}.
It is well-known that a solution exists, and may be not unique. 

It is also known that $\St$ consists of line segments and the angle between two segments adjacent to the same vertex is greater or equal than $2\pi/3$.
Let us call a \textit{Steiner} (or \textit{branching}) point such a point of $\Sigma$ that does not belong to $S$ and which is
not an interior point of a segment of $\Sigma$. 
It turns out that a Steiner point $x$ is adjacent to exactly 3 line segments. Hence the angle between any pair of segments adjacent to $x$ is equal to $2\pi/3$.

In particular we need to consider this problem for $S = \{x_1,x_2,x_3\}$.
It is well-known that for such $S$ the Steiner problem has a unique solution. If $x_1,x_2,x_3$ form a triangle with all angles smaller than $2\pi/3$, then
\[
\St (x_1,x_2,x_3) = [x_1T] \cup [x_2T] \cup [x_3T],
\]
where $T$ is the point satisfies $\angle x_1Tx_2 = \angle x_1Tx_3 = \angle x_2Tx_3 = 2\pi/3$ ($T$ is usually called Fermat point or Torricelli point).
Consider outward equilateral triangle $x_1x_2X$ (i.e. $X$ and $x_3$ lie on different sides of $(x_1x_2)$). It is known that $X$, $T$ and $x_3$ belong to the same line.

In the other case there is an angle of size at least $2\pi/3$, say, $\angle x_1x_2x_3$. The only solution is
\[
\St (x_1,x_2,x_3) = [x_1x_2] \cup [x_2x_3].
\]

For a more detailed survey see book~\cite{hwang1992steiner} and paper~\cite{gilbert1968steiner}.
Further we use Melzak algorithm~\cite{melzak1961problem} to find Steiner trees on at most four terminals.

\section{Adapting of the horseshoe method}
\label{sect:adapting}

\begin{definition}
    Let $M$ be a closed convex curve. By $N$ we denote $\conv M$. By $M_r$ we denote $N \cap \partial B_r(M)$. Finally, $N_r$ is $\conv M_r$ (see Fig.~\ref{MMrNNr}).
\end{definition}

\begin{figure}[h]
    \centering
    \begin{tikzpicture}[scale=0.9]
    \fill[lightgray!50] plot [smooth cycle, tension=1] coordinates 
        {(0,0) (0,3) (4,3) (4.5, 0)};
    \draw[very thick] plot [smooth cycle, tension=1] coordinates 
        {(0,0) (0,3) (4,3) (4.5, 0)};
    \draw[very thick, pattern=dots] plot [smooth cycle, tension=1] coordinates 
        {(0.7,0.7) (0.7,2.3) (3.3,2.3) (3.8, 0.7)};
    \draw[<->] (2.2,-0.5) -- (2.2,0.3);
        
    \node[above] at (2, 3.6) {$M$};
    \node[above] at (2, 2.6) {$M_r$};
    \node[right] at (2.2, -0.1) {$r$};
    
    \draw[fill = lightgray!50] (6,1) rectangle (6.5,1.5);
    \draw[pattern=dots] (6,0) rectangle (6.5,0.5);
    \node[right] at (6.5, 1.25) {$N$};
    \node[right] at (6.5, 0.25) {$N_r$};
\end{tikzpicture}

    \caption{Definitions of $N$, $M_r$, $N$, and $N_r$}
    \label{MMrNNr}
\end{figure}
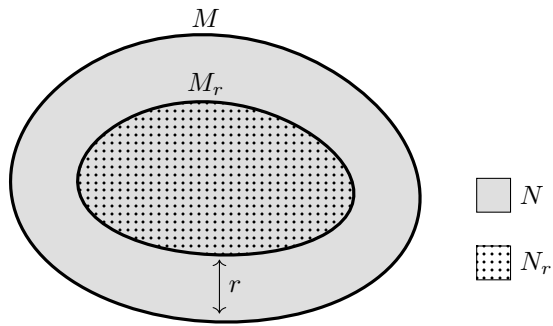

Finding the set of solutions for a given $M$ is a rather difficult problem. The only known nontrivial result is the following.

\begin{theorem}[Cherkashin -- Teplitskaya, 2018~\cite{cherkashin2018horseshoe}]
Suppose that $r > 0$, and $M$ is a convex closed curve with the radius of curvature at least $5r$. Let $\Sigma$ be an arbitrary minimizer for $M$. Then $\Sigma$ 
is the union of an arc of $M_r$ and two tangent segments to $M_r$ at the ends of the arc (i.e. it is a horseshoe, see Fig.~\ref{horseshoe}). 
In the case of $M$ being a circumference with radius $R$, one can weaken the assumption to $R > 4.98r$.
\label{horseshoeT}
\end{theorem}

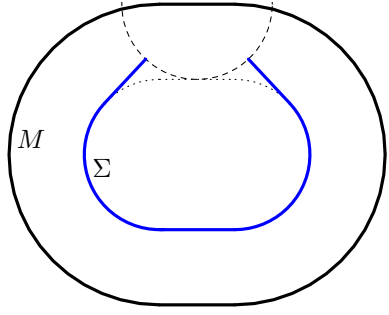
\begin{figure}[H]
	\begin{center}
		\definecolor{yqqqyq}{rgb}{0.,0.,1.}%{0.5019607843137255,0.,0.5019607843137255}
		\definecolor{xdxdff}{rgb}{0.,0.,0.}
		\definecolor{ffqqqq}{rgb}{0.,0.,1.}
    \definecolor{ttzzqq}{rgb}{0.,0.,0.}
    \begin{tikzpicture}[line cap=round,line join=round,>=triangle 45,x=1.0cm,y=1.0cm]
    \clip(-5.754809259201689,-2.03198488708635626) rectangle (0.592320066152345,2.0268250892462824);
    \draw [very thick, color=ttzzqq] (-3.,2.)-- (-2.,2.);
    \draw [very thick, color=ttzzqq] (-2.,-2.)-- (-3.,-2.);
    \draw [very thick, color=ffqqqq] (-3.,-1.)-- (-2.,-1.);
    \draw [dotted] (-3.,1.)-- (-2.,1.);
    \draw [shift={(-3.,0.)},very thick, color=ttzzqq]  plot[domain=1.5707963267948966:4.71238898038469,variable=\t]({1.*2.*cos(\t r)+0.*2.*sin(\t r)},{0.*2.*cos(\t r)+1.*2.*sin(\t r)});
    \draw [shift={(-3.,0.)},dotted]  plot[domain=1.5707963267948966:4.71238898038469,variable=\t]({1.*1.*cos(\t r)+0.*1.*sin(\t r)},{0.*1.*cos(\t r)+1.*1.*sin(\t r)});
    \draw [shift={(-2.,0.)},dotted]  plot[domain=-1.5707963267948966:1.5707963267948966,variable=\t]({1.*1.*cos(\t r)+0.*1.*sin(\t r)},{0.*1.*cos(\t r)+1.*1.*sin(\t r)});
    \draw [shift={(-2.,0.)},very thick,color=ttzzqq]  plot[domain=-1.5707963267948966:1.5707963267948966,variable=\t]({1.*2.*cos(\t r)+0.*2.*sin(\t r)},{0.*2.*cos(\t r)+1.*2.*sin(\t r)});
    \draw [dash pattern=on 2pt off 2pt] (-2.5008100281931047,2.) circle (1.cm);
    \draw [very thick,color=yqqqyq] (-3.7309420990521054,0.6824394829091469)-- (-3.1832495111022516,1.2690579009478955);
    \draw [very thick,color=yqqqyq] (-1.8178907146013357,1.2695061868000874)-- (-1.2695061868000874,0.6829193135917684);
    \draw [shift={(-2.,0.)},very thick,color=ffqqqq]  plot[domain=-1.5707963267948966:0.7517515639553677,variable=\t]({1.*1.*cos(\t r)+0.*1.*sin(\t r)},{0.*1.*cos(\t r)+1.*1.*sin(\t r)});
    \draw [shift={(-3.,0.)},very thick,color=ffqqqq]  plot[domain=2.390497746093128:4.771687465439039,variable=\t]({1.*1.*cos(\t r)+0.*1.*sin(\t r)},{0.*1.*cos(\t r)+1.*1.*sin(\t r)});
    %\begin{scriptsize}
    \draw[color=ttzzqq] (-4.7089822121844977,0.2025724503290166) node {$M$};
    %\draw [fill=xdxdff] (-2.5008100281931047,2.) circle (1.5pt);
    %\draw [fill=ffqqqq] (-3.1832495111022516,1.2690579009478955) circle (1.5pt);
    %\draw [fill=ffqqqq] (-1.8178907146013357,1.2695061868000874) circle (1.5pt);
    \draw[] (-3.763965847825129,-0.1818662773850787) node {$\Sigma$};
    %\draw[color=black] (-2.059620164509451,0.8028001329511188) node {$M_r$};
    %\end{scriptsize}
    \end{tikzpicture}
    %\caption{Подкова.}
    \caption{The horseshoe}
    \label{horseshoe}
	\end{center}
\end{figure}

In this section we adapt several lemmas from~\cite{cherkashin2018horseshoe} to the case of rectangle.

Note that the finiteness of $\H (\Sigma)$ implies that $\Sigma$ is path-connected (see, for example,~\cite{EiSaHa}).
By the absence of loops the path in $\Sigma$ between every couple of points of $\Sigma$ is unique.

\begin{lemma}
Let $N$ be convex, $\Sigma$ be an arbitrary minimizer for $M$. Then $\Sigma$ is a subset of $N$.
\end{lemma}

\begin{proof}
Suppose the contrary, then one can project $\Sigma \setminus N$ on $N$, and the length strictly decreases.
\end{proof}

% \begin{lemma}
% Пусть $\Sigma$ --- произвольный минимайзер для $M$. Тогда 
% \begin{itemize}
%     \item [(i)] $\Sigma$ не содержат циклов;
%     \item [(ii)] $\Sigma \cap \Int(N_r)$ не содержит энергетических точек.
% \end{itemize}
% \label{prelim}
% \end{lemma}

% \begin{proof}
% Первая часть доказана в статье~\cite{paolini2004qualitative} (Theorem 5.5) для произвольных минимайзеров; вторая часть следует из basic property (a).
% \end{proof}

\begin{lemma}
\begin{itemize}
    \item [(i)] The closure of every connected component of $\Sigma \cap \Int(N_r)$ is a solution of the Steiner problem for
some set of points belonging to $M_r$, and in particular consists only of line segments of positive length.
    \item [(ii)] The length of any line segment in $\Sigma \cap N_r$ does not exceed $2r$.
    %Внутри $N_r$ нет отрезков $\Sigma$ длиной больше $2r$.
\end{itemize}
\label{ShortSegments}
\end{lemma}

\begin{proof}
Let $S$ be the closure of an arbitrary connected component of $\Sigma \cap \Int(N_r)$. Clearly all points of $S$ are non-energetic. Since $\Sigma$ consists of a finite number of curves, $S\subset \Sigma$ contains only finite number of branching points (see~\cite{teplitskaya2018regularity,teplitskaya2019regularity}) and thus the set $S \cap M_r$ is finite. No change in the set $\Int(\Sigma \cap N_r)$ influences the value of $F_M (\Sigma)$, so if we replace $S$ with a Steiner tree connecting $S \cap M_r$, then the length of the resulting set should remain the same by the optimality of $\Sigma$, and thus $S$ is itself a Steiner tree connecting $S \cap M_r$ as claimed.

Assume the contrary to the statement of item~(ii), i.e. there exists a segment $[AB]\in \Sigma \cap N_r$ with a length greater than $2r$. As $\Sigma$ doesn't contain loops, the set $\Sigma \, \setminus \, ]AB[$ consists of exactly two connected components, let us call them $\Sigma_1$ and $\Sigma_2$; these connected components are closed. 
By the definition $\gamma(\Sigma_1)$ and $\gamma(\Sigma_2)$ are closed and, as set $M$ is connected, the set $\gamma(\Sigma_1) \cap \gamma(\Sigma_2)$ is not empty. Let us consider an arbitrary point $A \in \gamma(\Sigma_1) \cap \gamma(\Sigma_2)$. 
As the sets $\Sigma_1$ and $\Sigma_2$ are closed, there exists such points $\sigma_1 \in \Sigma_1$ and $\sigma_2 \in \Sigma_2$ that $\dist(A,\sigma_1), \dist(A,\sigma_2) \leq r$.
Then $\Sigma_1 \cup \Sigma_2 \cup [A\sigma_1] \cup [A\sigma_2]$ has smaller length and not greater energy than $\Sigma$. So we get a contradiction.
\end{proof}
%Предположим противное пункту~(ii), то есть, что существует такой отрезок %$[AB]\in \Sigma \cap N_r$. Раз $\Sigma$ не содержит циклов, то $\Sigma \, %\setminus \, ]AB[$ состоит из ровно двух компонент связности, назовем их %$\Sigma_1$ и $\Sigma_2$; эти компоненты замкнуты.
%По определению $\gamma(\Sigma_1)$ и $\gamma(\Sigma_2)$ замкнуты и множество %$\gamma(\Sigma_1) \cap \gamma(\Sigma_2)$ непусто (так как $M$ связно); %рассмотрим произвольную точку $A \in \gamma(\Sigma_1) \cap \gamma(\Sigma_2)$. 
%Раз $\Sigma_1$ и $\Sigma_2$ замкнуты, существуют точки $\sigma_1 \in \Sigma_1$ %и $\sigma_2 \in \Sigma_2$, такие что $|A\sigma_1|, |A\sigma_2| \leq r$.
%Тогда замена $\Sigma$ на $\Sigma_1 \cup \Sigma_2 \cup [A\sigma_1] \cup %[A\sigma_2]$ не увеличивает энергию и уменьшает длину. Противоречие.

It appears that $\Sigma$ is located near the boundary of $M$.

\begin{figure}[h]
    \centering
    \begin{tikzpicture}[scale=1.2]
        \def\rr{1.5}
        \def\hh{(-0.832*\rr, -0.5547*\rr)}
        
        \coordinate(B) at (0.535*\rr, \rr);
        \coordinate(B1) at (-.191*\rr, 2.0897*\rr);
        \coordinate(B2) at (1.845*\rr, \rr);
        \coordinate(B15) at (.827*\rr, 1.545*\rr);
        \coordinate(B3) at (3.845*\rr, \rr);
        \coordinate(C2) at (7, \rr);
        \coordinate(C1) at (-3+1.2*\rr, 4.5);
        
        \fill[gray!10!white] (C1) -- (B1) -- (B2) -- (C2) -- (7, 4.5) -- cycle;
        \draw (5.5, 3.5) node{\large $K_r$};

        \draw[very thick] (7,0) -- 
        (0, 0) node[below left]{$A_i$} -- 
        (-3, 4.5);
        \draw (C1) -- (B1) -- 
        (B2) node[above right, pos=0.7]{$2r/\sqrt{3}$} -- 
        (B3) node[above, pos=0.5]{$2r$} --
        (C2);
        
        \draw [shift={(B)}, dashed] (0,0) -- (0, -\rr) node[right,pos=0.45]{$r$};
        \draw [shift={(B2)}, dashed] (0,0) -- (0, -\rr);
        \draw [shift={(B3)}, dashed] (0,0) -- (0, -\rr);
        \draw [dashed] (B1) -- (B) -- (B2);
        \draw [dashed] (0, 0) -- (B15);
        
        \def\pp{0.2}
        \draw [shift={(B)}] (0, -\rr+\pp) -- (\pp, \pp-\rr) -- (\pp, -\rr);
        \draw [shift={(B2)}] (0, -\rr+\pp) -- (\pp, \pp-\rr) -- (\pp, -\rr);
        \draw [shift={(B3)}] (0, -\rr+\pp) -- (\pp, \pp-\rr) -- (\pp, -\rr);
        \draw [shift={(B15)}, rotate = 240] (\pp, 0) -- (\pp, \pp) -- (0, \pp);
        
        \def\qq{0.3}
        \draw (\qq, 0) arc (0 : 62 : \qq);
        \draw[rotate=62] (\qq+0.07, 0) arc (0 : 62 : \qq+0.03);  
        \draw[shift={(B)}] (\qq, 0) arc (0 : 62 : \qq);
        \draw[shift={(B)},rotate=62] (\qq+0.07, 0) arc (0 : 62 : \qq+0.03); 
        %node[right, pos=0.5]{$\angle A_i/2$};
        \fill (B15) circle (1.4pt);
        \fill (B3) circle (1.4pt);
        
        \draw [decorate , decoration={brace, mirror, amplitude=3pt},      
            xshift=0pt,yshift=-2pt]
        (0, 0) -- (1.845*\rr, 0) node [midway,yshift=-0.3cm] {$p_M(i)$};
        
    \end{tikzpicture}
    %\caption{Вид области без точек ветвления вблизи угла}
    \caption{The region without branching points near the angle}
    \label{fig:scary_formula}
\end{figure}
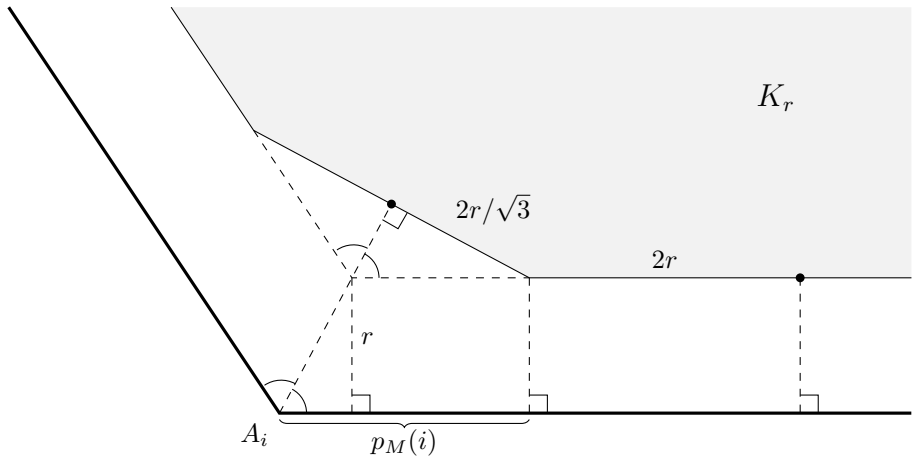

\begin{lemma}
Let $M$ be an $n$-gon with vertices $A_i$ and angles $\angle A_i  \geq \pi/3$, $i=1\dots n$, and let $r>0$ be sufficiently small. 
Let us define the following sets (see Fig.~\ref{fig:scary_formula}) 
\[
K_r \defeq \conv \left( N_r \setminus \bigcup_{i=1}^n B_{r\sqrt{1+p_M(i)^2}}(A_i) \right ),
\]
where 
\[
p_M(i) \defeq \frac{2}{\sqrt 3 \sin \frac{\angle A_i}{2}}+\cot \frac{\angle A_i}{2}.
\]
%где \[h(r)=r\sqrt{1+(\frac{2}{\sqrt 3 \cos \frac{\angle A_i}{2}}+\tg \frac{\angle A_i}{2})^2}\].

%$h(r)=r\sqrt{\frac 4 3 + (\sin \frac{ \angle A_i}{2} + \frac{2}{\sqrt 3} \cot \frac{\angle A_i}{2})^2}$. 
In other words, the boundary $\partial K_r$ is $2n$-gon, such that its sides parallel to the sides of 
$M$ at the distance $r$, and sides of length $ \frac{4r}{\sqrt 3} $, each of which is the base of an isosceles triangle with the apex at the nearest vertex of the polygon $A_i$, alternate.

Then $\Sigma$ has no branching points in the area $\Int K_r$.
%\textcolor{red}{ надо изменить определение $K_r$ чтобы согласовать с $K_l$. У меня немного ужасно получается: вместо $4r/\sqrt 3$ получается $r\sqrt{\frac 4 3 + (\sin (\frac{A_i}{2})+\cot \frac{A_i}{2}\frac{2}{\sqrt 3})^2}$
%Можно еще добавить:
%Иными словами $\partial K_r$ --- это $2n$--- угольник, в котором чередуются стороны параллельные сторонам $M$ и отстоящие от них на расстояние $r$ и стороны длины $\frac{4r}{3}$, каждая из которых является основанием равнобедренного треугольника с вершиной в ближайшей вершине многоугольника $A_i$. }
%\textcolor{red}{я пока убрала условие $\angle A_i  \geq \pi/3$}

%\textcolor{red}{а надо ли согласовывать? что плохого в том, чтобы это были разные контуры?} 
\label{KarKar}
\end{lemma}

\begin{proof}
Since the length of a segment of $\Sigma$ in $\Int N_r$ is not greater than $2r$, it is easy to see that for an arbitrary point $a\in \Sigma$ the statement $\overline {B_{4r}(a)} \nsubseteq \Int N_r$ holds. Otherwise let us consider such a point $a$. Recall that each point of $\Sigma$ in $\Int N_r$ is a center of a segment or of a regular tripod, and all segments are not longer than $2r$. Then there exists such a regular hexagon with sides of length $ 2r $, angles of $ 2 \pi / 3 $ and a vertex in $ a $, that it has to contain an endpoint of $\Sigma $. However this is impossible since this hexagon is contained inside the ball $ B_{4r}(a) $ which lies in $ \Int N_r $ and therefore does not contain energetic points of $ \Sigma $. 

We define set $K_l \subset N$ as follows (see Fig.~\ref{lemma25KL}): 
\begin{itemize}
\item let $\partial K_l$ be $2n$-gon, in which the sides parallel to the sides $ M $ and the sides with length  $\frac{4r}{\sqrt 3}$ having equal side angles alternate;
\item the distance from a side $M$ to the side of $\partial K_l$, parallel to it, is equal to $l$.

\end{itemize}

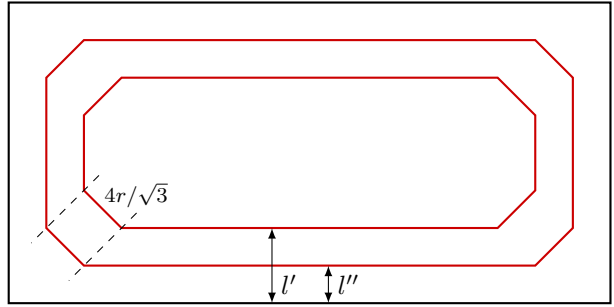
\begin{figure}[h]
\tikzset{>={latex[length=6pt, width=4pt]}}
    \hfill
    \begin{tikzpicture}[]
    
    \draw[thick] (-4, -2) -- (4, -2) -- (4, 2) -- (-4, 2) -- cycle;
    
    \def\aa{3}
    \def\bb{1}
    \def\rr{.5}
    \draw[thick, red!80!black] 
    (-\aa + \rr, -\bb) -- (-\aa, -\bb + \rr)  -- 
    (-\aa, \bb - \rr) -- (-\aa + \rr, \bb) --
    (\aa - \rr, \bb) -- (\aa, \bb - \rr) --
    (\aa, -\bb + \rr) -- (\aa - \rr, -\bb) -- cycle;
    
    \edef\bbbb{\bb}
    \edef\aaaa{\aa}
    \def\aa{3.5}
    \def\bb{1.5}
    \def\rr{.5}
    \draw[thick, red!80!black] 
    (-\aa + \rr, -\bb) -- (-\aa, -\bb + \rr)  -- 
    (-\aa, \bb - \rr) -- (-\aa + \rr, \bb) --
    (\aa - \rr, \bb) -- (\aa, \bb - \rr) --
    (\aa, -\bb + \rr) -- (\aa - \rr, -\bb) -- cycle;
    
    \draw[<->] (-.5, -2) -- (-.5, -\bbbb)  node[right, pos=0.25]{$l'$};
    \node[black,above right]  at(-\aaaa-.1 + \rr/2, -\bbbb-.1 + \rr/2) {\footnotesize$4r/\sqrt{3}$};
    
    \draw[dashed, thin] (-\aaaa + 0.2 + \rr, -\bbbb + 0.2) -- (-\aa + \rr - 0.2, -\bb - 0.2);
    \draw[dashed, thin] (-\aaaa + 0.2, -\bbbb +0.2 + \rr) -- (-\aa - 0.2, -\bb - 0.2 + \rr);
    
    \draw[<->] (.25, -2) -- (.25, -\bb)  node[right, pos=0.5]{$l''$};

    \end{tikzpicture}
    \hfill
    $\mathstrut$
    \caption{Sets $K_{l'}$ and $K_{l''}$ for the rectangle with $l'>l''$}  
    \label{lemma25KL}
\end{figure}

Let $l_0$ be the maximal number such that each side of $\partial K_{l_0}$ is not shorter than $\frac{4r}{\sqrt 3}$, and $r$ is small enough to $B_{4r}(K_{l_0}) \subset \Int N_r$ holds. Assume $l>0$ doesn't exceed $l_0$. Note that then all sides of $ \partial K_l $ are not degenerate. It is easy to see that in view of the condition on the angles of the polygon $ M $, all the angles of the polygon $ \partial K_l $ are at least $ 2 \pi / 3 $.

Thus for any $ l $, such that the intersection  $\Sigma \cap \partial K_l$ is nonempty, all sides of $ K_l $ have length at least $ \frac{4r}{\sqrt 3} $.

Suppose the opposite to the assertion of the lemma: let $ \Int K_r $ contain a branching point. Then consider a branching point on the innermost layer: let $ a \in \partial K_{l(a)} $ be such a branching point that $ l (a)> r $ is maximal.
Let $ a $ lie on the interior of a side of $ \partial K_{l(a)} $. Then one or two segments of $ \Sigma $ of length at most $2r $ go from $ a $ into $ \Int K_{l(a)} $. We are going to show that the end of at least one of them is contained in $ \Int K_{l(a )} $. Since the end of the segment can only be at the branching point, this would contradict the choice of $ l(a) $. Indeed, if only one segment goes into $ \Int K_{l(a)} $, the statement is obvious. Let us examine in detail the case where two segments go out in $ \Int K_
{l (a)} $.

\begin{figure}[h]
\tikzset{>={latex[length=6pt, width=4pt]}}
    \hfill
    \begin{tikzpicture}[]
    \footnotesize
    \draw [very thick] (5.27,3.38) -- (2,0)-- (-2,0) -- (-4.83,3.82);
    \draw [thick, blue] (3.45,1.5)-- (0.61,0) node[above left, black,pos=.5]{$\leq 2r$};
    \draw [thick, blue] (0.61,0)-- (-4.276,3.072) node[above right, black,pos=.5]{$\leq 2r$};
    \fill [black] (2,0) circle (2pt);
    \draw[] (2,0.) node[below] {$B_2$};
    \fill [black] (-2,0) circle (2pt);
    \draw[] (-2,0) node[below] {$B_1$};
    \draw[] (-2,0.2) node[right] {$\geq 2\pi/3$};
    \fill [blue] (0.61,0) circle (2pt);
    \draw[] (0.61,0.) node[below] {$a$};
    \draw[] (0.65,0.49) node {$2\pi/3$};
    \draw[] (0.,0.2) node {$\alpha$};
    \fill [blue] (-4.276,3.072) circle (2pt) node[below left,black]{$C_1$};
    \fill [blue] (3.45,1.5) circle (2pt) node[below right,black]{$C_2$};
    
    \end{tikzpicture}
    \hfill
    $\mathstrut$
    \caption{The case when two segments go out in $ \Int K_{l (a)} $}
    %\caption{Случай, когда в $\Int K_{l(a)}$ выходят два отрезка}  
    \label{lemma2.5.th.sin}
\end{figure}
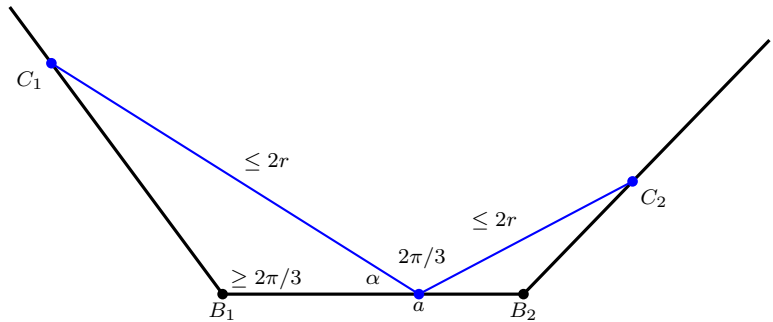

Let none of these segments end in $ \Int K_{l (a)} $. Then both segments intersect the sides of $ K_{l (a)} $; it is clear that these sides are adjacent to the one, containing point $ a $. Let us denote the ends of the side on which $ a $ lies as $ B_1 $ and $ B_2 $, and the points of intersection of the tripod segments with the sides outgoing from $ B_1 $ and $ B_2 $, as $ C_1 $ and $ C_2 $, respectively (see Fig.~\ref{lemma2.5.th.sin}). Let $ \angle B_1aC_1 = \alpha $. Then
\begin{gather*}
    \angle aC_1B_1 \leq \pi /3 - \alpha,\qquad
    \angle B_2aC_2 = \pi /3 - \alpha,\qquad
    \angle aC_2B_2 \leq \alpha.
\end{gather*}
Moreover, $|aC_1|, |aC_2| \leq 2r$. Hence, by the law of sines 
\begin{equation*}
    |aB_1| + |aB_2| =
    \frac{\sin \angle aC_1B_1 \cdot |aC_1|}{\sin aB_1C_1}
    + 
    \frac{\sin \angle aC_2B_2 \cdot |aC_2|}{\sin aB_2C_2}
    \leq
    \frac{\sin(\pi /3 - \alpha) \cdot 2r}{\sin 2\pi/3}
    +
    \frac{\sin\alpha \cdot 2r}{\sin 2\pi/3} 
    =
    \frac{4r}{\sqrt 3} (\sin (\pi/ 3 - \alpha) + \sin \alpha).
\end{equation*}
%Далее, синус на промежутке от $0$ до $\pi$ выпуклый вверх, поэтому 
Furthermore, the sine in the interval from $ 0 $ to $ \pi $ is concave, therefore
\begin{equation}
    \frac {\sin (\pi/3 - \alpha)}2 + \frac{\sin \alpha} 2 <
    \sin \left(\frac{\pi/3 - \alpha + \alpha}2 \right) = \sin \frac\pi6 = \frac 12.
\end{equation}
Thus the length of the side of $K_l$, containing a point $a$, is equal to  $|aB_1| + |aB_2| < 4r/\sqrt 3$. This contradicts the condition on the length of the minimal side of $K_l$.

Now, let point $a$ coincide with a vertex of $K_{l(a)}$, then there exists a segment of $\Sigma$ with one end at $a$ and another end in $\Int K_{l(a)}$ or at the interior of a side $\partial K_{l(a)}$. The first case contradicts the definition of $l(a)$, and the second has already been considered. Thus  the initial assumption about the presence of branching points in $\Int K_r$ was wrong.
\end{proof}

Recall that $p_M(i) = \frac{2}{\sqrt 3 \sin \frac{\angle A_i}{2}}+\cot \frac{\angle A_i}{2}$.
Let \[
P_r \defeq \Int\left(\conv\left(N_r \setminus \bigcup_{i=1}^n B_{\sqrt{1+(p_M(i)+2)^2}r}(A_i)\right)\right).\]
Clearly, $P_r \subset K_r\subset N_r$.

\begin{corollary}
The statement $\Sigma \cap P_r =\emptyset$ holds.
\label{CorKar}
\end{corollary}
\begin{proof}
The set $\Sigma$ has no terminal points or points of order $2$ in the set $\Int N_r \supset \Int K_r$, all segments of $\Sigma \cap \Int N_r$ are not longer than $2r$, and $\Sigma$ has no branching points in the area $\Int K_r$. Then the ends of any segment $s$ intersecting area $\Int K_r$ lie on $\partial K_r$ not farther than $2r$ from vertices of $K_r$, and hence are at a distance not greater than $(p_M(i)+2)r$ from vertices of $M_r$. Then all points of $s$ are located at a distance at most $\sqrt{1+(p_M(i)+2)^2}r$ from the vertices of $ M $.
\end{proof}

\section{Derivation in the picture}
\label{diff}
The following lemma is essentially proved in~\cite{cherkashin2018horseshoe}; we prove it for sake of completeness.

\begin{lemma}\label{lm:energetic_points_angles}
\begin{itemize}
    \item [(i)] Let $Q$ be an energetic point of degree 1 (i.e. $Q$ is the end of the segment $[QX] \subset \Sigma$) with unique $y(Q)$. Then $Q$, $X$ and $y(Q)$ lie on the same line.
    
    \item [(ii)] Let $W$ be an energetic point of degree 2 (i.e. $W$ is the end of the segments $[WZ_1]$ and $[WZ_2] \subset \Sigma$) with unique $y(W)$. Then 
    $\angle Z_1Wy(W) = \angle y(W)WZ_2$.
\end{itemize}
\end{lemma}

\begin{proof}
Suppose the contrary to~(i), then $Y := \partial B_r(y(Q)) \cap [y(Q)X]$ differs from $Q$. Replace $[QX]$ with $[YX]$ in $\Sigma$. Clearly $M$ is still covered; by the triangle inequality $r  + |YX| = |y(Q)X| < |y(Q)Q| + |QX| = r + |QX|$, so the length of $\Sigma$ decreases with the replacement. Contradiction.

\begin{figure}[h]
    \centering
    \begin{tikzpicture}[scale=2.5]
            
            \def\qwe{2.923717048}
            \def\hh{2.2}
            \coordinate (Z1) at (-1, \hh);
            \coordinate (Z2) at (2.5, \hh);
            \coordinate (ZZ1) at (-.8276,1.82);
            \coordinate (ZZ2) at (1.5014,1.321);
            \coordinate (Y) at (0.522, 0);
        
            \def\area{(-\qwe+1, -.4369524410) rectangle (\qwe, \hh)}
            
            %\fill[yellow] \area;
            \clip \area;
            
            \draw (\qwe, -.4369524410) arc (73 : 107 : 10);
            \draw[dashed] (0,0) circle (2); 
            \draw[thick, blue] (Z1) -- (0, 0) -- (Z2);
            \draw (-\qwe, 0) -- (\qwe, 0) node[pos=0.95, above]{$l$};
            
            \def\rr{.03}
            \draw[thick, blue] (ZZ1) -- (Y) -- (ZZ2);
            \fill (Y) circle (\rr);
            \draw[shift={(Y)}] (0, -.05) node[below]{$Y$};
            \def\aaa{0.15}
            \draw[shift={(Y)}] (\aaa,0) arc (0 : 53 : \aaa);
            \draw[shift={(Y)}] (-\aaa,0) arc (180 : 127 : \aaa);
            
            \draw (Z1) node[below left]{$Z_1$};
            \draw (Z2) node[below right]{$Z_2$};
            \fill (ZZ1) circle(\rr);
            \draw[shift={(ZZ1)}] (-.05, -.1) node[below]{$Z'_1$};
            \fill (ZZ2) circle(\rr);
            \draw[shift={(ZZ2)}] (0, .1) node[above]{$Z'_2$};
            \fill (0,0) circle (\rr);
            \draw (0, -.05)  node[below]{$W$};
            
            \draw(-1.6, 1.2) node[below right]{$B_\varepsilon(W)$};
    \end{tikzpicture}
    \caption{ %Ситуация вблизи точки $W$ в случае 2 леммы~\ref{lm:energetic_points_angles}.}
   Picture near the point $W$ in Lemma~\ref{lm:energetic_points_angles}, case (ii).}
    \label{energeticPICT}
\end{figure}
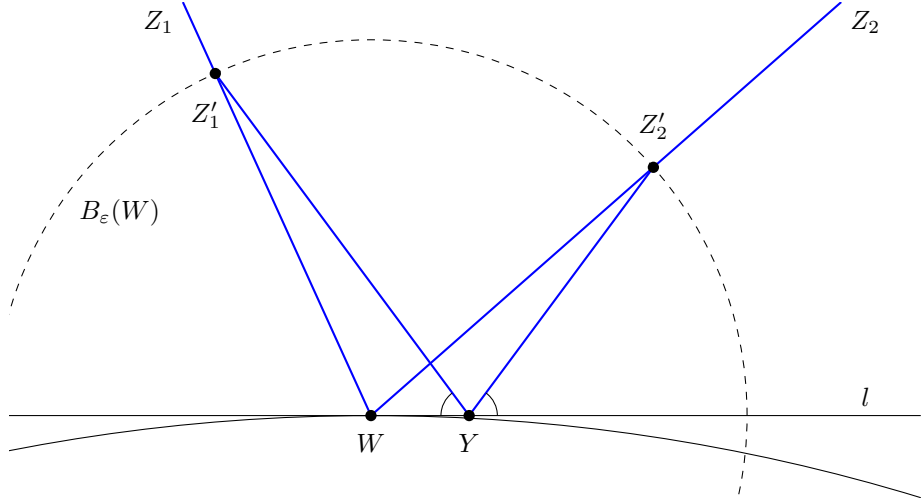

Suppose the contrary to~(ii) i.e. that $\angle Z_1Wy(W) \neq \angle y(W)WZ_2$. Let $l$ be the tangent line to $B_r(y(W))$
at the point $W$. Then 
\[
\H (\Sigma \cap B_{\varepsilon}(W)) - \H ([Z'_1Y ] \cup [Y Z'_2]) = O(\varepsilon),
\] 
where $Z'_1$ and $Z'_2$ are intersections of segments $Z_1Y$ and $Z_2Y$ with $\partial B_{\varepsilon}(W)$, and $Y$ is such a point in $l$ that $\angle Z'_1Yy(W) = \angle y(W)YZ'_2$ (see Fig.~\ref{energeticPICT}). On the other hand, $\dist (Y, \partial B_r(y(W))) = O(\varepsilon^2)$; let $V \in \partial B_r(y(W_1))$ be such a point that $\dist (Y, \partial B_r(y(W))) = \dist (Y, V)$. Then the set $(\Sigma \setminus B_\varepsilon (W)) \cup [Z_1Y ] \cup [Z_2Y ] \cup [YV]$ is connected, covers $M$ and has strictly lower length than $\Sigma$, giving a desired contradiction.
\end{proof}

In this section we use results from~\cite{cherkashin2020minimizers}. The facts stated below hold for a convex smooth curve $M$ with radius of curvature more than $r$. 
They are also applicable for a partially smooth $M$ when $y_1$ is a smooth point.

Consider a point $y_1 \in M$ such that $B_r(y_1) \cap \Sigma = \emptyset$. Suppose there exists an energetic point $x \in \partial B_r(y_1) \setminus M_r$. Denote the order of $x$ as $d\in \{1,2\}$, and the number of corresponding points $y(x)$ as $k \in \{1,2\}$ (denote the second such point as $y_2$, if it exists).

Fix a small enough $l$ so that $\overline{B_l(x)} \cap \Sigma$ is a union of $d$ segments of the form $[z_ix]$, $z_i \in \partial B_l(x)$, $1 \leq i \leq d$. Let $\varepsilon > 0$ be sufficiently small and denote by $y_1^\varepsilon$ the point obtained by moving $y_1$ along $M$ a distance $\varepsilon$ in such direction that $x \notin B_r(y_1^\varepsilon)$. Similarly, for a sufficiently small $\varepsilon < 0$ denote by $y_1^\varepsilon$ the point obtained by moving $y_1$ along $M$ a distance $-\varepsilon$ in the other direction. In the case $k=2$ denote $y_2^\varepsilon = y_2$.

Let
\[
\Gamma(\varepsilon) = \min_{x'} \sum_{i=1}^d |z_ix'|,
\]
there the minimum is over all points $x'$ which satisfy $|y_j^\varepsilon x'|=r$ for $1 \leq j \leq k$. Denote by $x_\varepsilon$ the point achieving the value of $\Gamma(\varepsilon)$.

Note that $x_0 = x$ because $\Sigma$ is a minimizer. By the derivative of length of $\Sigma$ in a neighborhood of $x$ when moving $y_1(x)$ along $M$ we mean the derivative of $\Gamma(\varepsilon)$ in zero $\Gamma'(0)$.

\begin{proposition}
Let $M$ be a convex smooth curve with radius of curvature more than $r$; $\Sigma$ be an arbitrary minimizer for $M$, $r$.
Let $x \in \Sigma$ be an energetic point, $y(x) \in M$ be an arbitrary corresponding point. 
Then the derivative of length of $\Sigma$ in a neighborhood of $x$ when moving $y$ along $M$ is nonnegative.
\label{nonnegative}
\end{proposition}

%\textcolor{red}{Наверное, надо писать в сторону увеличения $\gamma$ от окрестности $x$.}

\begin{proposition}
Let $y \in M$ be a point such that $B_r(y) \cap \Sigma = \emptyset$ and $\partial{B_r(y)}$ contains energetic points $x_1$ and $x_2$.
Define $Y = \partial B_r(y) \cap M_r$. Then
\begin{itemize}
    \item [(i)] points $x_1$ and $x_2$ lie on opposite sides of the line $(yY)$;
    \item [(ii)] derivatives of length of $\Sigma$ in neighborhoods of $x_1$ and $x_2$ when moving $y$ along $M$ are equal.
\end{itemize}
\label{diffproposition}
\end{proposition}

Also in~\cite{cherkashin2020minimizers} the derivative of length of $\Sigma$ in a neighborhood of $x$ when moving $y$ along $M$ is calculated.
The derivative depends on the behavior of $\Sigma$ in the neighborhood of $x$. In the present paper we need the following cases.

\paragraph{1. $x$ has order 1, and there is unique corresponding $y(x)$.} Then the derivative is equal to
\[
\cos\alpha,
\]
where $\alpha = \angle ([xy(x)) , l)$, $l$ is a tangent ray to $M$ at point $y(x)$, in the direction of increasing $\gamma (x)$.

\paragraph{2. $x$ has order 2, and the unique corresponding $y(x)$.} Since $x$ has order 2, $B_\varepsilon (x) \cap \Sigma = [xz_1]\cup [xz_2]$ for small enough $\varepsilon > 0$. Then the derivative is equal to
\[
2\cos\alpha\cos \frac{\angle z_1xz_2}{2},
\]
where $\alpha = \angle ([xy(x)) , l)$, $l$ is a tangent ray to $M$ at point $y(x)$, in the direction of increasing $\gamma (x)$.

\section{The behavior of \texorpdfstring{$\Sigma$}{Sigma} near the angles of \texorpdfstring{$M$}{M}}
\label{sect:nearangles}

The major part of lemmas in this section require $M$ to be a polygon $A_1\dots A_n$ with angles at least $\pi/3$. 
Some lemmas use stronger restrictions on $M$.

\subsection{Enclosing of a minimizer} %Рассмотрим точку $y \in M$, такую что $B_r(y)$ касается $K_r$ (многоугольник, определенный в доказательстве леммы~\ref{KarKar}). Пусть $B_r(y)$ не пересекает $\Sigma$ и точка $\overline{B_r(y)} \cap M_r$ не лежит в $\Sigma$.
%Тогда левая и правая окрестности точки $y$ в $M$ покрыты разными точками; назовем их $x_l$ и $x_r$.
%Поскольку $\Sigma$ связно и не содержит циклов, в $\Sigma$ существует единственный путь  между $x_l$ и $x_r$.
%По следствию~\ref{CorKar} этот путь обходит $K_r$.
%Существование еще одной точки со свойствами $y$ лишает возможности соединить $x_l$ с $x_r$ в $\Sigma$.
Let $C_i(t^-_i,t^+_i)$ be a quadrangle, which is bounded by perpendiculars to the sides of $M$, which are incident to $A_i$ and at a distance of $(p_M(i)+2+t^\pm_i)r$ from $A_i$ (where $t^\pm_i\in[0;2]$ will be chosen during the proof for every $A_i$ and for each sides incident to it). Let us define heptagon $\alpha_i(t^-_i,t^+_i)$ as follows: 
\[
\alpha_i(t^-_i,t^+_i) \defeq \overline{C_i(t^-_i,t^+_i) \setminus P_r},\quad i=1 \ldots n.
\]
Note that in light of Corollary~\ref{CorKar} it is true that $\Sigma \cap \partial \alpha_i(t^-_i,t^+_i) \subset \partial C_i(t^-_i,t^+_i) \setminus P_r$. 
We show that for some choice of parameters $t^\pm_i$ this intersection consists of exactly two points.
%(ввиду пустоты $K_r$ речь идет о точках на пересечении перпендикуляров и кольца $N \setminus N_r$).% Предположим противное. В таком случае найдется сторона $\partial C_i \setminus (M\cup \partial P_r)$, содержащая две точки $\Sigma$.

\begin{lemma}
There exists such $(t_1^\pm,t^\pm_2, \ldots, t^\pm_n)$ that $\Sigma$ intersects a boundary of each heptagon $\alpha_i(t^-_i,t^+_i)$ exactly at two points.
\label{twopoints}
\end{lemma}
\begin{proof}
Clearly, in view of connectivity of $\Sigma$ and Corollary~\ref{CorKar} for arbitrary $t^-,t^+$ there is at most one point $x\in M\setminus (\bigcup_{i=1}^{n} \alpha_i(t^-,t^+))$ such that $B_r(x)\cap \Sigma = \emptyset$. Then energetic terminal points of the set $\Sigma \setminus \overline{B_r(\bigcup_{i=1}^{n} \alpha_i(t^-,t^+))}$ belong to the circumference $\partial B_r(x)$ if such point $x$ exists, and don't exist otherwise.

Then each connected component of $\Sigma \setminus \bigcup_{i=1}^{n} \alpha_i(t^-_i,t^+_i)$ is a minimal network, contained in $N \setminus N_r$ and connecting points from sets $\partial B_r(x) \cap N$ and $\overline{\partial \alpha_i(t^-_i,t^+_i) \setminus (N_r \cup M)}$ for some $i$. Moreover each connected component connects points from exactly two sets, and there are no two components connecting the same pair of sets. 

It is easy to see that if there is $ i $ such that $ \sharp (\partial C_i (t^-_i,t) \setminus M \cap \Sigma) \geq 2 $ for any $ t \in [0; 2] $, then in view of the co-area inequality, the length of $ \Sigma $ on the part of the rectangle $ N \setminus N_r $ between two perpendiculars drawn at the distance $ (p_M (i) +2) r $ and $ (p_M(i) +4) r $ from $ A_i $ is $ 4r $. Then the replacement of the indicated subset of $\Sigma$ with three sides of this rectangle (two perpendiculars to $ M $ and a side that is contained in $ M_r $) decreases the length and does not increase the energy, which contradicts the optimality of $ \Sigma $. Analogously we can fix $t^+_i$ and variate $t^-_i$. Thus, for each $i$ there are such $ t^\pm_i \in [0; 2] $ that each of perpendiculars drawn at the distance $ (p_M (i) + 2 + t^\pm_i) r $ from vertex $ A_i $ intersects $ \Sigma $ at exactly one point. In what follows, such $ C_i (t^-_i,t^+_i) $ will be called $ C_i $, and $ \alpha_i(t^-_i,t^+_j) $, accordingly, will be called $ \alpha_i$.
\end{proof}

Thus $\Sigma \cap \alpha_i$ consists of one or two connected components. As $\Sigma$ is connected, there are at least $n-1$ heptagons $\alpha_i$ such that $\Sigma \cap \alpha_i$ is connected. We want to show that all $n$ heptagons are like this. 

Note that inside the connected components of $N \setminus \bigcup_{i=1}^{n} \alpha_i$, set $ \Sigma $ has two variants of behavior: either it is a segment connecting sides of $\alpha_i$, or (there is at most one such component) it is a union of two segments, the distance between which is $(2-o_r (1)) r$.

\begin{proposition}
Let segment $[xa]\subset \Sigma$ be such that $|za|<r$, where $z \defeq [xa)\cap M$. And let $\gamma([ax]) \cap H \subset \overline {B_r( \Sigma \setminus [ax[)}$, where $H$ is one of two half-planes bounded by $(za)$. Then point $a$ can not be terminal.
\label{hs}
\end{proposition}

\begin{figure}
    \centering
    \begin{tikzpicture}[scale=1.4]
        \coordinate (C1) at (-2, 0);
        \coordinate (C2) at (2, 0);
        \coordinate (Z) at (0, 0);
        \coordinate (A) at (0.39363, 0.7888);
        \coordinate (X) at (1, 2);
        \coordinate (Z1) at (-0.75169, 0.00149);
        \coordinate (A1) at (0.16441, 1.04667);
        \coordinate (I) at (2.27019, 0.98315);
        
        \draw[very thick] (C1) -- (C2);
        \draw[thick, blue] (X) -- (A);
        \draw[blue, dashed] (A) -- (Z);
        
        \draw[shift={(A)}] (1.63383, -0.943295) arc (-30 : 210 : 1.88659);
        \draw[shift={(Z1)}] (1.38983, 0) arc (0: 140 : 1.38983);
        \draw[blue, dashed] (A1) -- (Z1);
        \draw[blue, thick, dashed, dash pattern=on 3pt off 2pt] (A1) -- (X);
        \draw[dashed] (Z1) -- (A);
        
        \draw[dotted] (A) -- (I) node[pos=0.5, below]{$r$};
        
        \def\nodesize{0.04}
        
        \fill[blue] (A) circle (\nodesize); 
        \draw[shift={(A)}] (0, -0.1) node[below]{$a$};
        \fill[blue] (A1) circle (\nodesize); 
        \draw[shift={(A1)}] (-0.1, 0) node[left]{$a'$};
        \fill (Z) circle (\nodesize) node[below]{$z$};
        \fill (Z1) circle (\nodesize) node[below]{$z'$};
        \draw (X) node[right]{$x$};
        
    \end{tikzpicture} 
    \caption{To Proposition~\ref{hs}}
    \label{fig:proposition42}
\end{figure}
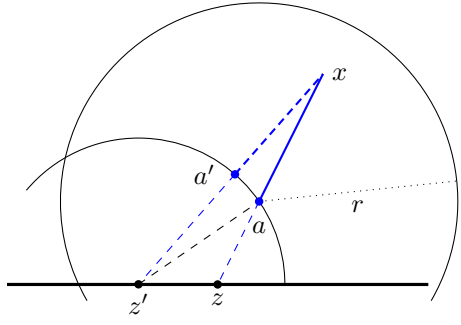

\begin{proof}
 Let $z'\in M \setminus H$ be such a point that $[zz'] \subset B_r(a)\cap M$. And let $a' \defeq \partial B_{|z'a|}(z') \cap [xz']$.
 Then, in view of the triangle inequality for $ \triangle z'ax $ inequality $ | xa '| <| xa | $ holds. In addition, since the intersection $ B_r (a) \cap M \setminus H $ outside $ [zz '] $ is on the such side of the perpendicular to $ [aa'] $ where the points are closer to $ a '$ than to $ a $, $ B_r (a) \cap M \setminus H \subset B_r (a ') \cap M \setminus H $ executes. Thus set $ \Sigma \setminus [ax] \cup [a'x] $ is shorter and has not greater energy than $ \Sigma $. If point $ a $ has degree $ 1 $, then set $ \Sigma \setminus [ax] \cup [a'x] $ is connected, which contradicts the optimality of $ \Sigma $.
 \end{proof}

\begin{lemma}
For every heptagon $\alpha_i$ set $\Sigma \cap \alpha_i$ is connected.
\label{hujugl}
\end{lemma}
\begin{proof}
Assume the contrary, i.e. there exists such $i$, that \[\Sigma \cap \alpha_i=: \Sigma_1 \sqcup \Sigma_2.\]
Then $\dist(\Sigma_1,\Sigma_2)= (2-o_r(1))r$, otherwise the addition to $\Sigma$ of the shortest segment connecting $\Sigma_1$ with $\Sigma_2$ and replacement of an arbitrary connected component of $\Sigma \setminus \bigcup_{i=1}^{n} \alpha_i$ (each of them is a segment) with the union of two segments with the distance equal to $(2-o_r(1))r$ between them, decreases the length, maintains the connectivity and doesn't increase the energy, which contradicts the optimality of $\Sigma$. 
At least one of connected components (without loss of generality, $\Sigma_1$) covers the points belonging to only one side (we will call this side \emph{the first}).
Clearly $\Sigma_1$ is a segment: let $z_1\defeq \partial \alpha_i \cap \Sigma_1$ and let $z_2$ be such a point of $\Sigma_1$, that covers the closest to $A_i$ point of the first side (we will refer to this point as $y$). Then set  $\Sigma \setminus \Sigma_1 \cup [z_1z_2]$ is connected, has not greater energy and smaller length than $\Sigma$, so $\Sigma_1$ coincides with segment $[z_1z_2]$.
As point $z_2$ is energetic and $\Sigma_1$ covers only points of the first side, $y$ is the single point corresponding to $z_2$ ($y=y(z_2)$), and in view of Lemma~\ref{lm:energetic_points_angles}(i) points $z_1$, $z_2$ and $y$ belong to the same line.
Clearly, $\dist(y,\Sigma_1)=\dist(y,\Sigma_2)=r$.
As $\dist(\Sigma_1,\Sigma_2) = (2-o_r(1))r$, all points of $(\Sigma_1 \cup \Sigma_2) \cap \partial B_r(y)$ are located at the distance $o_r(1)r$ from the first side (in particular; it means that point $y$ cannot coincide with $A_i$). 
Let $x$ be an arbitrary point of $\Sigma_2 \cap \partial B_r(y)$. Then $x$ lies at the distance $o_r(1)r$ from the first side. The order of point $x$ is either $1$ or $2$.

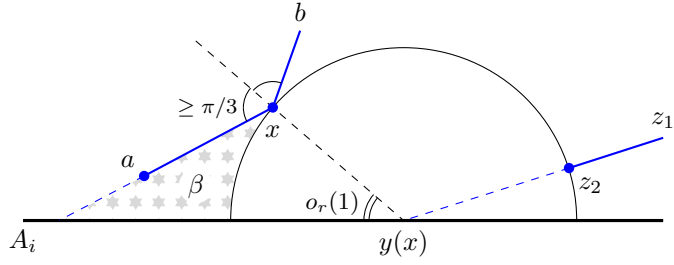
\begin{figure}
    \centering
    \begin{tikzpicture}[scale=2.3]
        \coordinate (C1) at (-2.2, 0);
        \coordinate (C2) at (1.5, 0);
        \coordinate (Y) at (0, 0);
        \coordinate (A) at (-1.5, 0.25768);
        \coordinate (B) at (-0.59841, 1.09832);
        \coordinate (X) at (-0.75645, 0.65406);
        \coordinate (Z1) at (1.5, 0.4787);
        \coordinate (Z2) at (0.95574, 0.30365);
        \coordinate (L) at (-1.20164, 1.03899);
        \coordinate (Z) at (-1.9925, 0);
        \coordinate (beta) at (-1.2, 0.2);
        
        \def\oblastj{(X) arc (139.152 : 180 : 1) -- (Z) -- cycle}
        \draw[white, pattern color=gray!30!white, pattern=sixpointed stars] \oblastj;
        
        \draw[very thick] (C1) -- (C2);
        \draw[blue, dashed] (Z2) -- (Y);
        \draw (1, 0) arc (0 : 180 : 1);
        \draw[dashed] (L) -- (Y);
        \draw[thick, blue] (Z1) -- (Z2);
        \draw[thick, blue] (A) -- (X) -- (B);
        \draw[blue, dashed] (A) -- (Z);
    
        \def\nodesize{0.03}
    
        \fill[blue] (A) circle (\nodesize)
            node[black, above left]{$a$};
        \draw (Y) node[below]{$y(x)$};
        \draw (B) node[above]{$b$};
        \draw (Z1) node[above]{$z_1$};
        \fill[blue] (Z2) circle (\nodesize)
            node[black, below right]{$z_2$};
        \fill[blue] (X) circle (\nodesize);
        \draw[shift={(X)}] (0, -.05) node[black, below]{$x$};
        
        \fill[white] (beta) circle (0.1) node[black]{$\beta$};
        
        \draw (C1) node[below]{$A_i$};
        
        \draw[shift={(Y)}] (-.2, 0) arc(180 : 139.152 : 0.2);
        \draw[shift={(Y)}] (-.23, 0) arc (180 : 139.152 : 0.23);
        \draw[shift={(Y)}] (-.2, .1) node[left]{\small$o_r(1)$};
        
        \def\rr{0.17}
        \draw[shift={(X)}]
        (-\rr, 0) arc (180 : 208 : \rr);
        \draw[shift={(X)}]
        (-\rr, 0) arc (180 : 139 : \rr);
        \def\rrr{0.15}
        \draw[shift={(X)}]
        (-0.7547 *\rrr, \rrr * .656) arc (139 : 70.4 : \rrr);
        
        \draw[shift={(X)}] (-.15, 0) node[left]{\footnotesize $\geq \pi/3$};
        
    \end{tikzpicture} 
    \caption{Lemma~\ref{hujugl}. Case about the energetic point of order $2$ and the single point corresponding to it.}
    \label{fig:netrudno_videtj}
\end{figure}

\begin{itemize}
\item \textbf{Let $x$ be an energetic point of order $2$.} Denote by $a,b \neq x$ the ends of the maximal segments such that $([ax]\cup [xb])\subset \Sigma$. 
\begin{itemize}
\item  \textbf{Suppose there exists a point $y_1\neq y$ corresponding to $x$.} As $B_r(y_1)\cap \Sigma=B_r(y)\cap \Sigma =\emptyset$ the inequalities $\angle y_1xa \geq \pi/2$ and $\angle yxb \geq \pi/2$ (without loss of generality, points $a$, $y_1$, $y$ and $b$ are located in counterclockwise order) hold. Moreover, in view of the optimality of $\Sigma$ the inequality $\angle axb \geq 2\pi/3$ holds; then angle $\angle y_1xy$ does not exceed $\pi/3$, and thus $y_1 \in B_r(y)\cap M$. But it is impossible since the angle between $[yx)$ and the side of $M$ is $o_r(1)$, and point $y_1$ belongs to this angle, so $|y_1x|<r$. 

\item \textbf{Let $y$~ be the single point corresponding to $x$.} In view of Lemma~\ref{lm:energetic_points_angles}(ii) there exists such maximal distance $[ax],[bx] \subset \Sigma$ that $\angle axb \geq 2\pi/3$, $[y(x)x)$ contains the bisector of $\angle axb$, and $\angle y(x)xa, \angle y(x)xb \geq \pi/2$, where $y(x)$ is a point corresponding to $x$. Without loss of generality, ray $[xa)$ is directed towards the line containing the first side. Denote by $\beta$ the area bounded by ray $[xa)$, straight line containing the first side and a circumference $\partial B_r(y)$. In view of conditions on the point $x$, expression 
$\overline {\beta} \subset B_{o(r)}(x)\subset B_r(x)$ holds. Denote by $H^+$ the half-plane, bounded by $(ax)$ and containing $\beta$. Then $H^+ \cap \overline {B_r([ax])}\cap N \subset B_r(x)$, and thus, in view of Proposition~\ref{hs}, point $a$ can not be energetic point of order $1$.  
If $a$ is an energetic point of order $2$  or a branching point, then there exists such a segment $[ac_1] \subset \Sigma$ that $c_1 \in \Int \beta$. In view of Proposition~\ref{hs} point $c_1$ can not be an energetic point of order $1$, in case if it is an energetic point of order $2$ or a branching point, there exists a next point $c_2\in \Sigma$ such that the maximal segment $]c_1c_2[ \subset \Int \beta$ and $[c_1c_2)\cap [xa)=\emptyset$ (and hence cuts off a half-plane without points covering by $c_2$, thus in view of Proposition~\ref{hs} point $c_2$ can not have order $1$). And so on, there exists an infinite sequence of points $c_i$ such that the maximal segments $[c_ic_{i+1}] \subset \Sigma$ satisfy $]c_ic_{i+1}[ \subset \Int \beta$ and $[c_{i}c_{i+1})\cap [xa)=\emptyset$ for every number $i$, which is impossible. %as a corresondent point for the limit point can note exist
\end{itemize}

\item \textbf{Let $x$ be an energetic point of order $1$.}

\begin{itemize}
\item \textbf{Let $y$ be the single point corresponding to $x$.} Consider the maximal segment $[wx]\subset \Sigma_2$. By Lemma~\ref{lm:energetic_points_angles}(i) the expression $y \in (wx)$ holds. Point $w$ can not be an energetic point of order $2$: assume the contrary, then, in view of the expressions $\angle y(w)yw=\angle y(w)yx=o_r(1)$ and $|y(w)w|=r$, there exists only one point $y(w)$ corresponding to $w$. Then in view of item~(ii) of Lemma~\ref{lm:energetic_points_angles} inequality $\angle ((yw], (wy(w)]) \geq \pi/3$ holds. But it contradicts the law of sines% А это невозможно ввиду малости угла между первой стороной и $(yw)$: по теореме синусов для треугольника 
\[
r=|y(w)w|=\sin \angle y(w)yw \frac{|yw|}{\sin \angle yy(w)w}\leq o_r(1)\frac{(p_M(i) + 4)r}{\sin (\pi/3-o_r(1))}=o(r).
\]
As $\Sigma$ is connected, point $w$ also can not be an energetic point of order $1$. And $w$ can not be a branching point, otherwise there will be such a maximal segment $[wc_1]\subset \Sigma$, that $[wc_1)$ intersects a line containing the first side. Denote by $\delta$ an area, bounded by $[wc_1)$, $[wx]$, $M$ and $\partial B_r(y)$. In what follows, we will apply the same reasoning to area $ \delta $ as to area $ \beta $ above. The distance from $M$ to an arbitrary point of area $\delta$ is $o(r)$. As $B_r(y) \cap \Sigma =\emptyset$ due to Proposition~\ref{hs} point $c_1$ can not be terminal. If $c_1$ has order $2$ or is branching point, there exists such a maximal segment $[c_1c_2]\subset \Sigma$, that $]c_1c_2[ \subset \Int \delta$ and $[c_1c_2)\cap[wc_1)=\emptyset$. And hence line $(c_1c_2)$ cut off the half-line, where there is no points, covered by $c_2$, so in view of Proposition~\ref{hs} point $c_2$ can not be of order $1$. And so on, thus there exists an infinite sequence of such points $c_i$ that maximal segments $[c_1c_2] \subset \Sigma$ satisfy $]c_ic_{i+1}[ \subset \Int \delta$ and $[c_{i}c_{i+1})\cap [wc_1)=\emptyset$, which is impossible.

\item \textbf{Suppose there exists a point $y' \neq y$, corresponding to $x$.}
As the distance from $x$ to the first side is $o_r(1)r$, the point $y'$ also is at a distance $o_r(1)r$ from the first side, otherwise the segment $[yy'] \subset B_r(y) \cup B_r(y')$ separates $x$ from another points of $\Sigma$. Thus the angle between the circles $\partial B_r(y)$ and $\partial B_r(y')$ is $o_r(1)$. Consider segment $[xw]\subset \Sigma_2$, then the angle between $(xw)$ and the first side is $\pi/2 - o_r(1)$.

Let $w' := \partial B_\varepsilon (x) \cap [xw]$ for sufficiently small $\varepsilon>0$. Consider triangle $\triangle yy'w'$, there $\angle yw'y' = \pi - o_r(1)$ and
$|y'w'|, |yw'| > r$. Then one can replace in $\Sigma$ segment $[xw']$ with segments $[w'v]$ and $[w'v']$, where $v := \partial B_r(y) \cap [yw']$,  $v' := \partial B_r(y') \cap [y'w']$.
This replacement decreases the length: let $t$ be the intersection point of the first side and line $(xw')$, then
\[
|yv| = |yw'| - r = \sqrt{|tw'|^2 + |ty|^2 - 2|tw'|\cdot |ty| \cos \angle ytw'} - \sqrt{|tx|^2 + |ty|^2 - 2|tx|\cdot |ty| \cos \angle ytw'} = 
\]
\[
\frac{|tw'|^2 + |ty|^2 - 2|tw'|\cdot |ty| \cos \angle ytw' - |tx|^2 + |ty|^2 - 2|tx|\cdot |ty| \cos \angle ytw'}{\sqrt{|tw'|^2 + |ty|^2 - 2|tw'|\cdot |ty| \cos \angle ytw'} + \sqrt{|tx|^2 + |ty|^2 - 2|tx|\cdot |ty| \cos \angle ytw'}}.
\]
Since $\angle ytw' = \frac{\pi}{2} - o_r(1)$, one has $\cos \angle ytw' = o_r(1)$ holds. Hence the numerator of the similar expression for $|yv'|$ is equal to
\[
|tw'|^2 - |tx|^2 + |ty|(|tw'|-|tx|) o_r(1) = (|tw'|-|tx|)(|tw'| + |tx| + o_r(1)|ty|) = \varepsilon r o_r(1). 
\]
On the other hand, the denominator of the expression for $|yv|$ is greater than $2r$, thus
$|yv| = \varepsilon o_r(1)$, similarly $|yv'| = \varepsilon o_r(1)$. By the definition, $|xw'| = \varepsilon$, thus the replacement reduces the length. Clearly it maintains connectivity and doesn't increase the energy.

\end{itemize}

\end{itemize}

\end{proof}

    Set $\overline{N \setminus \bigcup \alpha_i \setminus N_r} $ consists of $n$ connected components, each of them is a rectangle, we denote it by $l^{i,i+1}$. As there are no loops in $\Sigma$, as in view of~\ref{CorKar} statement $\Sigma \cap P_r =\emptyset$ holds and as $\Sigma \cap \alpha_i$ is connected for every $i$, there is only one number $j$ such that set $\Sigma \cap l^{j,j+1}$ consists of two connected components, for other numbers $j$ sets $\Sigma \cap l^{j,j+1}$ are connected. Wherein $l^{i,i+1} \cap M \subset \gamma (\Sigma \cap l^{i,i+1})$for every $i$.
As the sets $\Sigma \cap \partial \alpha_i \cap l^{i,i+1}$ and $\Sigma \cap \partial \alpha_{i+1} \cap l^{i,i+1}$ are points, each connected set $\Sigma \cap l^{i,i+1}$ is a segment connecting these points.

Let number $i$ be such that $\Sigma \cap l^{i,i+1}$ is not connected and let 
\[
\Sigma^i \sqcup \Sigma^{i+1} \defeq l^{i,i+1} \cap \Sigma,
\]
where $\Sigma^j$ is connected for $j=i,i+1$. Without loss of generality, $\Sigma^j\cap \partial \alpha_j \neq \emptyset$ where $j=i,i+1$. Denote $a_j\defeq \partial \alpha_j \cap M \cap l^{i,i+1}$ and $b_j \defeq \partial \alpha_j \cap \Sigma \cap l^{i,i+1}$. Thus $\gamma(\Sigma^i) \cup \gamma(\Sigma^{i+1}) \supset [a_ia_{i+1}]$.

Note that for arbitrary point $y \in l^{i,i+1}$ inequality $|yb_i|>r$ or $|yb_{i+1}|>r$ holds. Then there exists such $j \in  {i,i+1}$, that $|b_jy|>r$, where $y$ is the farthest from $A_j$ point from set $\gamma(\Sigma^j) \cap l^{i,i+1}$.
%Пусть $y$ --- самая дальняя от $A_j$ точка из множества $\gamma(\Sigma^j) \cap l^{i,i+1}$, где $j=i$, если $|b_iy|>r$, и $j=i+1$ в противном случае. 
Thus $\Sigma^j$ is a set of minimal length covering segment $[a_jy]$ and connecting point $b_j$ with set $\partial B_r(y)$. Then $\Sigma^j = [b_jc_j]$, where $c_j \defeq [b_jy]\cap \partial B_r(y)$. Hence point $c_j$ is energetic and $y=y(c_j)$ so $B_r(y)\cap \Sigma =\emptyset$. The second connected component $\Sigma^{2i+1-j}$ also is a segment connecting $\partial B_r(y)$ with $b_{2i+1-j}$.

By Lemma~\ref{lm:energetic_points_angles}(i) $b_j$, $c_j$ and $y$ are collinear, $b_{2i+1-j}$, $c_{2i+1-j}$ and $y$ are also collinear.
So 
\[
|b_jy| + |yb_{2i+1-j}| =  |b_jc_j| + |b_{2i+1-j}c_{2i+1-j}| + 2r.
\]
It is well-known that the minimal length of $|b_jy| + |yb_{2i+1-j}|$ over $y \in (a_ia_{i+1})$ is reached when the angle of incidence coincides with the angle of reflection, i.e.
\[
\angle b_jyA_j = \angle A_{2i+1-j}yb_{2i+1-j};
\]
obviously, such a point belongs to $[A_1A_n]$.
Hence $[b_{i}c_{i}]$ and $[b_{i+1}c_{i+1}]$ are almost parallel to $A_1A_n$.

%если $b_i,b_{i+1} \notin \overline{B_r(y)}$, то, поскольку множества $\Sigma^j$ являются минимальными множествами, покрывающими $[a_jy]$ и связывающими $b_j$ и $\partial B_r(y)$, выполняется $\Sigma^j = [b_jc_j]$, где $c_j \defeq [b_jy]\cap \partial B_r(y)$, где $j=i,i+1$. Тогда $H(\Sigma \cap l^{i,i+1})=\sqrt{|a_ib_i|^2+|a_iy|)^2}+ \sqrt{(|a_{i+1}b_{i+1}|^2+|a_{i+1}y|^2}-2r$. %В этом случае обозначим ${x_l,x_l} \defeq \{c_i,c_{i+1}\}$.
%В противном случае, то есть если, не умаляя общности $b_{i+1} \in \overline{B_r(y)}$, выполняется $\Sigma^i = [b_ic_i]$ и $\Sigma^{i+1} = \{b_{i+1}\}$, а значит $H(\Sigma \cap l^{i,i+1})=\sqrt{|a_ib_i|^2+|a_iy|)^2}-r$. %В этом случае обозначим ${x_l,x_l} \defeq \{c_i,b_{i+1}\}$.
 %В таком случае при фиксированных $a_i,b_i,a_{i+1},b_{i+1}$ надо найти такое положение $y$, что $H(\Sigma \cap l^{i,i+1})$ %минимально. В таком случае достаточно минимизировать 

Consider the following auxiliary problem.

\begin{problem}
Let $M$ be a polygon and $r < r_0(M)$. Find a connected set $\Sigma'$ of minimal length (one-dimensional Hausdorff measure $\H$) such that
\[
F_M(\Sigma') \leq r
\]
and $\Sigma'$ has a cycle around $P_r$.
\label{TheProblemCycle}
\end{problem}

By the same standard compactness argument as in~\cite{miranda2006one}, Theorem 2.1 allows to prove that Problem~\ref{TheProblemCycle} has a solution.

Recall that point $y \in M\setminus \bigcup_{i=1}^n \alpha_i$ satisfies $B_r(y)\cap \Sigma =\emptyset$;  $c_i, c_{i+1} = \partial B_r(y) \cap \Sigma$. Consider the set $\Sigma_1 := \Sigma \cup [c_ic_{i+1}]$; clearly $\H(\Sigma_1) - \H(\Sigma) = |c_{i}c_{i+1}| \leq 2r$. Also $|c_{i}c_{i+1}| = (2-o_r(1))r$.
Note that $\Sigma_1$ contains a cycle $\mathcal{C}_1$ around $P_r$.

Let $\Sigma'$ be an arbitrary solution of Problem~\ref{TheProblemCycle} with a cycle $\mathcal{C}$ around $P_r$. Obviously $\Sigma' \cap P_r$ is empty and 
$\Sigma' \cap l^{i,i+1}$ is a segment. Hence
\[
\H(\Sigma') - \H(\Sigma) = 2r - o(r).
\]

Now we want to find the topology of optimal $\Sigma'$. We show that in the case of a rectangle the length of a set satisfying conditions of Problem~\ref{TheProblemCycle} (except the minimal length) with another topology exceeds $\H(\Sigma')$ by over $4\cdot 10^{-6}r$.

\subsection{Study of \texorpdfstring{$\mathcal{C}$}{C}}
\label{Sect:studyofC}
First of all, note that $\mathcal{C}$ is a convex polygon with each angle no smaller than $2\pi/3$.

Let $C_1$, $C_2$ be the points on rays $[A_1A_2)$ and $[A_1A_4)$, respectively, such that
\[
|A_1C_1| = |A_1C_2| = (p_M(1) + 3)r =  \left(\frac{2}{\sqrt{3} \sin \frac{\angle A_1}{2}} + \cot \frac{\angle A_1}{2} + 3\right)r.
\]
Let $D_1,D_2\in M_r$ be such points that $|C_1D_1|=|C_2D_2|=r$. Denote $\mathcal{S}=\conv\{A_1, C_1, C_2, D_1, D_2\}$. Consider $\mathcal{C}\cap \mathcal{S}$: the part of $\mathcal{C}$ in the vicinity of angle $A_1$. %By Lemma~\ref{hujugl} $\mathcal{C} \cap \mathcal{S} \subset \Sigma$.
We are going to show that for a more specific choice of $\angle A_1$ the set of inner vertices of polygonal chain $\mathcal{C}\cap \mathcal{S}$ consists of two energetic points of degree two and a Steiner point between them.

Since $\angle A_1 > \pi/3$, and long sides of $\mathcal{C}$ are almost parallel to sides of $M$
\[
\turn (\mathcal{C} \cap \mathcal{S}) = \pi - \angle A_1 + o_r(1) < 2\pi/3.
\]

\begin{lemma}
\label{uniqueYforC}
Let $W \in \mathcal{C}$ be an energetic point of degree 2. Then $y(W)$ is unique.
\end{lemma}

\begin{proof}
Suppose the contrary, i.e. there exist different points $y_1(W)$ and $y_2(W)$. 

Suppose that $y_1(W)$ and $y_2(W)$ belong to the same side of $M$, without loss of generality it is $A_1A_2$ (in particular $y_i(W)$ may coincide with $A_i$).
Put $U_1 = \partial(B_r (y_1(W))) \cap \partial B_r (A_1A_2) \cap N$, $U_2 = \partial(B_r (y_2(W))) \cap \partial B_r(A_1A_2) \cap N$.
Since $B_r(y_1(W)) \cap \Sigma' = B_r(y_2(W)) \cap \Sigma' = \emptyset$, convex polygon $\mathcal{C}$ crosses line $(U_1U_2)$ twice at points $Y_1,Y_2 \in [U_1U_2]$.
The length of $[U_1U_2]$ is at most $2r$; on the other hand, the corresponding side of $P_r$ is a subset of line $(U_1U_2)$ and has length bigger than $2r$ for $r < r_0(M)$. Hence, $\mathcal{C}$ cannot be a cycle around $P_r$, contradiction.

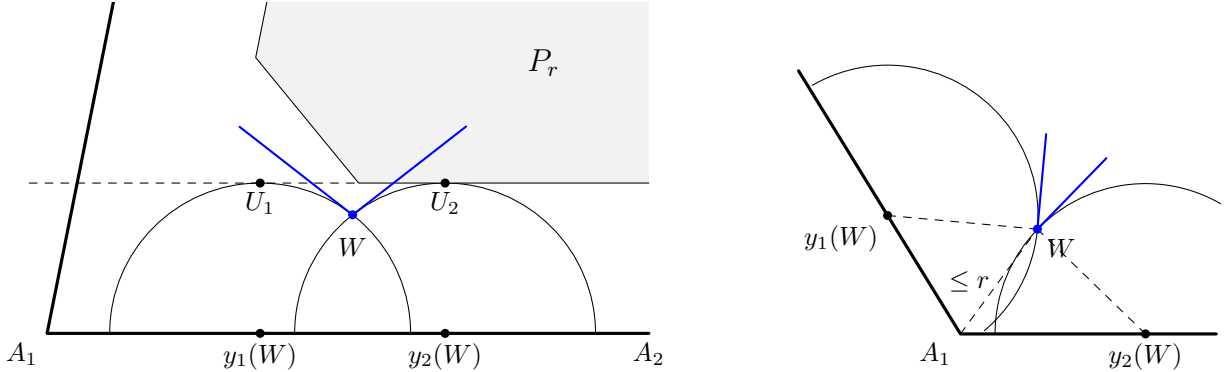
\begin{figure}[h]
    \hfill
    \begin{tikzpicture}[line cap=round,line join=round,>=triangle 45, scale = 2]
        
        \clip (-0.3, -0.3) rectangle (4.2, 2.2);
    
        \draw [very thick] (4,0)-- (0,0) -- (0.5,2.5);
        \draw [] (2.4164,0) arc (0 : 180 : 1);
        \draw [] (1.6457,0) arc (180 : 0 : 1);
        
        \fill[gray!10]
        (1.5198039027185568,2.5)-- (1.386782595922718,1.834893466020806)-- (2.071231517720798,1) -- (4,1) -- (4, 2.5) -- cycle;
        
        \draw  (1.5198039027185568,2.5)-- (1.386782595922718,1.834893466020806)-- (2.071231517720798,1) -- (4,1);
        
        \draw [,dashed] (-0.11911169737519867,1)-- (2.071231517720798,1);

        \draw[thick, blue]
        (1.28, 1.37387)-- 
        (2.031, 0.788) -- 
        (2.785, 1.376);

        \def\nodesize{0.03}
        \fill[] (0,0) %circle (\nodesize)
            node[below left]{$A_1$};
        \node[below] at(4, 0) {$A_2$};
        \fill[blue] (2.031,0.7888) circle (\nodesize);
        \node[black, below] at(2.031,0.7) {$W$};
        \fill[] (1.4164,0) circle (\nodesize)
            node[below]{$y_1(W)$};
        \fill[] (2.6457,0) circle (\nodesize)
            node[below]{$y_2(W)$};
        \fill[] (1.4164,1) circle (\nodesize)
            node[below]{$U_1$};
        \fill[] (2.6457,1) circle (\nodesize)
            node[below]{$U_2$};
            
        \node at (3.3, 1.8) {\large $P_r$};
        
    \end{tikzpicture}
    \hfill
    \begin{tikzpicture}[line cap=round,line join=round,>=triangle 45,scale=2]
        \coordinate (W) at (0.5113, 0.697);
        \coordinate (y1) at (-0.4845,0.788);
        \coordinate (y2) at (1.23,0);
    
        \draw [very thick] (1.7,0)-- (0,0) -- (-1.076,1.75);
        \draw [] (0.5155,0.788) arc (0 : 120 : 1);
        \draw [] (0.5155,0.788) arc (0 : -50 : 1);
        \draw [] (0.23,0) arc (180 : 60 : 1);
        \draw [dashed] (y1) -- (W);
        \draw [dashed] (y2) -- (W);
        \draw [dashed] (0,0) -- (W) node[pos=0.5, left]{$\leq r$};
        \draw [blue, thick] (0.5695,1.329)--(W)-- (0.9699,1.16888);
        
        \def\nodesize{0.03}
            
        \fill [] (0,0) node [below left]{$A_1$};
        \fill [] (y1) circle (\nodesize)
            node[below left]{$y_1(W)$};
        \fill [] (y2) circle (\nodesize)
            node[below]{$y_2(W)$};
        \fill [blue] (W) circle (\nodesize)
            node[black, below right]{$W$};
    \end{tikzpicture}
    \hfill
    $\mathstrut$
    %\caption{Случаи расположения точек в лемме~\ref{uniqueYforC}}  
    \caption{Cases of arrangement of points in Lemma~\ref{uniqueYforC}}
    \label{lemma3.1}
\end{figure}

Thus $y_1(W)$ and $y_2(W)$ belong to different sides of $M$. 
Suppose that $|A_1W| \leq  r$, then $\angle A_1 y_1(W) W \leq \angle WA_1y_1(W)$ and $\angle A_1 y_2(W) W \leq \angle WA_1y_2(W)$. By the sum of the angles in quadrangle $A_1y_1(W)Wy_2(W)$, we have
\[
\angle W = 2\pi - \angle A_1 - \angle y_1(W) - \angle y_2(W) \geq  2\pi - \angle A_1 - \angle WA_1y_1(W) - \angle WA_1y_2(W) = 2\pi - 2\angle A_1 > \frac{2\pi}{3}.
\]
So the angle between $\partial B_r(y_1(W))$ and $\partial B_r(y_2(W))$ is smaller than $\pi/3$. So $W$ cannot be a vertex of $\mathcal{C}$, since all the angles of $\mathcal{C}$ have measure at least $2\pi/3$.

So $|A_1W| > |y_2(W)W| = r$ and $|A_1W| > |y_1(W)W| = r$.
Hence open circles $B_r(y_1(W))$ and $B_r(y_2(W))$ cannot intersect, otherwise $\Sigma'$ does not cover $A_1$; so $B_r(y_1(W))$ and $B_r(y_2(W))$ are tangent. 
Again, $W$ cannot be a vertex of $\mathcal{C}$, since all the angles of $\mathcal{C}$ have measure at least $2\pi/3$; contradiction.
\end{proof}

\begin{lemma}\label{lm:energetic_neighbours}
$\mathcal{C}$ contains no energetic points $W_1$, $W_2 \in \mathcal{S} \setminus M_r$ of degree 2, such that $[W_1W_2] \subset \Sigma'$.
\end{lemma}

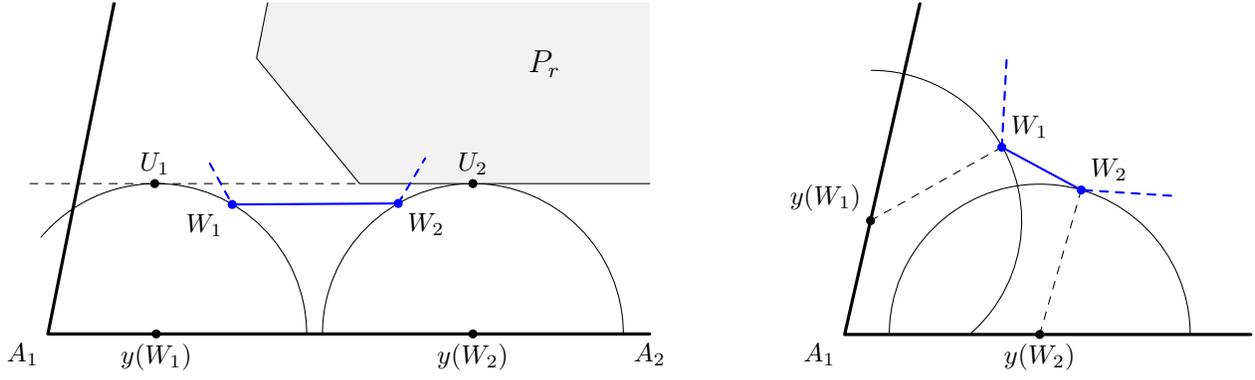
\begin{figure}[h]
    \hfill
    \begin{tikzpicture}[line cap=round,line join=round,>=triangle 45, scale = 2]
        
        \clip (-0.3, -0.3) rectangle (4.2, 2.2);
    
        \draw [very thick] (4,0)-- (0,0) -- (0.5,2.5);
        \draw [] (1.72,0) arc (0 : 140 : 1);
        \draw [] (3.824999342326124,0) arc (0 : 180 : 1);
        
        \fill[gray!10]
        (1.5198039027185568,2.5)-- (1.386782595922718,1.834893466020806)-- (2.071231517720798,1) -- (4,1) -- (4, 2.5) -- cycle;
        
        \draw  (1.5198039027185568,2.5)-- (1.386782595922718,1.834893466020806)-- (2.071231517720798,1) -- (4,1);
        
        \draw [,dashed] (-0.11911169737519867,1)-- (2.071231517720798,1);

        \draw [thick, blue, dashed] 
        (1.076604524, 1.133423838) --
        (1.224691222536784,0.8611887231248121);
        \draw[thick, blue]
        (1.224691222536784,0.8611887231248121)-- 
        (2.3285156383197783,0.8680460423595855);
        \draw[thick,blue,dashed]
        (2.3285156383197783,0.8680460423595855) --
        (2.507822365, 1.168589361);

        \def\nodesize{0.03}
        \fill[] (0,0) %circle (\nodesize)
            node[below left]{$A_1$};
        \fill[] (0.72,0) circle (\nodesize)
            node[below]{$y(W_1)$};
        \fill[] (2.824999342326124,0) circle (\nodesize)
            node[below]{$y(W_2)$};
        \fill[blue] (1.224691222536784,0.8611887231248121) circle (\nodesize)
            node[black, below left]{$W_1$};
        \fill[blue] (2.3285156383197783,0.8680460423595855) circle (\nodesize)
            node[black, below right]{$W_2$};
        \fill[] (0.71,1) circle (\nodesize)
            node[above]{$U_1$};
        \fill[] (2.824999342326124,1) circle (\nodesize)
            node[above]{$U_2$};
            
        \node at (3.3, 1.8) {\large $P_r$};
        \node[below] at(4, 0) {$A_2$};
        
    \end{tikzpicture}
    \hfill
    \begin{tikzpicture}[line cap=round,line join=round,>=triangle 45,scale=2]
        \draw [very thick] (2.7,0)-- (0,0) -- (0.5,2.2);
        \draw [] (0.83,0) arc (-49.22 : 90 : 1);
        \draw [] (2.2969093477858555,0) arc (0 : 180 : 1);
        \draw [dashed] (1.0445427735477741,1.245860012001794)-- (0.17208443509607213,0.7571715144227174);
        \draw [dashed] (1.5724079138778357,0.9613014824087514)-- (1.2969093477858555,0);
        \draw [blue, thick] (1.0445427735477741,1.245860012001794)-- (1.5724079138778357,0.9613014824087514);
        \draw [blue, thick, dashed] (1.044542773547774,1.2458600120017942)-- (1.0776332073831516,1.8446255640573688);
        \draw [blue, thick, dashed] (2.170867053122538,0.9230677042234854)-- (1.5724079138778357,0.9613014824087515);
        
        \def\nodesize{0.03}
            
        \fill [] (0,0) node [below left]{$A_1$};
        \fill [] (0.17208443509607213,0.7571715144227174) circle (\nodesize)
            node[above left]{$y(W_1)$};
        \fill [] (1.2969093477858555,0) circle (\nodesize)
            node[below]{$y(W_2)$};
        \fill [blue] (1.0445427735477741,1.245860012001794) circle (\nodesize)
            node[black, above right]{$W_1$};
        \fill [blue] (1.5724079138778357,0.9613014824087514) circle (\nodesize)
            node[black, above right]{$W_2$};
    \end{tikzpicture}
    \hfill
    $\mathstrut$
    %\caption{Случаи расположения точек в лемме~\ref{lm:energetic_neighbours}}  
    \caption{Cases of arrangement of points in Lemma~\ref{lm:energetic_neighbours}}
    \label{lemma3.3}
\end{figure}

\begin{proof}
Suppose the contrary. By Lemma~\ref{uniqueYforC} there are unique points $y(W_1)$, $y(W_2)$.

Consider the case that $y(W_1)$ and $y(W_2)$ belong to the same side of $M$, say $[A_1A_2]$.
Put $U_1 = \partial(B_r (y(W_1))) \cap \partial B_r (A_1A_2) \cap N$, $U_2 = \partial(B_r (y(W_2))) \cap \partial B_r(A_1A_2) \cap N$.
Since $B_r(y(W_1)) \cap \Sigma' = B_r(y(W_2)) \cap \Sigma' = \emptyset$, a convex polygon $\mathcal C$ intersects $[U_1U_2]$ at two points. Since $W_1,W_2 \in \mathcal{S}$, the length of $[U_1U_2]$ is smaller than the length of the corresponding side of $P_r$. Hence $\mathcal {C}$ cannot be a cycle around $P_r$, contradiction.

In the case when $y(W_1)$ and $y(W_2)$ belong to different sides of $M$, polygonal chain $y(W_1)W_1W_2y(W_2)$ splits $M$ into two parts. The smaller part cannot be covered by $\Sigma'$, because $\Sigma' \cap B_r(y(W_1)) = \emptyset$ and $\Sigma' \cap B_r(y(W_2)) = \emptyset$. We got a contradicion.

%Значит $y(W_1)$ и $y(W_2)$ принадлежат разным сторонам $M$. Заметим, что ломаная $y(W_1)W_1W_2y(W_2)$ делит $M$ на две части, причем меньшая часть не покрывается $\Sigma$. Противоречие.
\end{proof}

Now we analyze all possible arrangements of energetic points in $\mathcal{C}\cap \mathcal{S}$.

\begin{lemma}\label{lm:CycleInAngle}
Suppose that $\frac{\pi}{2} \leq \angle A_1 < \frac{7\pi}{12}$. 
Then $\mathcal{C} \cap \mathcal{S} = Z_1W_1VW_2Z_2$, where $Z_1 := \mathcal{C} \cap [C_{1}D_{1}]$, $Z_2 := \mathcal{C} \cap [C_{2}D_{2}]$, $W_1$ and $W_2$ are energetic points of degree two, and $V$ is a Steiner point.
\end{lemma}

\begin{proof}
Since $\turn \mathcal{C} \cap \mathcal{S} = \pi - \angle A_1 + o_r(1) < \frac{\pi}{2}$, $\mathcal{C} \cap \mathcal{S}$ contains at most one branching point. From the other hand, 
$\turn \mathcal{C} \cap \mathcal{S} > \frac{\pi}{3}$, so polygonal chain $\mathcal{C} \cap \mathcal{S}$ contains at least two inner vertices. 
By Lemma~\ref{lm:energetic_neighbours} they cannot all be energetic, so there is exactly one Steiner point $V \in \mathcal{C} \cap \mathcal{S}$.

By Lemma~\ref{lm:energetic_neighbours} there is at most one vertex of polygonal chain $\mathcal{C} \cap \mathcal{S}$ on the both sides of $V$.
Since we have at least two vertices the only case (up to symmetry) we have to deal with is $\mathcal{C} \cap \mathcal{S} = Z_1W_1VZ_2$, where $W_1$ is an energetic point of degree two, and $V$ is a Steiner point.

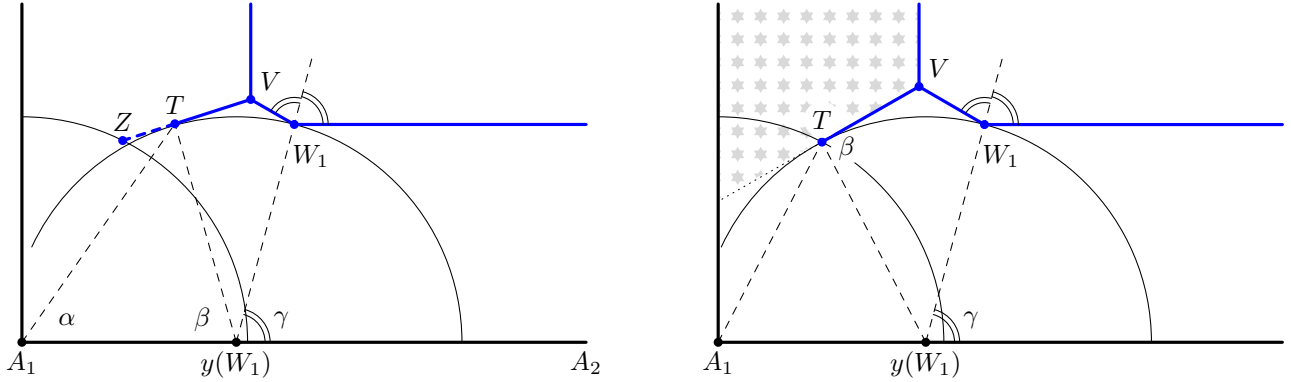
\begin{figure}[h]
    \centering
    \hfill
    \begin{tikzpicture}[line cap=round,line join=round,>=triangle 45,x=1cm,y=1cm, scale=3]
        
        \coordinate(W1) at (1.207, 0.966);
        \coordinate(yW1) at (.95, 0);
        \coordinate(V) at (1.014, 1.077);
        \coordinate(T) at (0.678, 0.969);
        \coordinate(Z) at (0.4467, 0.8947);

        \draw[shift={(yW1)}] (1, 0) arc (0:155:1);
        \draw (1, 0) arc (0 : 90 : 1);
        
        \draw[dashed] (yW1) -- (1.2841, 1.2558);
        \draw[dashed] (0,0) -- (T);
        \draw[dashed] (yW1) -- (T);
        
        \draw[very thick] (2.5, 0) -- (0, 0) -- (0, 1.5);
        \draw[blue, very thick] (2.5, 0.966) -- (W1) -- (V) -- (1.014, 1.5);
        \draw[blue, very thick] (V) -- (T);
        \draw[blue, dashed, very thick] (T) -- (Z);

        \def\nodesize{0.02};
        
        \fill[blue] (W1) circle (\nodesize);
        \fill[blue] (V) circle (\nodesize) node[black, above right]{$V$};
        \draw[shift={(W1)}] (-0.05, -0.05) node[black, below right]{$W_1$};
        \fill[blue] (T) circle (\nodesize) node[black, above]{$T$};
        \fill[blue] (Z) circle (\nodesize) node[black, above]{$Z$};
        \fill[black] (yW1) circle(\nodesize) node[below]{$y(W_1)$};
        \fill[black] (0,0) circle(\nodesize) node[below]{$A_1$};
        \draw (0.2, 0.1) node{$\alpha$};
        \draw[shift={(yW1)}] (-0.15, 0.1) node{$\beta$};
        
        \def\rr{0.15};
        \draw[shift={(yW1)}] (\rr, 0) arc (0 : 75 : \rr);
        \draw[shift={(yW1)}] (0.85*\rr, 0) arc (0 : 75 : 0.85*\rr);
        \draw[shift={(W1)}] (0.2588*0.65*\rr, 0.966*0.65*\rr) arc (75 : 150 : 0.65*\rr);
        \draw[shift={(W1)}] (0.2588*0.8*\rr, 0.966*0.8*\rr) arc (75 : 150 : 0.8*\rr);
        \draw[shift={(W1)}] (\rr, 0) arc (0 : 75 : \rr);
        \draw[shift={(W1)}] (0.85*\rr, 0) arc (0 : 75 : 0.85*\rr);
        \draw[shift={(yW1)}] (0.2, 0.1) node{$\gamma$};
        
        \draw[] (2.5, 0) node[below] {$A_2$};

    \end{tikzpicture}
    \hfill
    \begin{tikzpicture}[line cap=round,line join=round,>=triangle 45,x=1cm,y=1cm, scale=3]

        \coordinate(W1) at (1.18, 0.9656);
        \coordinate(yW1) at (.92, 0);
        \coordinate(V) at (0.89, 1.134);
        \coordinate(T) at (0.46, 0.888);
        \coordinate(T1) at (0, .625);
        
        \def\oblastj{(V) -- (T1) -- (0, 1.5) -- (0.89, 1.5) -- cycle}
        \draw[pattern color=gray!30!white, pattern=sixpointed stars] \oblastj;
        \draw[very thick,white] \oblastj;
        
        \draw[dotted] (T) -- (T1);

        \draw[shift={(yW1)}] (1, 0) arc (0:155:1);
        \draw (1, 0) arc (0 : 90 : 1);
        
        \draw[dashed] (yW1) -- (1.258, 1.25528);
        \draw[dashed] (0,0) -- (T);
        \draw[dashed] (yW1) -- (T);
        
        \draw[very thick] (2.5, 0) -- (0, 0) -- (0, 1.5);
        \draw[blue, very thick] (2.5, 0.9656) -- (W1) -- (V) -- (0.89, 1.5);
        \draw[blue, very thick] (V) -- (T);

        \def\nodesize{0.02};
        
        \fill[blue] (W1) circle (\nodesize);
        \fill[blue] (V) circle (\nodesize) node[black, above right]{$V$};
        \draw[shift={(W1)}] (-0.05, -0.05) node[black, below right]{$W_1$};
        \fill[blue] (T) circle (\nodesize);
        \fill[white, shift={(T)}] (0, .1) circle(0.08) node[black]{$T$};
        \fill[black] (yW1) circle(\nodesize) node[below]{$y(W_1)$};
        \fill[black] (0,0) circle(\nodesize) node[below]{$A_1$};
        
        \def\rr{0.15};
        \draw[shift={(yW1)}] (\rr, 0) arc (0 : 75 : \rr);
        \draw[shift={(yW1)}] (0.85*\rr, 0) arc (0 : 75 : 0.85*\rr);
        \draw[shift={(W1)}] (0.2588*0.65*\rr, 0.966*0.65*\rr) arc (75 : 150 : 0.65*\rr);
        \draw[shift={(W1)}] (0.2588*0.8*\rr, 0.966*0.8*\rr) arc (75 : 150 : 0.8*\rr);
        \draw[shift={(W1)}] (\rr, 0) arc (0 : 75 : \rr);
        \draw[shift={(W1)}] (0.85*\rr, 0) arc (0 : 75 : 0.85*\rr);
        \draw[shift={(yW1)}] (0.2, 0.1) node{$\gamma$};
        
        \fill[white,shift={(T)}] (0.11, -0.03) circle (0.07) node[black]{$\beta$}; 
        
    \end{tikzpicture}
    \hfill\phantom{}
%
%
%
    %\caption{Случаи в лемме~\ref{lm:CycleInAngle}}
    \caption{Cases in Lemma~\ref{lm:CycleInAngle}}
    \label{fig:absurd_case}
\end{figure}

Let $[VT]$ be the maximal segment of $\Sigma' \cap \mathcal{C}$, containing $V$ and different from $[VZ_2]$ and $[VW_1]$. Then $T$ cannot be a Steiner point, otherwise the depicted figure (bounded by $VT$, $M$ and $VZ_2$, see the right part of Fig.~\ref{fig:absurd_case}) contains an energetic point, which is an absurd.

Suppose that $T$ is an energetic point of degree 2. Then no $y(T)$ can lie on $[A_1A_n]$, so all $y(T)$ belong to $[A_1A_2]$. 
If $T \notin B_r(y(W_1))$ then the neighbourhood of $y(W_1)$ cannot be covered by $\Sigma'$. Indeed, consider $y(T)$ with the smallest distance from $y(W_1)$.
Then $\Sigma' \cap \Int (y(T) T V W_1 y(W_1)) = \emptyset$, which implies that $]y(W_1)y(T)[$ is not covered by $\Sigma'$, so $y(W_1) = y(T)$.

Recall that $\partial B_r(A_1) \cap \Sigma' \neq \emptyset$, and consider an arbitrary point $Z \in \partial B_r(A_1) \cap \Sigma'$. Note that $[ZT] \cup [TV] \cup \mathcal {C}$ covers $M \cap \mathcal{S}$, so $T$, $Z$, $A_1$ lie on the same line. Denote $\angle TA_1y(W_1)$ by $\alpha$, $\angle Ty(W_1)A_1$ by $\beta$ and $\angle W_1y(W_1)A_2$ by $\gamma$. Since $T$, $Z$, $A_1$ are collinear, $\angle A_1Ty(W_1) = \angle y(W_1)TV = \pi - \alpha - \beta$.
Sum of the angles in quadrilateral $y(W_1)W_1VT$ gives that
\[
2\pi = \angle y(W_1) + \angle W_1 + \angle V + \angle T = (\pi - \beta - \gamma) + (\pi - \gamma) + \frac{2\pi}{3} + (\pi - \alpha - \beta),
\]
so
\[
\alpha + 2\beta + 2 \gamma = \frac{5\pi}{3}.
\]
One can find turning of $\mathcal C \cap \mathcal{S}$ in two ways:
\[
\turn \mathcal C \cap \mathcal{S} = \pi - \angle A_1 + o_r(1) = \frac{\pi}{3} + (\pi - 2\gamma),
\]
so
\[
\alpha + 2\beta + \angle A_1 + o_r(1)= \frac{4\pi}{3}.
\]
Proposition~\ref{diffproposition} gives that 
\[
2\cos^2 \gamma = 2 \cos \beta \cos (\alpha + \beta). 
\]
Transform the second equation
\[
1 + \cos (2\gamma) = 2\cos^2 \gamma = 2 \cos \beta \cos (\alpha + \beta) = \cos \alpha + \cos (\alpha + 2\beta).
\]
Then
\[
\cos \alpha = 1 + \cos (2\gamma) - \cos (\alpha + 2\beta) = 1 + \cos \left (\frac{\pi}{3} + \angle A_1 +o_r(1) \right) - \cos \left (\frac{4\pi}{3} - \angle A_1 + o_r(1) \right ) =
\]
\[
1 - 2 \sin \frac{5\pi}{6}\sin \left (\frac{\pi}{2} - \angle A_1 + o_r(1) \right ) = 1 - \cos (\angle A_1 + o_r(1)).
\]
So for $\angle A_1 > \frac{\pi}{2}$ we have $\cos \alpha > 1$; if $\angle A_1 = \frac{\pi}{2}$, then $\alpha = o_r(1)$, so $T$ is $o(r)$-close to the side $A_1A_2$. Then
consider the angles in quadrilateral $TVW_1y(W_1)$: $\angle T$, $\angle W_1$ are at least $\pi/2$, $\angle y(W_1) > \pi/2 + o_r(1)$ and $\angle V = 2\pi/3$, which gives sum strictly greater than $2\pi$.

Finally $T$ may be a point of degree $1$. This means that $T \in \partial B_r (A_1) \cap \partial B_r (W_1)$. Note that $\angle A_1Ty(W_1) + \angle y(W_1)TV \geq \pi$ otherwise one can replace $[VT]$ with $[VT']$, where $T' := [A_1V] \cap \partial B_r(A_1)$ this replacement does not change the energy and decreases the length of $\Sigma'$.

Denote $\angle C_1y(W_1)W_1$ by $\alpha$ and $\angle y(W_1)TV$ by $\beta$.  
Consider the sum of angles in quadrilateral $W_1VTy(W_1)$:
\[
\angle W_1y(W_1)T = 2\pi - \angle W_1 - \angle V - \angle T = 2\pi - (\pi - \alpha) - \frac{2\pi}{3} - \beta  = \frac{\pi}{3} + \alpha - \beta.
\]
So 
\[
\angle y(W_1)A_1T = \angle A_1y(W_1)T = \pi - \angle C_1y(W_1)W_1 - \angle W_1y(W_1)T = \pi - \alpha - \left (\frac{\pi}{3} + \alpha - \beta \right ) = \frac{2\pi}{3} + \beta - 2\alpha.
\]
Consider the sum of angles in triangle $A_1y(W_1)T$:
\[
\angle y(W_1)TA_1 = \pi - \angle y(W_1)A_1T - \angle A_1y(W_1)T = 4\alpha - 2\beta - \frac{\pi}{3}.
\]
One can find turning of $\mathcal C \cap \mathcal{S}$ in two ways:
\[
\turn \mathcal C \cap \mathcal{S} = \pi - \angle A_1 = \frac{\pi}{3} + (\pi - 2\alpha),
\]
so $4\alpha = 2\pi/3 + 2\angle A_1$.

Recall that $\angle A_1TW_1 + \angle W_1TV \geq \pi$, it implies
\[
\left(4\alpha - 2\beta - \frac{\pi}{3}\right) + \beta \geq \pi, 
\]
so $\beta \leq 2\angle A_1 - 2\pi/3 < \pi/2$ since $\angle A_1 < 7\pi/12$.
From the other hand $B_r (y(W_1)) \cap \Sigma' = \emptyset$, hence $\beta \geq \pi/2$; contradiction.

\end{proof}

\section{Application of computational methods}
\label{Application}

In this section $M$ is a rectangle $A_1\dots A_4$. Define the following orthogonal coordinate system: $A_1$ is the origin point, $[A_1,A_2)$ is the $X$-axis, and $[A_1,A_4)$ is the $Y$-axis. 

Note that the intersections $\Sigma' \cap l^{i,i+1}$ are segments, parallel to $A_{i}A_{i+1}$.
Otherwise reflect a part $\Sigma' \cap \alpha_j$ of minimal length from perpendicular bisectors to the sides of $M$; denote the resulting set by $\sigma$. The set $\sigma'$ is obtained by connecting the components of $\sigma$ by segments, parallel to $A_{i}A_{i+1}$. By the choice of $j$ the length of $\sigma' \cap \left (\cup \alpha_i \right)$ is at most $\H (\Sigma' \cap \left (\cup \alpha_i \right))$ so $\H \left (\Sigma' \cap \left (\cup l^{i,i+1} \right) \right) \leq \H \left (\sigma' \cap \left (\cup l^{i,i+1} \right) \right)$ which implies the desired condition.

\begin{lemma}
\label{disjointcircles}
Let $W_1,W_2$ be defined as in Lemma~\ref{lm:CycleInAngle}. (Recall that by Lemma~\ref{uniqueYforC} $y(W_1)$ and $y(W_2)$ are defined uniquely.) Then
\[
B_r(y(W_1)) \cap B_r(y(W_2)) = \emptyset.
\]
\end{lemma}

\begin{proof}
Let $Q \in \Sigma'$ be a point that covers $A_1$, hence $|QA_1| \leq r$. From the other hand $B_r(y(W_1)) \cap \Sigma' = B_r(y(W_2)) \cap \Sigma' = \emptyset$, so $|y(W_1)Q|, |y(W_2)Q| \geq r$. Then $\angle y(W_1)A_1Q \geq \angle A_1y(W_1)Q$, $\angle y(W_2)A_1Q \geq \angle A_1y(W_2)Q$. Consider the sum of angles in quadrilateral $A_1y(W_1)Qy(W_2)$:
\[
2\pi = \angle A_1 + \angle y(W_1) + \angle y(W_2) + \angle Q =  \angle y(W_1)A_1Q + \angle y(W_2)A_1Q + \angle A_1y(W_1)Q + \angle A_2y(W_2)Q + \angle Q \leq 2\angle A_1 + \angle Q = \pi + \angle Q.
\]
So $\angle Q \geq \pi$. If $\angle Q = \pi$ then $B_r(y(W_1))$, $B_r(y(W_2))$ are tangent, which implies the conclusion of lemma. 

Otherwise $\angle Q > \pi$, which implies that $Q$ lies inside triangle $A_1y(W_1)y(W_2)$. Suppose the contrary, i.e. $B_r(y(W_1)) \cap B_r(y(W_2)) \neq \emptyset$. Then  $\Sigma'$ cannot intersect the side $[y(W_1)y(W_2)]$ of triangle $A_1y(W_1)y(W_2)$. Since $\Sigma' \subset N$ we got a contradiction.

\end{proof}

% Consider the angle $A_1$. Опустим перпендикуляры $C_{11}D_{11}$ и $C_{12}D_{12}$ из точек $C_{11}$ и $C_{12}$ на стороны
% $A_1A_2$ и $A_1A_3$ соответственно. Заметим, что $\mathcal{C}$ пересекает $[C_{11}D_{11}]$ и $[C_{12}D_{12}]$, а значит и $\Sigma'$ тоже. Также заметим, что $|A_1D_{11}| = |A_1D_{12}| = (4+\sqrt{3})r$.

Recall that Lemma~\ref{lm:CycleInAngle} states that $\mathcal{C} \cap \mathcal{S}=Z_1W_1VW_2Z_2$.
By Lemma~\ref{uniqueYforC}, points $y(W_1)$ and $y(W_2)$ are uniquely determined.
Consider the following parameters:
\[
\alpha := \frac{\angle Z_1W_1V}{2}, \quad x := \dist (A_1, y(W_1)), \quad y := \dist (A_1, y(W_2)).
\]
Define also
\[
\beta := \frac{\angle VW_2Z_2}{2}.
\]

One can find turning of $\mathcal C \cap \mathcal{S}$ in two ways:
\[
\turn \mathcal C \cap \mathcal{S} = \pi - \angle A_1 = (\pi - \angle W_1) + (\pi - \angle V) + (\pi - \angle W_2) = (\pi - 2\alpha) + \frac{\pi}{3} + (\pi - 2\beta).
\]
Hence
\[
2\alpha + 2\beta = \frac{4 \pi}{3} + \angle A_1, \quad \mbox{ so } \quad \beta = \frac{2 \pi}{3} + \frac{\angle A_1}{2} - \alpha = \frac{2 \pi}{3} + \frac{\pi}{4} - \alpha = \frac{11 \pi}{12} - \alpha.
\]
By Lemma~\ref{lm:CycleInAngle} and the fact that $\alpha, \beta \leq \frac{\pi}{2}$,
\[
x,y\in\left[0,\left(\frac{4}{\sqrt{6}} + 2 + 1\right)  r \right], \quad \alpha\in\left[\frac{5 \pi}{12},\frac{\pi}{2}\right]. 
\]
Since $\Sigma'\cap B_r(y(W_1))=\emptyset$, the left neighborhood of $y(W_1)$ is covered by some point $Q_1\neq W_1$ (see Fig.~\ref{minimizer90}): 
\[
Q_1\in\Sigma',\quad \dist(y(W_1),Q_1)=r,\quad \angle A_1y(W_1)Q_1\leq \frac{\pi}{2}.
\]
Similarly,
\[
Q_2\in\Sigma',\quad \dist(y(W_2),Q_2)=r,\quad \angle A_1y(W_2)Q_2\leq \frac{\pi}{2}.
\]
Finally, there is a point in $\Sigma'$ at a distance at most $r$ from $A_1$, so there is a point in $\Sigma'$ at a distance exactly $r$ from $A_1$:
\[
Q\in\Sigma',\quad\dist(Q,A_1)=r.
\]
Note that some of the points $Q$, $Q_1$ and $Q_2$ may coincide (in particular, some of them coincide for the optimal $\Sigma'$).

\begin{figure}
    \centering
    \begin{tikzpicture}[line cap=round,line join=round,>=triangle 45,x=1cm,y=1cm, scale = 2]
    
        \coordinate (V) at (1.4511769417267553,1.2941342560758087);
        \coordinate (W1) at (2.8223786956631924,0.9942091784463614);
        \coordinate (W2) at (0.9881455681973705,2.748339302798502);
    
        \draw [shift={(2.7149165651566225,0)},] (0:0.2558385377085195) arc (0:83.83096092621358:0.2558385377085195);
        
        \draw [shift={(W1)},] (0:0.22) arc (0:83.8309609262136:0.22);
        \draw [shift={(W1)},] (83.8309609262136:0.2558385377085195) arc (83.8309609262136:167.66192185242716:0.2558385377085195);

        \draw [shift={(0,2.5948194798051945)},] (8.830960926213633:0.2558385377085195) arc (8.830960926213633:90:0.2558385377085195);
        \draw [shift={(0,2.5948194798051945)},] (8.830960926213633:0.20040685453834028) arc (8.830960926213633:90:0.20040685453834028);
        \draw [shift={(W2)},] (8.83096092621364:0.2558385377085195) arc (8.83096092621364:90:0.2558385377085195);
        \draw [shift={(W2)},] (8.83096092621364:0.20040685453834028) arc (8.83096092621364:90:0.20040685453834028);
        \draw [shift={(W2)},] (-72.33807814757273:0.21319878142376622) arc (-72.33807814757273:8.830960926213644:0.21319878142376622);
        \draw [shift={(W2)},] (-72.33807814757273:0.157767098253587) arc (-72.33807814757273:8.830960926213644:0.157767098253587);
        
        \draw [very thick] (0,4.5)-- (0,0);
        \draw [very thick] (5,0)-- (0,0);
        \draw [] (1,0) arc (0 : 90 : 1);
        \draw [] (0,1.5948194798051945) arc (-90 : 90 : 1);
        \draw [] (3.7149165651566225,0) arc (0 : 180 : 1);
        \draw [,dashed] (0.5257398542969663,0.8506454053269227)-- (0,0)
            node[pos=0.4,below right]{$r$};
        \draw [thick, blue] (0.9881455681973705,4.043597731705948)-- (W2);
        \draw [thick, blue] (W2)-- (V);
        \draw [thick, blue] (V)-- (W1);
        \draw [thick, blue] (W1)-- (4.430166711325904,0.9942091784463614);
        \draw [,dashed] (0,2.5948194798051945)-- (W2)
            node[pos=0.5, below]{$r$};
        \draw [,dashed] (W2)-- (1.560382488301416,2.837242915294451);
        \draw [,dashed] (2.7149165651566225,0)-- (W1)
            node[pos=0.5, left]{$r$};
        \draw [,dashed] (W1)-- (2.891073104860997,1.629750385478656);
        \draw [] (4.430166711325904,1)-- (4.430166711325904,0);
        \draw [] (1,4.043597731705948)-- (0,4.043597731705948);
        \draw [thick, blue,dashed] (1.2367685238557085,1.0588168446924053)-- (1.0328898481033684,0.8350560302118302);
        \draw [thick, blue] (1.2367685238557085,1.0588168446924053)-- (V);
        
        \draw [decorate , decoration={brace, mirror, amplitude=3pt},      
            xshift=0pt,yshift=-4pt]
        (0, 0) -- (2.7149165651566225,0) node [midway,yshift=-0.5cm] {$x$};
        \draw [decorate , decoration={brace, amplitude=3pt},      
            xshift=-4pt,yshift=0pt]
        (0, 0) -- (0,2.5948194798051945) node [midway,xshift=-0.5cm] {$y$};
        
        \def\nodesize{0.03}
        
        \fill [] (0,0) circle (\nodesize)
            node[below left]{$A_1$};
        \fill [] (0,4.5)
            node[above]{$A_4$};
        \fill [] (5,0)
            node[right]{$A_2$};
            
        \fill [blue] (W1) circle (\nodesize)
            node[black, below left]{$W_1$};
        \draw[shift={(W1)}] (0.15, 0.05)
            node[above right]{\large $\alpha$};
        \fill [] (2.7149165651566225,0) circle (\nodesize)
            node[below right]{$y(W_1)$};

        \fill [blue] (W2) circle (\nodesize)
            node[black, above left]{$W_2$};
        \draw[shift={(W2)}] (0.1, 0.15)
            node[above right]{\large $\beta$};
        \fill [] (0,2.5948194798051945) circle (\nodesize)
            node[above left]{$y(W_2)$};

        \fill [blue] (V) circle (\nodesize)
            node[black, above right]{$V$}
            node[black, left]{$2\pi/3$};
        \draw[shift={(V)}] (0.1, -0.05) 
            node[below]{$2\pi/3$};
            
        \fill [blue] (0.9881455681973705,4.043597731705948) circle (\nodesize)
            node[black, above]{$Z_2$};
        \fill [blue] (4.430166711325904,0.9942091784463614) circle (\nodesize)
            node[black, right]{$Z_1$};
        \fill [] (0,4.043597731705948) circle (\nodesize)
            node[left]{$C_{2}$};
        \fill [] (4.430166711325904,0) circle (\nodesize)
            node[below]{$C_{1}$};
        \fill [] (1.1,4.043597731705948)
            node[right]{$D_{2}$};
        \fill [] (4.430166711325904,1.1) 
            node[above]{$D_{1}$};
            
        \fill [blue] (1.977522850, .6754631806) circle (\nodesize)
            node[black, above left]{$Q_{1}$};
            
        \fill [blue] (.3846153846, 1.671742557) circle (\nodesize)
            node[black, below right]{$Q_{2}$};
            
        \fill [blue] (0.8,0.6) circle (\nodesize)
            node[black, right]{$Q$};
    \end{tikzpicture}
    %\caption{Поведение $\mathcal{C}$ в окрестности $A_1$ для прямоугольника}
    \caption{Behaviour of $\mathcal C$ in the neighbourhood of $A_1$ for a rectangle}
    \label{minimizer90}
\end{figure}
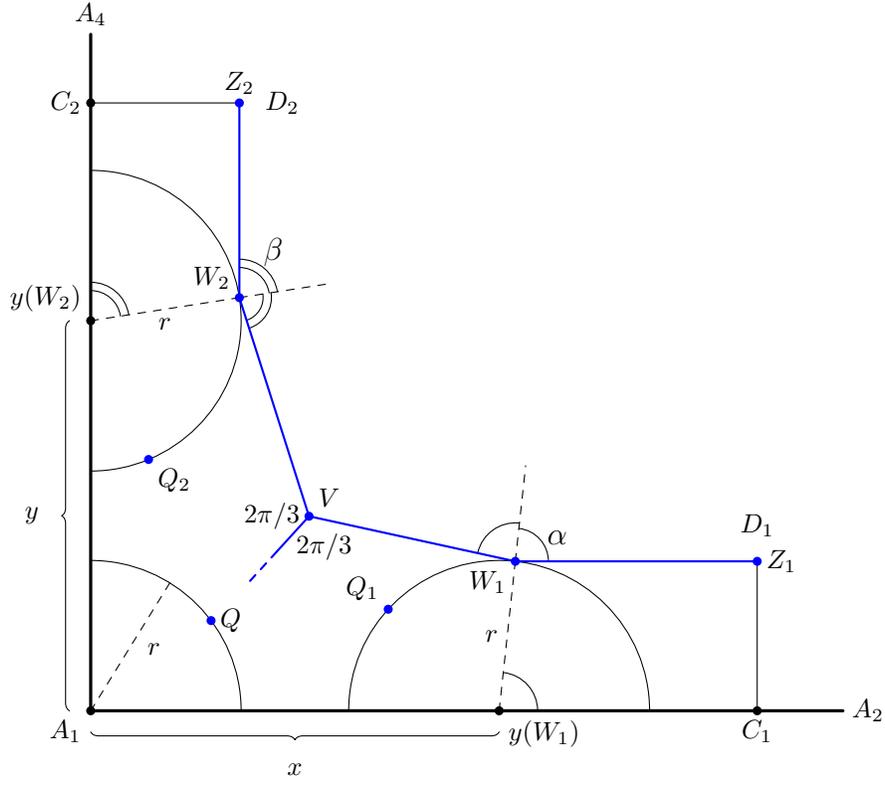

We use the following parametrization for the points $Q_1,Q_2,Q$:
\[
\xi_1=\angle A_1y(W_1)Q_1\in\left[0,\frac{\pi}{2}\right],\quad \xi_2=\angle A_1y(W_2)Q_2\in\left[0,\frac{\pi}{2}\right],\quad \xi=\angle A_2A_1Q\in\left[0,\frac{\pi}{2}\right].
\]

As we can see, each $\Sigma'$ corresponds to a set of parameters $\alpha(\Sigma'),x(\Sigma'),y(\Sigma'),\xi_1(\Sigma'),\xi_2(\Sigma'),\xi(\Sigma')$, i.e. a point $p(\Sigma')$ in the following $6$-dimensional parametric space:
\[
P=\left[0,\left(\frac{4}{\sqrt{6}} + 3\right)  r \right]\times\left[0,\left(\frac{4}{\sqrt{6}} + 3\right)  r \right]\times\left[\frac{5\pi}{12},\frac{\pi}{2}\right]\times\left[0,\frac{\pi}{2}\right]\times\left[0,\frac{\pi}{2}\right]\times\left[0,\frac{\pi}{2}\right].
\]

On the other hand, any $p\in P$ corresponds to a set of points
\[
W_1(p),V(p),W_2(p),Q_1(p),Q_2(p),Q(p),y(W_1)(p),y(W_2)(p),Z_1(p),Z_2(p);
\]
here $Z_1(p)\in C_1D_1$, $Z_2(p)\in C_2D_2$ are such points that $W_1(p)Z_1(p)$, $W_2(p)Z_2(p)$ are parallel to $A_1A_2$, $A_1A_4$; other points are defined as described above. Thus we can define 
\[
L(p)=C(p) + \H(\St(p)),
\]
where
\[
C(p) = |Z_1(p)W_1(p)|+|W_1(p)V(p)|+|V(p)W_2(p)|+|W_2(p)Z_2(p)|,
\]
$\St(p)$ is a Steiner tree on points $V(p),Q_1(p),Q_2(p),Q(p)$. 

By the definition of Steiner tree,
\[
\H(\Sigma' \cap \mathcal{S}) \geq L(p(\Sigma')).
\]

Consider a point $p_0 = (x_0, y_0, \alpha_0, \xi_{1,0}, \xi_{2,0}, \xi_0)$, where
\[
x_0=y_0=\sqrt{2} r,\ \alpha_0=\frac{11 \pi}{24}, \xi_{1,0} = \xi_{2,0} = \xi_0 = \frac{\pi}{4}.
\]
%i.e. if $\Sigma'_0$ is a minimizer, then $p(\Sigma'_0)=(x_0,y_0,\alpha_0,\xi_{1,0},\xi_{2,0},\xi_0)$. 
Note that $Q_1(p_0)=Q_2(p_0)=Q(p_0)$, so $\St(p_0)$ is the segment $V(p_0)Q(p_0)$ and there is a solution of Problem~\ref{TheProblemCycle} $\Sigma'_0$ such that $p(\Sigma'_0)=p_0$, $\H(\Sigma'_0 \cap \mathcal{S})=L(p_0)$. Denote $L_0=L(p_0)$.

Consider a subspace of $P$
\[
P_0 = \overline{B_\frac{r}{10}\left(x_0\right)} \times \overline{B_\frac{r}{10}\left(y_0\right)} \times \overline{B_\frac{1}{30}\left(\alpha_0\right)} \times\left[0,\frac{\pi}{2}\right]\times\left[0,\frac{\pi}{2}\right]\times\left[0,\frac{\pi}{2}\right],
\]
where $B_l(o)$ is a one-dimensional ball $(o-l, o+l)$.

We use computer search to show that if $p\in P \setminus P_0$, then either $L(p)\geq L_0$, or $p\neq p(\Sigma')$ for any minimizer $\Sigma'$ (in the latter case we will call such $p$ \textit{unobtainable}). Further we show that if $p(\Sigma') \in P_0$ for some minimizer $\Sigma'$ and $L(p(\Sigma')) < L(p_0)$ then $\Sigma'$ has a certain specific topology. Also we provide an explicit example of such configuration.

For $c=(x_c,y_c,\alpha_c,\xi_{1,c},\xi_{2,c},\xi_c)$, $\beta_c = \frac{11 \pi}{24} - \alpha_c$, $\Delta=(\Delta x,\Delta y,\Delta \alpha,\Delta \xi_1,\Delta \xi_2,\Delta \xi)$ define a box
\[
P(c,\Delta) := \overline{B_{\Delta x}(x) \times B_{\Delta y}(y) \times B_{\Delta \alpha}(\alpha) \times B_{\Delta \xi_1}(\xi_1) \times B_{\Delta \xi_2}(\xi_2) \times B_{\Delta \xi}(\xi)}.
\]
The algorithm described below takes a box $P(c,\Delta)$ and either verifies the required conditions for all points in the box, or divides it into smaller boxes, after which the same algorithm is executed on each of these new boxes. Initially the algorithm is executed on the full parametric space $P$.

Consider a box $P(c,\Delta)$. First, if there is at least one obtainable point in the box, i.e. $p=p(\Sigma')\in P(c, \Delta)$ for some minimizer $\Sigma'$, then $B_r(y(W_1)(p)) \cap B_r(y(W_2)(p)) = \emptyset$ should hold by Lemma~\ref{disjointcircles}. By the Pythagorean theorem for the triangle $A_1y(W_1)(p)y(W_2)(p)$ we conclude that if
\[
(x_c+\Delta x)^2+(y_c+\Delta y)^2<4r^2,
\]
then all $p\in P(c, \Delta)$ are unobtainable.

Second, there is no need to consider $P(c,\Delta) \subset P_0$, i.e. $P(c,\Delta)$ which satisfies
\[
B_{\Delta x}(x_c) \subset B_\frac{r}{10}\left(x_0\right),\quad B_{\Delta y}(y_c) \subset B_\frac{r}{10}\left(y_0\right),\quad B_{\Delta \alpha}(\alpha_c) \subset B_\frac{1}{30}\left(\alpha_0\right).
\]

Suppose that neither of two conditions above hold. Then we obtain the lower bound on $L(c)$ in the following way: we calculate the coordinates of the points $Z_1(c)$, $W_1(c)$, $V(c)$, $W_2(c)$, $Z_2(c)$ and evaluate $C(c)$ explicitly; we use Melzak's algorithm~\cite{melzak1961problem} to get the lower bound on $\H(\St(c))$.

Now, there exists a function $\Err(c,\Delta)$ such that $L(p)\geq L(c)-\Err(c,\Delta)$ holds for any $p\in P(c,\Delta)$; we show how to obtain $\Err$ in Subsection~\ref{ErrorEstimation}.

If $L(c)-\Err(c,\Delta)\geq L_0$, then $L(p)\geq L_0$ for any $p\in P(c,\Delta)$. Otherwise, we split $P(c,\Delta)$ into $2^6$ boxes by dividing each of its sides into two equal halves, and consider each of these boxes in the same way.

% Note that because of symmetry it is sufficient to consider only those of $10^6$ initial boxes whose center satisfies $x\geq y$.

When checking inequalities, we take into account the numerical errors by using C-XSC interval arithmetic package~\cite{hofschuster2004c},~\cite{klatte1993c}. In some cases, we are unable to check some of the inequalities in the reconstruction phase of the Melzak's algorithm; in these cases we assume that the reconstruction phase is successful. That means that our implementation of the Melzak's algorithm can sometimes return value smaller than the length of a Steiner tree on given points; this is of no consequence since we are only interested in the lower bound on this length.

Each considered box is treated as a separate task; these tasks are run in parallel. The search has finished in approximately 50 hours of real time (350 hours of total CPU time) on a 2.3 GHz Quad-Core Intel Core i5 processor. The source code can be found at~\cite{c0pymaster_minimizer}.

\subsection{Error estimation} 
\label{ErrorEstimation}

We want to obtain the inequality of the form
\[
L(p)\geq L(c)-\Err(c,\Delta)
\]
for any $p\in P(c,\Delta)$. By the definition of $L$,
\[
L(c) - L(p) \leq C(c) - C(p) + |\H(\St(c)) - \H(\St(p))|.
\]

When a terminal of a Steiner tree shifts by a vector of length no more than $\eps$, the length of a tree changes by at most $\eps$ (note that the topology of a Steiner tree may change). Using this fact, we get
\[
|\H(\St(c)) - \H(\St(p))| \leq |V(c) - V(p)| + |Q_1(c) - Q_1(p)| + |Q_2(c) - Q_2(p)| + |Q(c) - Q(p)| \leq |V(c) - V(p)| + (\Delta \xi_1 + \Delta \xi_2 + \Delta \xi) r.
\]

Denote $p=(x,y,\alpha,\xi_1,\xi_2,\xi)$, $\beta = \beta(p) = \frac{11 \pi}{12} - \alpha$. For brevity, we are going to omit the arguments and write $V$ instead of $V(p)$ (and so on). We have
\[
W_1 = (x + r\cos\alpha, r\sin\alpha), \quad W_2 = (r\sin\beta, y + r\cos\beta).
\]
So 
\[
|Z_1W_1| = \left( \frac{4}{\sqrt{6}} + 4 \right)r - x - r\cos\alpha, \quad \quad  |W_2Z_2| = \left( \frac{4}{\sqrt{6}} + 4 \right)r - y - r\cos\beta.
\]

Since $V$ is the intersection point of $(W_1V)$ and $(W_2V)$ we have
\begin{equation}
    \label{xuistem}
    V = W_1 + l_1 \cdot (\cos 2\alpha, \sin 2\alpha) = W_2 + l_2 \cdot (\sin 2\beta, \cos 2\beta)
\end{equation}
for some $l_1$ and $l_2$ (in fact, $l_1=|W_1V|$ and $l_2=|W_2V|$).
The solution of system~\eqref{xuistem} is
\[  
l_1 = -\frac{x \cos 2\beta + y \sin 2\beta + r\cos(\alpha + 2\beta) + r\sin\beta}{\cos 2(\alpha + \beta)}, \quad l_2 = -\frac{y \cos 2\alpha + x \sin 2\alpha + r\cos(2\alpha + \beta) + r\sin \alpha}{\cos 2(\alpha + \beta)}.
\]
Recall that $\beta = \frac{11\pi}{12} - \alpha$, so $2(\alpha + \beta) = \frac{11\pi}{6}$ and $\cos 2(\alpha + \beta) = \frac{\sqrt{3}}{2}$. So
\[
|W_1V| = -\frac{2}{\sqrt{3}} (x \cos 2\beta + y \sin 2\beta + r\cos(\alpha + 2\beta) + r\sin\beta), \quad |VW_2| = -\frac{2}{\sqrt{3}} (y \cos 2\alpha + x \sin 2\alpha + r\cos(2\alpha + \beta) + r\sin \alpha). 
\]
Hence
\[
|W_1V| + |VW_2| = -\frac{2}{\sqrt{3}} (x(\cos 2\beta + \sin 2\alpha) + y(\sin 2\beta + \cos 2\alpha) + r((\sin \alpha + \cos(\alpha + 2\beta)) + (\sin\beta + \cos(2\alpha + \beta))))=
\]
\[
\frac{2}{\sqrt{3}} \left(x\sin\left( 2\alpha - \frac{2\pi}{3} \right) + y\cos\left( 2\alpha - \frac{2\pi}{3} \right) -
r\left ( \cos\left( \alpha - \frac{\pi}{6} \right) + \sin\left( \alpha - \frac{\pi}{4} \right) \right) \right).
\]

Summing up
\[
C(p) = |Z_1W_1| + |W_1V| + |VW_2| + |W_2Z_2| =
\]
\[
2\left( \frac{4}{\sqrt{6}} + 4 \right)r - x - \cos\alpha - y - \cos\beta
+ \frac{2}{\sqrt{3}}\left( x \sin\left( 2\alpha - \frac{2\pi}{3} \right) + y \cos\left( 2\alpha - \frac{2\pi}{3} \right) - r\cos\left( \alpha - \frac{\pi}{6} \right) - r\sin\left( \alpha - \frac{\pi}{4} \right) \right).
\]
Then
\[
\frac{\partial C}{\partial x} = -1 + \frac{2}{\sqrt{3}} \sin\left( 2\alpha - \frac{2\pi}{3} \right), \quad \frac{\partial C}{\partial y} = -1 + \frac{2}{\sqrt{3}} \cos\left( 2\alpha - \frac{2\pi}{3} \right),
\]
and, keeping in mind that $\frac{\partial \beta}{\partial \alpha} = -1$,
\[
\frac{\partial C}{\partial \alpha} = \sin \alpha - \sin \beta + \frac{2}{\sqrt{3}}\left( 2x \cos\left( 2\alpha - \frac{2\pi}{3} \right) - 2y \sin\left( 2\alpha - \frac{2\pi}{3} \right) - \sin\left( \alpha - \frac{\pi}{6} \right) + \cos\left( \alpha - \frac{\pi}{4} \right) \right).
\]

Recall the following classical bound. Let $f: [a,b] \to \mathbb{R}^k$ be a differentiable function, then
\begin{equation}
|f(b) - f(a)|\leq (b - a) \cdot \sup_{a < x < b} |f'(x)|.
\label{meanvalue}    
\end{equation}
% Let $f: \mathbb{R}^k \to \mathbb{R}$ be a function, then
% \[
% |f(\bar v + \bar \varepsilon) - f(\bar v)| \leq \varepsilon |\max \nabla f|,
% \]
% where $|\max \nabla f|$ is coordinate-wise maximum of the gradient of $f$ on the parallelepiped spanned by $\bar v$ and $\bar v + \bar \varepsilon$. 

For a smooth function $F: P(c, \Delta) \to \mathbb{R}^k$ (in particular, for functions $C$, $V$) several applications of~\eqref{meanvalue} give
\[
|F(c)-F(p)|\leq \Delta_x F \cdot \Delta x + \Delta_y F \cdot \Delta y + \Delta_\alpha F \cdot \Delta \alpha,
\]
where
\[
\Delta_x F = \sup_{p \in P(c, \Delta)} \left|\frac{\partial F}{\partial x}(p)\right|,\quad \Delta_y F = \sup_{p \in P(c, \Delta)} \left|\frac{\partial F}{\partial y}(p)\right|,\quad \Delta_\alpha F = \sup_{p \in P(c, \Delta)} \left|\frac{\partial F}{\partial \alpha}(p)\right|.
\]

Using calculations above, we obtain the following bounds:
\begin{equation}
\Delta_x C \leq 1 - \frac{2}{\sqrt{3}} \sin\left( 2(\alpha_c - \Delta \alpha) - \frac{2\pi}{3} \right), \quad \Delta_y C \leq 1 - \frac{2}{\sqrt{3}} \cos\left( 2(\alpha_c + \Delta \alpha) - \frac{2\pi}{3} \right), \quad
\Delta_\alpha C \leq \max( F_1, F_2 ),
\label{eq-deltac}
\end{equation}
where
\[
F_1 = \sin (\alpha_c + \Delta \alpha) - \sin (\beta_c - \Delta \alpha) +
\]
\[
\frac{2}{\sqrt{3}}\left( 2(x_c + \Delta x) \cos\left( 2(\alpha_c - \Delta \alpha) - \frac{2\pi}{3} \right) - 2(y_c - \Delta y) \sin\left( 2(\alpha_c - \Delta \alpha) - \frac{2\pi}{3} \right) - \sin\left( \alpha_c - \Delta \alpha - \frac{\pi}{6} \right) + \cos\left( \alpha_c - \Delta \alpha - \frac{\pi}{4} \right) \right),
\]
\[
F_2 = -\sin (\alpha_c - \Delta \alpha) + \sin (\beta_c + \Delta \alpha) -
\]
\[
\frac{2}{\sqrt{3}}\left( 2(x_c - \Delta x) \cos\left( 2(\alpha_c + \Delta \alpha) - \frac{2\pi}{3} \right) - 2(y_c + \Delta y) \sin\left( 2(\alpha_c + \Delta \alpha) - \frac{2\pi}{3} \right) - \sin\left( \alpha_c + \Delta \alpha - \frac{\pi}{6} \right) + \cos\left( \alpha_c + \Delta \alpha - \frac{\pi}{4} \right) \right).
\]

Now, using~\eqref{xuistem}, we get
\[
\frac{\partial V}{\partial x} = \left( 1 - \frac{\cos 2\beta \cos 2\alpha}{\cos 2(\alpha + \beta)} , -\frac{\cos 2\beta \sin 2\alpha}{\cos 2(\alpha + \beta)} \right) = -\frac{\sin 2\alpha}{\cos 2(\alpha + \beta)} \cdot (\sin 2\beta, \cos 2\beta),
\]
so
\[
\Delta_x V = \frac{2}{\sqrt{3}}\sin 2(\alpha_c - \Delta \alpha).
\]
Similarly,
\[
\frac{\partial V}{\partial y} = \left( -\frac{\sin 2\beta \cos 2\alpha}{\cos 2(\alpha + \beta)} , -\frac{\sin 2\beta \sin 2\alpha}{\cos 2(\alpha + \beta)} \right) = -\frac{\sin 2\beta}{\cos 2(\alpha + \beta)} \cdot (\cos 2\alpha, \sin 2\alpha),
\]
so
\[
\Delta_y V = \frac{2}{\sqrt{3}}\sin 2(\beta_c - \Delta \alpha).
\]
Again, by~\eqref{xuistem}
\[
\frac{\partial V}{\partial \alpha} = \left( \frac{4  y \cos 4 \beta - 4  x \sin 4  \beta - 4  y + 4  \cos 3 \beta - 4  \cos\beta - \sin 3 \alpha  + \sin\left(3  \alpha + 4  \beta\right) - 3  \sin\left(\alpha + 4  \beta\right) + 3  \sin \alpha }{2  {\left(\cos\left(4  \alpha + 4  \beta\right) + 1\right)}} \right. ,
\]
\begin{equation}
    \left. -\frac{4  x \cos 4 \beta + 4 y \sin 4  \beta + 4  x - \cos \alpha  - \cos\left( 3 \alpha + 4  \beta\right) + 3  \cos\left(\alpha + 4  \beta\right) + 3  \cos \alpha + 4  \sin 3  \beta  - 4  \sin \beta }{2  {\left(\cos\left(4 \alpha + 4 \beta\right) + 1\right)}} \right).
\label{derivativeValpha}
\end{equation}

\begin{proposition}
\label{vmesto}
The following holds
\[
\left| \frac{\partial V}{\partial \alpha} \right| = \left| \frac{4 x \cos 2\beta + 4 y \sin 2\beta + 4 \sin \beta + 3 \cos(\alpha + 2\beta) - \cos(3\alpha + 2\beta)}{\cos (4\alpha + 4\beta) + 1} \right|.
\]
\end{proposition}

\begin{proof}
Clearly,
\[
\left| \frac{\partial V}{\partial \alpha} \right|^2 = \left(\frac{\partial V}{\partial \alpha} \right)_x^2 + \left(\frac{\partial V}{\partial \alpha} \right)_y^2.
\]
Now we substitute~\eqref{derivativeValpha} and open the brackets in SageMath~\cite{sagemath}; the source code can be found at~\cite{c0pymaster_minimizer}.

\end{proof}

Now $\Delta_\alpha V$ can be estimated as a maximum of two values, similar to estimation of $\Delta_\alpha C$ in~\eqref{eq-deltac}.

\section{The case \texorpdfstring{$p \in P_0$}{p in P0}}
\label{sect:thefinal}

Let $p \in P_0$; consider the Steiner tree $\St(p)$. 
Recall than $\St(p)$ connects points $Q_1 \in \partial B_r(y(W_1))$, $Q_2 \in \partial B_r(y(W_2))$, $Q \in \partial B_r(A_1)$ and $V$,
so by definition $\St(p)$ has at most 4 distinct terminals.
In this section we consider possible topologies of $\St(p)$.

\paragraph{Case 1. The terminals $\St(p)$ on the circumferences $\partial B_r(y(W_1))$ and $\partial B_r(y(W_2))$ coincide, i.e. $Q_1 = Q_2$.}
Recall that $B_r(y(W_1)) \cap B_r(y(W_2)) = \emptyset$, so these circumferences are tangent to each other at point $Q_1$.
So $A_1Q_1$ is a median in right-angled triangle $A_1y(W_1)y(W_2)$ with hypotenuses $2r$, so $|Q_1A_1| = r$, and hence $Q = Q_1 = Q_2$ and $\St(p) = [VQ]$.
Recall that $B_r(y(W_1)) \cap \Sigma' = B_r(y(W_2)) \cap \Sigma' = \emptyset$, hence $[VQ]$ is tangent to both circumferences, so $[VQ] \perp [y(W_1)y(W_2)]$. 

Point $V$ lies on the perpendicular bisector of $y(W_1)y(W_2)$, hence $|Vy(W_1)| = |Vy(W_2)|$ and $\angle QVy(W_1) = \angle QVy(W_2)$. 
Since $\angle QVW_1 = \angle QVW_2 = 2\pi/3$ the angles $\angle y(W_1)VW_1$ and $\angle y(W_2)VW_2$ are also equal.
Recall that $|y(W_1)W_1| = |y(W_2)W_2| = r$. 
%тут бы ссылочку на косой признак равенства треугольников по-английски. на вики это не написано.
So the triangles $y(W_1)W_1V$ and $y(W_2)W_2V$ are either equal or $\angle y(W_1)W_1V + \angle y(W_2)W_2V = \pi$, but 
\[
\angle y(W_1)W_1V + \angle y(W_2)W_2V = (\pi-\alpha) + (\pi - \beta) = 13\pi/12.
\]
Hence 
\[
\pi - \alpha = \angle y(W_1)W_1V = \angle y(W_2)W_2V = \alpha + \pi/12,
\]
so $\alpha = 11\pi/24$ and $\Sigma'(p) = \Sigma'_0$. 

We have to estimate $\H(\Sigma'_0)$. Recall that we use the coordinate system defined in Section~\ref{Application}, assume that $r=1$. By definition
\[
Z_1(p_0) = \left(p_M(i) + 3, \sin\frac{11\pi}{24}\right) \approx (5.632993, 0.991445) \quad \quad Z_2(p_0) = \left(\sin\frac{11\pi}{24}, p_M(i) + 3\right) \approx (0.991445, 5.632993). 
\]
Since $x_0 = y_0 = \sqrt{2}$
\[
W_1(p_0) = \left(\sqrt{2} + \cos \frac{11\pi}{24}, \sin\frac{11\pi}{24}\right) \approx (1.544740, 0.991445) \quad\quad
W_2(p_0) = \left(\sin\frac{11\pi}{24}, \sqrt{2} + \cos \frac{11\pi}{24}\right) \approx (0.991445, 1.544740).
\]
Solving system~\eqref{xuistem} we have
\[
V(p_0) \approx (1.108370, 1.108370).
\]
Finally, $Q(p_0) = Q_1(p_0) = Q_2(p_0) = \left( \frac{1}{\sqrt{2}}, \frac{1}{\sqrt{2}} \right) \approx (0.707107, 0.707107)$.

So we have $\H(\Sigma'_0) \approx 9.647504$. Since every coordinate has error term $\frac{1}{2}10^{-6}$, a point lies in $\frac{1}{\sqrt{2}}\cdot10^{-6}$-neighborhood of its approximation. Then by triangle inequality error of the length of every segment is at most $\sqrt{2}\cdot10^{-6}$. Since $\Sigma'_0$ consists of five segments, the total error is at most $5\sqrt{2}\cdot 10^{-6} < 8 \cdot 10^{-6}$. It implies the inequality
\begin{equation}
\H (\Sigma'_0) > 9.647496.
\label{sigmazero}    
\end{equation}

\begin{figure}[h]
    \centering
    \hfill
    \begin{tikzpicture}[line cap=round,line join=round,>=triangle 45,x=1cm,y=1cm, scale=3]
        
        \def\area{(-0.4,-0.2) rectangle (2.4,1.8)}
        
        %\fill[yellow] \area;
        \clip \area;
        
        \coordinate (yW1) at (1.6114223987970764,0);
        \coordinate (yW2) at (0,1.1846171755698447);
        \coordinate (W1) at (1.6465681639300307,0.9993821967561856);
        \coordinate (W2) at (0.9767880005986617,1.3988254898026702);
        \coordinate (Q) at (0.8057111993985382,0.5923085877849223);
        \coordinate (V) at (1.1393091294555364,1.0460983755512507);
        
        \draw [shift={(W1)},] (0:0.16811171893165086) arc (0:87.9858811966611:0.16811171893165086);
        \draw [shift={(W1)},] (87.9858811966611:0.19613033875359268) arc (87.9858811966611:174.73816999204195:0.19613033875359268);
        \draw [shift={(yW1)},] (0:0.16811171893165086) arc (0:87.98588119666111:0.16811171893165086);
        \draw [shift={(W2)},] (12.369084996020984:0.16811171893165086) arc (12.369084996020984:90:0.16811171893165086);
        \draw [shift={(W2)},] (12.369084996020984:0.1316875131631265) arc (12.369084996020984:90:0.1316875131631265);
        \draw [shift={(W2)},] (-65.26183000795807:0.19613033875359268) arc (-65.26183000795807:12.369084996020987:0.19613033875359268);
        \draw [shift={(W2)},] (-65.26183000795807:0.1597061329850683) arc (-65.26183000795807:12.369084996020987:0.1597061329850683);
        \draw [shift={(yW2)},] (12.36908499602097:0.16811171893165086) arc (12.36908499602097:90:0.16811171893165086);
        \draw [shift={(yW2)},] (12.36908499602097:0.1316875131631265) arc (12.36908499602097:90:0.1316875131631265);
        \draw[] (0.8695627035358539,0.5453689470405105) -- (0.9165023442802657,0.6092204511778262) -- (0.85265084014295,0.656160091922238);
        
        \draw [very thick] (0,4)-- (0,0);
        \draw [very thick] (0,0)-- (4,0);
        \draw [] (0,.1846171755698447) arc (-90 : 90 : 1);
        \draw [] (0.6114223987970764,0) arc (180 : 0 : 1);
        
        \draw [thick, blue] (V)-- (Q);
        \draw [thick, blue] (V)-- (W2);
        \draw [thick, blue] (W2)-- (0.9767880005986617,2.750530680269417);
        \draw [thick, blue] (V)-- (W1);
        \draw [thick, blue] (W1)-- (3.1468345584737527,0.9993821967561856);
        \draw [,dashed] (Q)-- (0,0);
        \draw [,dashed] (V)-- (yW1);
        \draw [,dashed] (V)-- (yW2);
        \draw [,dashed] (yW2)-- (Q) node[below left, pos=0.5]{$r$};
        \draw [,dashed] (Q)-- (yW1) node[below left, pos=0.5]{$r$};
        \draw [,dashed] (1.296092521940362,1.4688485482366445)-- (yW2) node[above, pos=0.6]{$r$};
        \draw [,dashed] (1.6577766077554446,1.3180982403559112)-- (yW1) node[right, pos=0.6]{$r$};

        \def\nodesize{0.025}
        
        \node[below left] at(0,0) {$A_1$};
        
        \fill (yW1) circle (\nodesize) node[black, below] {$y(W_1)$};
        \fill[blue] (W1) circle (\nodesize) node[black, below right] {$W_1$};
        \fill (yW2) circle (\nodesize) node[black, left] {$y(W_2)$};
        \fill[blue] (W2) circle (\nodesize) node[black, above left] {$W_2$};
        \fill[blue] (Q) circle (\nodesize) node[black, below] {$Q$};
        \fill[blue] (V) circle (\nodesize);
        \node[above right] at (1.1, 1.05) {$\frac{2\pi}{3}$};
        \fill[white] (1.2, .9) circle (0.075) node[black] {$V$};
        \draw[black] (1.9, 1.1) node {$\alpha$};
        \draw[black] (1.15, 1.6) node {$\beta$};
    \end{tikzpicture}
    \hfill
    \begin{tikzpicture}[line cap=round,line join=round,>=triangle 45,x=1cm,y=1cm,scale=4]
        
        \clip(0.35,-0.15) rectangle (2, 1.3);
        
        \coordinate (W1) at (1.6598760226544096,0.9969391568088625);
        \coordinate (yW1) at (1.5816947101735217,0);
        \coordinate (Q1) at (1.2221732712305209,0.9331368254121976);
        \coordinate (T) at (1.2005045377735641,0.9893779586480194);
        \coordinate (V) at (1.257657682658331,1.0601348326846107);
        
        \draw [shift={(W1)},]  (0:0.09) arc (0:85.51596484766341:0.09) ;
        \draw [shift={(W1)},]  (85.51596484766341:0.08028676698633866) arc (85.51596484766341:171.07080886590975:0.08028676698633866) ;
        \draw [shift={(yW1)},]  (0:0.0963441203836064) arc (0:85.51596484766341:0.0963441203836064);

        \draw [very thick] (0,4)-- (0,0);
        \draw [very thick] (0,0)-- (4,0);
        \draw [] (2.5816947101735217, 0) arc (0 : 180 : 1);
        \draw [] (0, .5904933321377084) arc (-90 : 90 : 1);
        \draw [thick, blue] (0.9881259165440313,1.0227464283958017) -- (T);
        \draw [thick, blue] (V)-- (T);
        \draw [thick, blue] (T)-- (Q1);
        \draw [,dashed] (Q1)-- (yW1);
        \draw [thick, blue] (V)-- (0.9833909434739174,1.7719933384558293);
        %\draw [,dashed] (0,1.5904933321377084)-- (0.9833909434739174,1.7719933384558293);
        %\draw [,dashed] (0.9833909434739174,1.7719933384558293)-- (1.3166362791519324,1.833498918644903);
        %\draw [] (0.9833909434739174,2.422818216943033)-- (0.9833909434739174,1.7719933384558293);
        \draw [thick, blue] (V)-- (W1);
        \draw [thick, blue] (W1)-- (2.2674807127100403,0.9969391568088625);
        \draw [,dashed] (yW1)-- (W1);
        \draw [,dashed] (W1)-- (1.681942908701103,1.2783279192981831);

        \def\nodesize{0.015}
        
        \fill[] (0,4) circle (\nodesize);
        \fill[] (0.5,0) %circle (\nodesize) 
            node[below left] {$A_1$};
        \fill[] (4,0) circle (\nodesize);
        \fill[] (yW1) circle (\nodesize) 
            node[black, below] {$y(W_1)$};
        \fill[white, shift={(Q1)}] (-0.085, -0.035) circle (0.05);
        \draw[shift={(Q1)}] (-0.06, -0.06) node[black] {$Q_1$};
        \fill[blue] (Q1) circle (\nodesize);
        %    node[black, below] {$Q_1$};
        %\fill[white, shift={(T)}] (-0.07, 0) circle (0.045);
        \fill[blue] (T) circle (\nodesize)  
            node[black, above left] {$T$};
        \fill[blue] (V) circle (\nodesize) 
            node[black, above right] {$V$};
        \fill[blue] (W1) circle (\nodesize)  
            node[black, below right] {$W_1$};
        \draw[shift={(W1)}] (0.13, 0.07) node {$\alpha$};
        \draw[shift={(yW1)}] (-0.12, 0.07) node {$\psi$};

        %\fill[] (0,1.5904933321377084) circle (\nodesize) node[below] {$e$};;
        %\fill[] (0.9833909434739174,1.7719933384558293) circle (\nodesize)
        %    node[below] {$g$};
        %\fill[] (0.9833909434739174,2.422818216943033) circle (\nodesize)
        %    node[below] {$h$};
        %\fill[] (1.3166362791519324,1.833498918644903) circle (\nodesize)
        %    node[below] {$j$};
    
    \end{tikzpicture}
    \hfill\phantom{}
    \caption{Cases 1 and 2}
    \label{caseVQandT}
\end{figure}
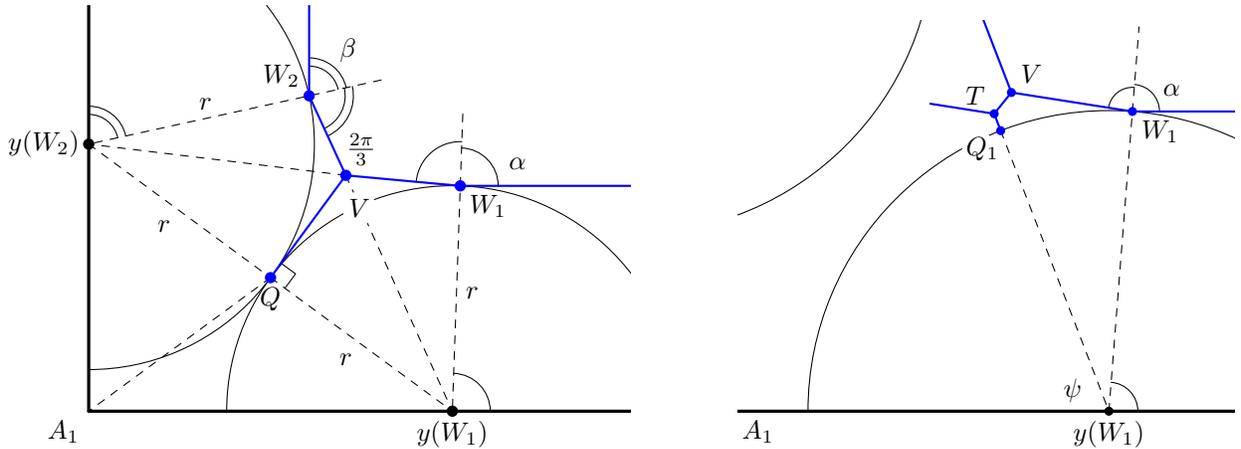

\paragraph{Case 2. $\St(p)$ has a branching point, say $T$.} Then at least one of terminals $Q_1$, $Q_2$ has degree 1, without loss of generality it is $Q_1$. 
If $Q_1$ is connected with $T$ by a segment of $\Sigma'$, then sum of the angles in quadrangle $VTy(W_1)W_1$ is equal to 
\[
\angle V + \angle T + \angle W_1  + \angle y(W_1)  = \frac{2\pi}{3} + \frac{2\pi}{3} + (\pi-\alpha) + (\pi-\alpha-\angle A_1y(W_1)Q_1) = 2\pi.
\]
Hence $2\alpha + \psi = \frac{4\pi}{3}$, where $\psi := \angle A_1y(W_1)Q_1$. From the other hand, Proposition~\ref{diffproposition} for $y(W_1)$ gives
$2\cos^2\alpha = \cos \psi$.
Substituting $\psi$, we have
\[
2\cos^2\alpha = \cos \left(\frac{4\pi}{3} - 2\alpha \right) = \cos (2\alpha) + 1.
\]
Summing up,
\begin{equation}
    2\sin\left(\frac{2\pi}{3}\right)\sin\left(2\alpha-\frac{2\pi}{3}\right) = 1, \quad \alpha =  \frac{1}{2}\left(\arcsin \left(\frac{1}{\sqrt 3}\right) + \frac{2\pi}{3}\right).
\label{aCasenotinP0}
\end{equation}
So $\alpha = 1.354\dots < \alpha_0 - \frac{1}{30}$, so $p$ does not lie in $P_0$. 

If $\Sigma'$ contains branching points or points of degree 2 in the path between $Q_1$ and $T$, then $2\alpha + \psi \leq \frac{4\pi}{3}$.
Let $t := \alpha + \frac{\psi}{2}$; note that $\frac{\pi}{3} \leq \alpha \leq t \leq \frac{2\pi}{3}$.
Similarly to~\eqref{aCasenotinP0},
\[
2\sin\left(t\right)\sin\left(2\alpha-t\right) = 1; \quad \alpha (t) = \frac{1}{2}\left(t + \arcsin \left(\frac{1}{2\sin t}\right)\right).
\]
The derivative
\[
\frac{d \alpha(t)}{d t} = \frac{1}{2}\left(1 - \frac{\cot t \csc t}{\sqrt{4 - \csc^2t}}\right)
\]
is positive on $(\pi/8, \pi) \supset (\pi/3 ,2\pi/3)$, so $\alpha (t)$ is maximal at $t = 2\pi/3$. Hence $p$ does not lie in $P_0$.

\paragraph{Case 3. $\St(p)$ consists of 2 line segments.} Since $Q_1$ and $Q_2$ are different, $\St(p) = [VQ_2] \cup [Q_2Q_1]$ (up to $Q_1-Q_2$ symmetry).
So $Q_1 \in B_r(A_1)$ or $Q_2 \in B_r(A_1)$.

\paragraph{Case 3a.} In this case $Q_1 \in B_r(A_1)$. There is a configuration better than $p_0$; we provide it explicitly (for simplicity $r=1$):
\[
Z_1 = (5.632993, 0.991394), \quad
W_1 = (1.545358, 0.991394), \quad
V = (1.108083, 1.108927), \quad
W_2 = (0.991495, 1.545397), \quad
\]
\[
Z_2 = (0.991495, 5.632993), \quad
Q_2 = (0.723714, 0.725155), \quad
Q_1 = (0.707224, 0.706989), \quad
\]
\[
y(W_1) = (1.414448, 0), \quad
y(W_2) = (0, 1.415255).
\]
Explicit calculations show that 
\[
0.999999 < |y(W_2)W_2|, |y(W_2)Q_2|, |A_1Q_1|, |Q_1y(W_1)|, |W_1y(W_1)| < 1. 
\]
Note that $W_1Z_1$ is parallel to $A_1A_2$ and $W_2Z_2$ is parallel to $A_1A_4$.
So $[y(W_1)C_1]$ is covered by $[W_1Z_1]$ and $[y(W_2)C_2]$ is covered by $[W_2Z_2]$.
Also $[Q_1Q_2]$ covers $[y(W_2)A_1]$ and $Q_1$ covers $[A_1y(W_1)]$.

The total length $\Sigma'(p) = [Z_1W_1] \cup [W_1V] \cup [VW_2] \cup [W_2Z_2] \cup [VQ_2] \cup [Q_2Q_1]$ is less than $9.647492$ which is smaller than the length of $\Sigma'_0$, see~\eqref{sigmazero}. 

The total length of a minimizer is $Per - 8\left(\frac{4}{\sqrt{6}} + 2 + 1\right) + 4\H(\Sigma'(p)) - 2r + o(r) \approx Per - 8.473981$. 
The term $2r + o(r)$ is the  difference between a solutions of Problem~\ref{TheProblem} and Problem~\ref{TheProblemCycle}.

\begin{figure}[h]
    \centering
    \begin{tikzpicture}[line cap=round,line join=round,>=triangle 45,scale=2.6]

    \coordinate (yW1) at(1.6064, 0);
    \coordinate (yW2) at(0,1.723);
    \coordinate (W1) at(1.8218,0.9765);
    \coordinate (W2) at(0.999, 1.76743);
    \coordinate (V) at(1.0373,1.34026);
    \coordinate (Q2) at(0.863,1.2177);
    \coordinate (Q1) at(0.8, 0.6);
    
    \draw [very thick] (0,3) -- (0,0) -- (3,0);
    \draw [] (1,0) arc (0 : 90 : 1);
    \draw [shift={(yW1)}] (1,0) arc (0 : 180 : 1);
    \draw [shift={(yW2)}] (0,1) arc (90 : -90 : 1);
    \draw [dashed] (yW1)-- (1.88642, 1.26945);
    \draw [dashed] (yW2)-- (1.2987, 1.780759);
    \draw [dashed] (yW2) -- (1.1219, 1.06611);
    \draw [dashed] (0,0) -- (Q1) -- (yW1);
    \draw [thick, blue] (0.9990000868263488,3)-- (W2) -- (V) -- (W1) -- (3,0.97653);
    \draw [thick, blue] (V) -- (Q2) -- (Q1);
    
    \def\zzz{0.13}
    \draw [thin,shift={(yW1)}] (\zzz, 0) arc(0 : 76 : \zzz);
    \draw [shift={(yW1)}] (\zzz/1.4, \zzz/1.4) node[right]{$\alpha$};
    \draw [thin,shift={(W1)}] (\zzz, 0) arc(0 : 76 : \zzz);
    \draw [thin,shift={(W1)}] (0.242*0.9*\zzz, 0.97*0.9*\zzz) arc(76 : 152 : 0.9*\zzz);
    
    \draw[thin,shift={(yW2)}] (0, \zzz) arc(90 : 5 : \zzz);
    \draw[thin,shift={(yW2)}] (0, 0.8*\zzz) arc(90 : 5 : 0.8*\zzz);
    \draw[thin,shift={(W2)}] (0, \zzz) arc(90 : 5 : \zzz);
    \draw[thin,shift={(W2)}] (0, 0.8*\zzz) arc(90 : 5 : 0.8*\zzz);
    \draw[thin,shift={(W2)}] (0.996 * .9 * \zzz, 0.087 * .9 * \zzz) arc(5 : -80 : .9* \zzz);
    \draw[thin,shift={(W2)}] (0.996 * .7*  \zzz, 0.087 * .7 *\zzz) arc(5 : -80 : .7 *\zzz);
    
    \draw[thin,shift={(Q2)}] (0.9*\zzz,0) arc(0: -29 : 0.9*\zzz);
    \draw[thin,shift={(Q2)}] (0.9*\zzz,0) arc(0 : 33 : 0.9*\zzz);
    \draw[thin,shift={(Q2)}] (0.75*\zzz, 0) -- (1.05*\zzz,0);
    \draw[thin,shift={(Q2)}] (0,-\zzz) arc(-90 : -96 : \zzz);
    \draw[thin,shift={(Q2)}] (0,-\zzz) arc(-90 : -29 : \zzz);
    \draw[thin,shift={(Q2)}] (0.425*\zzz, -0.736*\zzz) -- (0.575*\zzz,-0.996*\zzz);
    
    \draw [shift={(yW1)}] (\zzz/1.4, \zzz/1.4) node[right]{$\alpha$};
    \draw [shift={(yW2)}] (\zzz/1.7, -\zzz) node{$\gamma$};
    \draw [shift={(Q2)}]  (\zzz/1.7, -\zzz) node[below]{$q$};
    \draw [shift={(yW1)}] (-\zzz*1.4, \zzz/1.6) node[left]{$t$};
    \draw [shift={(0,0)}] (\zzz*1.4, \zzz/1.6) node[right]{$t$};
    
        \def\nodesize{0.02}
    
    \draw [] (0,0) node[below left] {$A_1$};
    \fill [] (yW1) circle (\nodesize) node[below] {$y(W_{1})$};
    \fill [] (yW2) circle (\nodesize) node[left] {$y(W_{2})$};
    \fill [blue] (W1) circle (\nodesize) node[below left, black] {$W_1$};
    \fill [blue] (W2) circle (\nodesize) node[above left, black] {$W_2$};
    \fill [blue] (V) circle (\nodesize) node[above right, black] {$V$};
    \fill [blue] (Q2) circle (\nodesize) node[left, black] {$Q_2$};
    \fill [blue] (Q1) circle (\nodesize) node[right, black] {$Q_1$};
    \end{tikzpicture}
    \hfill
    \begin{tikzpicture}[scale=2.6]
    
    \coordinate (W1) at (1.9487, 0.9853);
    \coordinate (yW1) at (1.778, 0);
    \coordinate (W2) at (0.996, 1.891);
    \coordinate (yW2) at (0, 1.801);
    \coordinate(V) at (1.106, 1.286);
    \coordinate (Q1) at (0.947, 0.557);
    \coordinate (Q2) at (0.435,0.9);
    
    \draw[very thick] (3, 0) -- (0, 0) -- (0, 3);
    \draw[shift={(yW1)}] (1, 0) arc (0 : 180 : 1);
    \draw[shift={(yW2)}] (0, 1) arc (90 : -90 : 1);
    \draw (1, 0) arc (0 : 90 : 1);
    \draw[dashed] (yW1) -- (2, 1.28);
    \draw[dashed] (yW2) -- (1.2948, 1.9180);
    \draw[dashed] (Q1) -- (yW1);
    \draw[dashed] (0,0) -- (Q2) -- (yW2);
    \draw[thick, blue] (3, 0.9853) -- (W1) -- (V) -- (W2) -- (0.996, 3);
    \draw[thick, blue] (V) -- (Q2) -- (Q1);

    \def\zzz{0.13}
    \draw [thin,shift={(yW1)}] (\zzz, 0) arc(0 : 80 : \zzz);
    \draw [shift={(yW1)}] (\zzz/1.4, \zzz/1.4) node[right]{$\alpha$};
    \draw [thin,shift={(W1)}] (\zzz, 0) arc(0 : 80 : \zzz);
    \draw [thin,shift={(W1)}] (0.1736*0.9*\zzz, 0.9848*0.9*\zzz) arc(76 : 152 : 0.9*\zzz);
    
    \draw[thin,shift={(yW2)}] (0, \zzz) arc(90 : 5 : \zzz);
    \draw[thin,shift={(yW2)}] (0, 0.8*\zzz) arc(90 : 5 : 0.8*\zzz);
    \draw[thin,shift={(W2)}] (0, \zzz) arc(90 : 5 : \zzz);
    \draw[thin,shift={(W2)}] (0, 0.8*\zzz) arc(90 : 5 : 0.8*\zzz);
    \draw[thin,shift={(W2)}] (0.996 * .9 * \zzz, 0.087 * .9 * \zzz) arc(5 : -80 : .9* \zzz);
    \draw[thin,shift={(W2)}] (0.996 * .7*  \zzz, 0.087 * .7 *\zzz) arc(5 : -80 : .7 *\zzz);
    
        \def\nodesize{0.02}
    
    \draw (0,0) node[below left] {$A_1$};
    \fill (yW1) circle(\nodesize) node[below] {$y(W_1)$};
    \fill (yW2) circle(\nodesize) node[left] {$y(W_2)$};
    
    \fill[blue] (W1) circle(\nodesize) 
        node[label={[black, shift={(0.3,-0.8)}]$W_1$}] {};
    \fill[blue] (W2) circle(\nodesize) node[above left, black] {$W_2$};
    \fill[blue] (V) circle(\nodesize);
    \draw[shift={(V)}] (-0.01, 0.02) node[left, black] {$V$};
    \fill[blue] (Q1) circle(\nodesize) node[right, black] {$Q_1$};
    \fill[blue] (Q2) circle(\nodesize) 
        node[label={[black,shift={(0.2,0.05)}]$Q_2$}] {};
    
    \end{tikzpicture}
    \caption{Cases 3a and 3b}
    \label{fig:TheOptimalCase}
\end{figure}
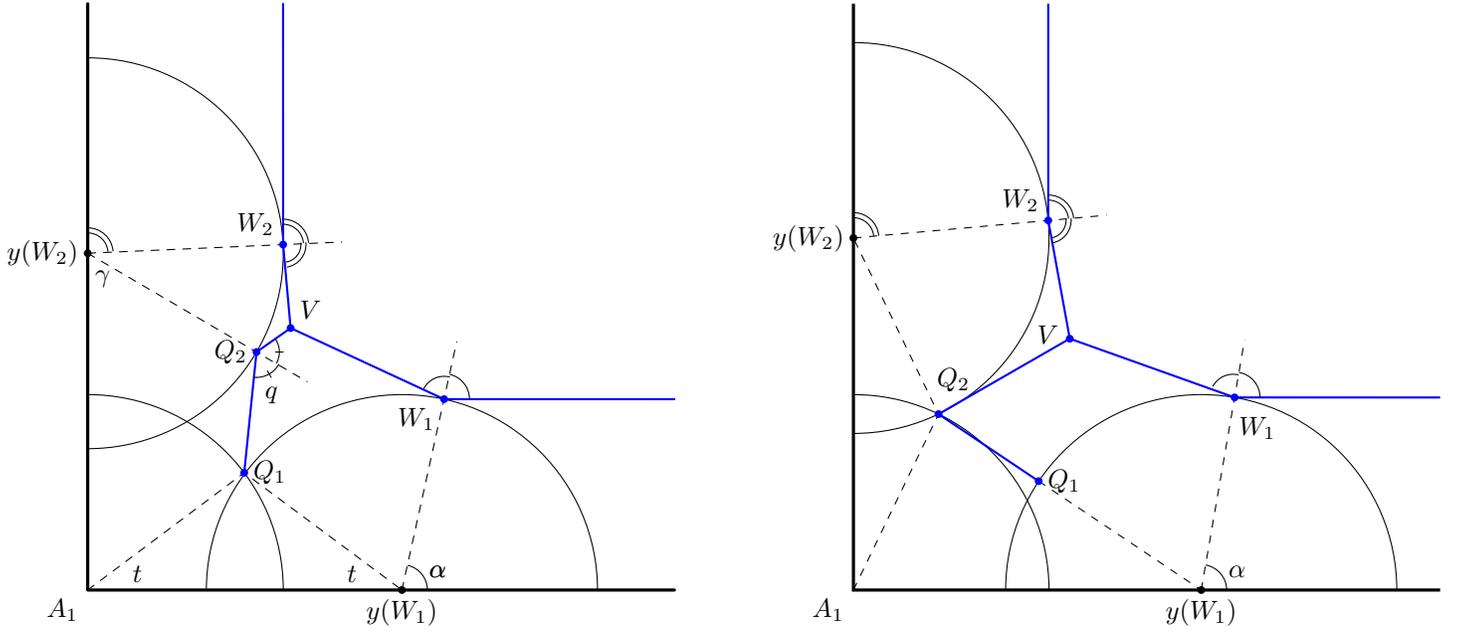

\paragraph{Case 3b. $Q_2 \in B_r(A_1)$.} Then by Proposition~\ref{diffproposition} 
\[
2\cos^2 \alpha = \cos \angle Q_1y(W_1)A_1,
\]
so 
\[
\angle Q_1y(W_1)A_1 = \arccos (2\cos^2 \alpha) > \frac{4\pi}{9}
\]
for $p \in P_0$. Since $|Q_2y(W_1)| = |Q_2Q_1| + |Q_1y(W_1)| > r$ and $|A_1Q_2| = r$, one has $\angle Q_2A_1y(W_1) > \angle Q_2y(W_1)A_1 > \frac{4\pi}{9}$. Hence $\angle y(W_2) A_1 Q_2 < \frac{\pi}{18}$.
Triangle $A_1Q_2y(W_2)$ is isosceles so $\angle A_1y(W_2)Q_2 = \angle y(W_2) A_1 Q_2 < \frac{\pi}{18}$.

Consider the sum of angles in quadrangle $y(W_2)W_2VQ_2$: angles $W_2$ and $Q_2$ are at least $\pi/2$, $\angle V = \frac{2\pi}{3}$, which implies that $\angle y(W_2) = \angle W_2y(W_2)Q_2 \leq \pi/3$. Summing up $\angle A_1y(W_2)W_2 = \angle y(W_2) A_1 Q_2 + \angle W_2y(W_2)Q_2$ is acute, which is a contradiction.

\paragraph{Case 4. $\St(p)$ consists of three line segments $[VQ_2]$, $[Q_2Q_1]$ and $[Q_1Q]$ (up to $Q_1-Q_2$ symmetry).}
Put $q = \frac{\angle VQ_2Q_1}{2}$, $t = \angle Q_1A_1y(W_1)$, $\delta = \angle Q_1y(W_1)A_1$ and $\gamma = A_1y(W_2)Q_2$ (see Fig.~\ref{Fig:case4}).

Application of Proposition~\ref{diffproposition} to $y(W_1)$ gives the following conditions: 
\begin{equation}
\begin{cases}
q+\gamma = 2\alpha - \pi/6\\
2\cos^2\left(\frac{11}{12}\pi - \alpha\right ) = 2\cos q \cos \gamma.
\end{cases}
\label{thefirsteqDiff}
\end{equation}
Transform the second equation
\[
\cos\left(\frac{11}{6}\pi - 2\alpha\right ) + 1 = \cos (q + \gamma) + \cos (q - \gamma) = \cos(2\alpha - \pi/6) + \cos (q - \gamma).
\]
Hence
\[
\cos (q - \gamma) = \cos\left(\frac{11}{6}\pi - 2\alpha\right ) + 1 - \cos(2\alpha - \pi/6) = 2\sin \frac{5\pi}{6}\sin(2\alpha-\pi) + 1 = 1-\sin(2\alpha).
\]
So we find $q$ and $\gamma$ in terms of $\alpha$:
\[
q = \frac{(q-\gamma)}{2} + \frac{(q + \gamma)}{2} = \frac{\arccos(1-\sin(2\alpha))}{2} + \alpha - \frac{\pi}{12};
\quad
\gamma =  2\alpha - \frac{\pi}{6} - q = - \frac{\arccos(1-\sin(2\alpha))}{2} + \alpha - \frac{\pi}{12}.
\]

Applying Proposition~\ref{diffproposition} to $y(W_2)$ we get the following conditions:
\[
\begin{cases}
t+2\delta = 2\pi/3 - 2\alpha + 2q\\
\cos^2\alpha = \cos \delta \cos (t+\delta).
\end{cases}
\]
Analogously to~\eqref{thefirsteqDiff}, we have
\[
\cos t = 1 + \cos (2\alpha) - \cos (2\pi/3 - 2\alpha + 2q).
\]
Hence
\[
t = \arccos \left( 1 + \cos (2\alpha) - \cos (2\pi/3 - 2\alpha + 2q) \right);
\quad
\delta = \frac{2\pi/3 - 2\alpha + 2q-t}{2} = \frac{\pi}{3} - \alpha + q - \frac{t}{2}.
\]

Now we can find $x$ and $y$ in terms of $\alpha$. By the sine rule for triangle $A_1Q_1y(W_1)$, we have
\[
x=|A_1y(W_1)|=\frac{|y(W_1)Q_1| \sin \angle A_1Q_1y(W_1)}{\sin \angle y(W_1)A_1Q_1}=\frac{r \sin (t+\delta)}{\sin t},
\]
and 
\begin{equation}
|A_1Q_1|=\frac{|y(W_1)Q_1| \sin \angle A_1y(W_1)Q_1}{\sin \angle y(W_1)A_1Q_1}=\frac{r \sin \delta}{\sin t}.
\label{aq1eq}
\end{equation}

Consider a quadrangle $A_1Q_1Q_2y(W_2)$. Projection of this quadrangle on line $A_1y(W_2)$ gives
\begin{equation}
y=|A_1y(W_2)|=|A_1Q_1|\cos \angle Q_1A_1y(W_2) + |Q_1Q_2| \cos \left(t+2\delta-\frac{\pi}{2}\right)+|Q_2y(W_2)|\cos \gamma,
\label{yeq}
\end{equation}
and projection on line $Q_2y(W_2)$ gives
\[
r=|Q_2y(W_2)|=|A_1y(W_2)|\cos \gamma + |Q_1Q_2|\cos (\pi - q) + |A_1Q_1|\sin (\gamma - t).
\]
Note that sines and cosines may be negative, which corresponds to the case on intersection of the projections of sides.

The previous condition implies
\[
|Q_1Q_2|=\frac{1}{\cos q}\left( y \cos \gamma + \frac{r \sin \delta \sin (\gamma - t)}{\sin t} -r \right).
\]

We substitute this expression into~\eqref{yeq} and get
\[
y=\frac{r \sin \delta}{\sin t} \sin t + \frac{\cos \left(t+2\delta-\frac{\pi}{2}\right)}{\cos q}\left( y \cos \gamma + \frac{r \sin \delta \sin (\gamma - t)}{\sin t} -r \right) + r \cos \gamma;
\]
\[
y=r \left(\sin \delta + \cos \gamma + \frac{\cos \left(t+2\delta-\frac{\pi}{2}\right)}{\cos q}\left(\frac{\sin \delta \sin (\gamma - t)}{\sin t} - 1 \right)\right)  \left(1 - \frac{\cos \left(t+2\delta-\frac{\pi}{2}\right) \cos \gamma}{ \cos q}\right)^{-1};
\]
\[
y=r \frac{ \cos \left( t + 2\delta - \frac{\pi}{2} \right) \left( 1 - \frac{\sin \delta \sin \gamma \cos t}{\sin t} + \sin \delta \cos \gamma \right) - (\sin \delta + \cos \gamma) \cos q }{ \cos \left( t + 2\delta - \frac{\pi}{2} \right) \cos \gamma - \cos q }.
\]

Recall that we need to check the following conditions:
\begin{itemize}
    \item [(i)] The distance between $A_1$ and $Q_1$ should be greater than $r$, otherwise it case 3a or case 5.
    From \eqref{aq1eq}, we get
    \[
    |A_1Q_1| = \frac{r \sin \delta}{\sin t} > r.
    \]
    
    \item [(ii)] The angle between $Q_2V$ and $A_1y(W_2)$ should be equal to $2\beta - \frac{2\pi}{3}$. 
    %\textcolor{red}{(наверное стоит на рисуночке (возможно отдельном) продлить $VQ_2$ до пересечения с прямой $A_1y(W_2)$)}. 
    Let $S$ be such a point that the triangle $W_1SW_2$ is equilateral, and $S$ and $Q_2$ are on the opposite sides of the line $W_1W_2$. Such $S$ lies on the ray $Q_2V$ since $V$ is the Steiner point in the triangle $W_1W_2Q$. Thus, the angle between $Q_2V$ and $A_1y(W_2)$ is equal to the angle between vectors $Q_2S$ and $(0,1)$.
    
    We express the coordinates of points $Q_2$, $W_1$, $W_2$, $S$ as functions of $\alpha$:
    \[
    Q_2 = (\sin \gamma , y - \cos \gamma),\quad W_1 = (x + \cos \alpha, \sin \alpha),\quad W_2 = (\sin \beta, y + \cos \beta),
    \]
    \[
    S = W_2 + (W_1 - W_2)\begin{pmatrix}
    \cos \frac{\pi}{3} & \sin \frac{\pi}{3}\\
    -\sin \frac{\pi}{3} & \cos \frac{\pi}{3}\\
    \end{pmatrix}.
    \]
    The second condition we need to check is
    \[
    2\beta - \frac{2\pi}{3} = \frac{(S - Q_2)\cdot (0, 1)^\mathsf{T}}{|SQ_2|}.
    \]
\end{itemize}

\begin{figure}[h]
    \centering
        \begin{tikzpicture}[line cap=round,line join=round,>=triangle 45, scale=3.5]
        \clip(-0.3,-0.3) rectangle (2.5 ,2.3);

        \coordinate (yW1) at (1.7587686969043363, 0);
        \coordinate (yW2) at (0, 1.5251434725073525);
        \coordinate (W1) at (1.9143092246823912,0.9878295117167355);
        \coordinate (W2) at (0.9944268882119729,1.6305719506385126);
        \coordinate (V) at (1.0740764315204938,1.2591572744280084);
        \coordinate (Q2) at (0.9102693894887516,1.111127001144467);
        \coordinate (Q1) at (0.8996130460006508,0.5117143417183696);
        \coordinate (Q) at (.7646710891, .4349571904);
        \coordinate (A) at (0, 0);
        
        \draw [very thick] (4,0) -- (0,0) -- (0,4);
        
        \draw [] (2.7587686969043363,0) arc (0 : 180 : 1);
        \draw [] (0,0.5251434725073525) arc (-90 : 90 : 1);
        
        \draw [thick, blue] 
            (0.9944268882119729, 2.720470361008444) -- 
            (0.9944268882119729, 1.6305719506385126) -- 
            (1.0740764315204938, 1.2591572744280084) --
            (W1) --
            (2.9775056909256583,0.9878295117167355);
            
        \draw [,dashed] (yW2)-- (1.3233823334691772,1.6654475884330455);
        \draw [,dashed] (yW1)-- (1.9669111149990113,1.3219012830366181);
        
        \draw [thick, blue] 
            (V) --
            (Q2) --
            (Q1) --
            (Q);
        
        \draw [, blue, dashed] 
            (Q) -- 
            (0, 0);
        
        \draw [,dashed] (yW2)-- (1.1255703781178321,1.0132019729907502);
        \draw [,dashed] (yW1)-- (0.6933471738604602,0.6345665919132711);
        
        \draw[blue, dashed] (.8192424506, 1.028867151) -- (0,.288528504667428);

        \draw [shift={(W1)},]  
            (0:0.09320635511994109) arc (0:81.05185357218244:0.09320635511994109);
        \draw [shift={(yW1)},]  
            (0:0.1118476261439293) arc (0:81.05185357218242:0.1118476261439293);
        \draw [shift={(W1)},] 
            (81.05185357218244:0.1118476261439293) arc (81.05185357218244:162.10370714436485:0.1118476261439293);
            
        \draw [shift={(W2)},]
            (6.051853572182469:0.09320635511994109) arc (6.051853572182469:90:0.09320635511994109);
        \draw [shift={(W2)},]
            (6.051853572182469:0.0689727027887564) arc (6.051853572182469:90:0.0689727027887564);
        \draw [shift={(W2)},]
            (-77.8962928556351:0.1118476261439293) arc (-77.8962928556351:6.051853572182465:0.1118476261439293);
        \draw [shift={(W2)},]
            (-77.8962928556351:0.08761397381274462) arc (-77.8962928556351:6.051853572182465:0.08761397381274462);
        \draw [shift={(yW2)},] (
            6.051853572182454:0.1118476261439293) arc (6.051853572182454:90:0.1118476261439293);
        \draw [shift={(yW2)},] 
            (6.051853572182454:0.08761397381274462) arc (6.051853572182454:90:0.08761397381274462);
            
        \draw [shift={(Q2)},] 
            (-24.457394260120623:0.09320635511994109) arc (-24.457394260120623:42.10370714436493:0.09320635511994109);
        \draw[] (0.9885572752163387,1.1232790070244527) -- (1.0161882937084283,1.1275679502762121);
        \draw [shift={(Q2)},]
            (-91.01849566460629:0.1118476261439293) arc (-91.01849566460629:-24.45739426012064:0.1118476261439293);
        \draw[] (0.9625098801185278,1.0283694226481979) -- (0.9774357345841781,1.0047244002206925);
        
        \draw [shift={(Q1)},]
            (88.98150433539372:0.1118476261439293) arc (88.98150433539372:149.22191120368805:0.1118476261439293);
        \draw[] (0.8433032178795815,0.5917586351863084) -- (0.8272146955592761,0.6146284333200052);
        \draw[] (0.8612514088781222,0.6017491708517958) -- (0.8502909411288282,0.6274734077470607);
        \draw [shift={(Q1)},]
            (149.22191120368805:0.09320635511994109) arc (149.22191120368805:209.63191386800077:0.09320635511994109);
        \draw[] (0.8207481420287545,0.5041650887681426) -- (0.7929134700386734,0.5015006465504157);
        \draw[] (0.8209149307803554,0.5208396340082653) -- (0.7931391254084862,0.5240603254046989);

        \draw [decorate , decoration={brace, mirror, amplitude=3pt},      
            xshift=0pt,yshift=-2pt]
        (0, 0) -- (1.7587686969043363, 0) node [midway,yshift=-0.3cm] {$x$};
        \draw [decorate , decoration={brace, amplitude=3pt},      
            xshift=-2pt,yshift=0pt]
        (0, 0) -- (0, 1.5251434725073525) node [midway,xshift=-0.3cm] {$y$};

        \def\nodesize{0.02}
        
        \fill [] (yW1) circle (\nodesize)
            node[below right]{$y(W_1)$};
        \fill [] (yW2) circle (\nodesize)
            node[above left]{$y(W_2)$};
        \fill [blue] (W1) circle (\nodesize)
            node[black, below left]{$W_1$};
        \fill [blue] (W2) circle (\nodesize)
            node[black, above left]{$W_2$};
        \fill [blue] (V) circle (\nodesize);
        %\fill [] (1.3233823334691772,1.6654475884330455) circle (\nodesize);
        %\fill [] (1.9669111149990113,1.3219012830366181) circle (\nodesize);
        \fill [blue] (Q2) circle (\nodesize);
        \draw [shift=(Q2)] (0, -0.02) 
            node[black, left] {$Q_2$};
        \fill [blue] (Q1) circle (\nodesize);
        \draw [shift=(Q1)] (0, 0.02)
            node[black, right] {$Q_1$};
        \fill [blue] (Q) circle (\nodesize)
            node[black, below] {$Q$};
        %\fill [] (1.1255703781178321,1.0132019729907502) circle (\nodesize);
        %\fill [] (0.6933471738604602,0.6345665919132711) circle (\nodesize);
     
        \fill [] (A) circle (\nodesize)
             node[black, below left] {$A_1$};

        \draw [shift={(V)}] (0.04, -0.02)
            node[below]{$V$};
        \draw [shift={(V)}] (0, 0.05)
            node[right]{\tiny $2\pi/3$};
        \fill[white, shift={(V)}] (-0.1, 0.03) circle (0.06);
        \draw [shift={(V)}] (0, 0.03)
            node[left]{\tiny $2\pi/3$};
        
        \draw [shift={(W1)}] (0.08,0.07) node[right] {$\alpha$};
        \draw [shift={(W2)}] (0.07,0.08) node[right] {$11\pi/12 - \alpha$};
        
        \draw (0.25, 0.07) node[right]{$t$};
        \draw [shift={(yW2)}] (0.07,-0.1) node[below]{$\gamma$};
        \draw [shift={(yW1)}] (-0.25, 0.07) node[left]{$\delta$};
        
        \draw [shift={(Q2)}] (0.07, -0.15) node[right]{$q$};
    \end{tikzpicture}
    \caption{Case 4}
    \label{Fig:case4}
\end{figure}
% \begin{center}
%     \includegraphics[scale=0.8]{pictures/Thecase.jpg}
% \end{center}

Now we can substitute $\alpha=\alpha_0+\Delta\alpha$ into all the expressions above and treat each of these expressions as a function of $\Delta\alpha$. We bound each of these functions in $\Delta \alpha \in [-1/30,1/30]$ from below and from above by two 6th order polynomials. To obtain these polynomials we consider the following functions:
\[
\sin(x+c), \quad 0<x+c<\frac{\pi}{2},
\]
\[
\cos(x+c), \quad 0<x+c<\frac{\pi}{2},
\]
\[
\arccos(x+c), \quad 0<c<\frac{85}{100}, \quad 0<x+c<\frac{85}{100},
\]
\[
\frac{1}{x+c}, \quad \frac{4}{10}<c, \quad \frac{4}{10}<x+c<2c,
\]
\[
\frac{1}{x+c}, \quad \frac{12}{10}<c, \quad \frac{12}{10}<x+c<2c,
\]
\[
\sqrt{x+c}, \quad \frac{3}{2}<c, \quad \frac{3}{2}<x+c<2c.
\]
We bound each of these functions by 6th order polynomials from below and from above, using Taylor approximation series with the Lagrange form of the remainder. Note that 6th order derivative of each of these functions does not change sign and is monotonous on the corresponding value interval of $x+c$. It means that in each case the remainder can be bound by zero from one side, and by $f^{(6)}(z) \cdot x^6 / 6!$ on the other side, where $z$ is one of the endpoints of the value interval of $x+c$.

Now suppose we want to estimate $f(g(\Delta \alpha))$, where
\[
q_1(x,c)\leq f(x+c)\leq q_2(x,c), \quad L_c<c<R_c, \quad L<x+c<R,
\]
is one of the functions above, and function $g$ is already estimated:
\[
p_1(\Delta \alpha)\leq g(\Delta \alpha)\leq p_2(\Delta \alpha).
\]
We split polynomials $p_1$, $p_2$ into non-constant and constant parts:
\[
p_i(\Delta \alpha)=p_i^{*}(\Delta \alpha)+c_i, \text{ where } p_i^{*}(0)=0.
\]
If the inequalities
\[
L_c<c_i<R_c, \quad L<p_1(\Delta \alpha)\leq p_2(\Delta \alpha)<R \text{ for } \Delta \alpha \in \left[-\frac{1}{30},\frac{1}{30}\right]
\]
hold, then (assuming that $f$ is increasing, otherwise $p_1^{*},c_1$ and $p_2^{*},c_2$ should be swapped)
\[
q_1(p_1^{*}(\Delta\alpha),c_1)\leq f(g(\Delta\alpha))\leq q_2(p_2^{*}(\Delta\alpha),c_2).
\]
We have bounded $f(g(\Delta\alpha))$ by 36th order polynomials. To obtain bounds by 6th order polynomials, we loose these bounds, using inequalities of the form
\begin{equation}\label{eq-reduceDegree}
-\frac{|a|}{30^{k-6}}x^6\leq ax^k\leq \frac{|a|}{30^{k-6}}x^6, \quad k>6.
\end{equation}

Arithmetic operations are handled similarly: suppose that functions $f$, $g$ are already estimated (all the following inequalities hold for all $\Delta \alpha \in \left[-\frac{1}{30},\frac{1}{30}\right]$):
\[
p_1(\Delta \alpha)\leq f(\Delta \alpha)\leq p_2(\Delta \alpha); \quad q_1(\Delta \alpha)\leq g(\Delta \alpha)\leq q_2(\Delta \alpha),
\]
then
\[
p_1(\Delta \alpha) + q_1(\Delta \alpha) \leq f(\Delta \alpha) + g(\Delta \alpha)\leq p_2(\Delta \alpha) + q_2(\Delta \alpha),
\]
\[
p_1(\Delta \alpha) - q_2(\Delta \alpha) \leq f(\Delta \alpha) - g(\Delta \alpha)\leq p_2(\Delta \alpha) - q_1(\Delta \alpha).
\]
Additionally, if $p_1(\Delta \alpha) \geq 0$, $q_1(\Delta \alpha) \geq 0$, then
\[
p_1(\Delta \alpha) \cdot q_1(\Delta \alpha) \leq f(\Delta \alpha) \cdot g(\Delta \alpha)\leq p_2(\Delta \alpha) \cdot q_2(\Delta \alpha).
\]
The obtained bounds in the latter case are polynomials of degree 12, we replace them by 6th order polynomials using \eqref{eq-reduceDegree}.

When checking the inequalities of the form
\[
L<p_1(\Delta \alpha)\leq p_2(\Delta \alpha)<R \text{ for } \Delta \alpha \in \left[-\frac{1}{30},\frac{1}{30}\right],
\]
we reduce the degrees of $p_1$, $p_2$ from 6 to 2, using inequalities similar to~\eqref{eq-reduceDegree}, and find extremal values of 2nd order polynomials explicitly.

Finally, we obtain polynomial bounds
\[
r_1(\Delta\alpha)\leq|A_1Q_1|-r\leq r_2(\Delta\alpha),
\]
\[
s_1(\Delta\alpha)\leq q + \gamma + \frac{(S - Q_2)\cdot (0, 1)^\mathsf{T}}{|SQ_2|} - \pi\leq s_2(\Delta\alpha).
\]

Since the conditions (i), (ii) should hold, the following conditions should also hold:
\[
r_2(\Delta\alpha)>0, \quad s_2(\Delta\alpha)\geq 0.
\]
Once again, we loose these conditions by replacing $r_2$ and $s_2$ with 2nd order polynomials, and solve the obtained inequalities explicitly.

The first inequality is satisfied only if
\[
\Delta \alpha \in \left[-\frac{1}{30}, -0.004\right];
\] 
the second inequality is satisfied only if
\[
\Delta \alpha \in \left[-0.0008,\frac{1}{30}\right].
\]
This means that conditions (i) and (ii) are not satisfied simultaneously for any
$\alpha\in\left[\alpha_0-\frac{1}{30},\alpha_0+\frac{1}{30}\right]$.

The described calculations are performed in SageMath~\cite{sagemath}; interval arithmetic library~\cite{mpfi} is used to take into account the numerical errors. The source code can be found at~\cite{c0pymaster_minimizer}.

\begin{figure}
\centering
\begin{tikzpicture}[line cap=round,scale=2.7]
    \coordinate (W1) at (2.947, 0.985);
    \coordinate (yW1) at (2.778, 0);
    \coordinate (W2) at (0.996, 2.062);
    \coordinate (yW2) at (0, 1.972);
    \coordinate (V) at (1.07, 1.656);
    \coordinate (Q1) at (1.819, 0.291);
    \coordinate (Q2) at (0.877, 1.492);
    \coordinate (Q) at (0.799, 0.601);

    \draw[very thick] (3.5, 0) -- (0, 0) -- (0, 3);
    \draw[shift={(yW1)}] (-1, 0) arc (180 : 60 : 1);
    \draw[shift={(yW2)}] (0,-1) arc (-90 : 90 : 1);
    \draw (1, 0) arc (0 : 90 : 1);
    
    \draw[dashed] (yW1) -- (2.9977, 1.2805);
    \draw[dashed] (yW2) -- (1.2948, 2.089);
    \draw[dashed] (Q1) -- (yW1);
    
    \draw[dashed] (0,0) -- (Q);
    
    \draw[blue, thick] (3.5, 0.985) -- (W1) -- (V) -- (W2) -- (0.996, 3);
    \draw[blue, thick] (V) -- (Q2) -- (Q) -- (Q1);

    \def\zzz{0.13}
    \draw [thin,shift={(yW1)}] (\zzz, 0) arc(0 : 76 : \zzz);
    \draw [shift={(yW1)}] (\zzz/1.4, \zzz/1.4) node[right]{$\alpha$};
    \draw [thin,shift={(W1)}] (\zzz, 0) arc(0 : 76 : \zzz);
    \draw [thin,shift={(W1)}] (0.242*0.9*\zzz, 0.97*0.9*\zzz) arc(76 : 152 : 0.9*\zzz);
    
    \draw[thin,shift={(yW2)}] (0, \zzz) arc(90 : 5 : \zzz);
    \draw[thin,shift={(yW2)}] (0, 0.8*\zzz) arc(90 : 5 : 0.8*\zzz);
    \draw[thin,shift={(W2)}] (0, \zzz) arc(90 : 5 : \zzz);
    \draw[thin,shift={(W2)}] (0, 0.8*\zzz) arc(90 : 5 : 0.8*\zzz);
    \draw[thin,shift={(W2)}] (0.996 * .9 * \zzz, 0.087 * .9 * \zzz) arc(5 : -80 : .9* \zzz);
    \draw[thin,shift={(W2)}] (0.996 * .7*  \zzz, 0.087 * .7 *\zzz) arc(5 : -80 : .7 *\zzz);

        \def\nodesize{0.02}
    
    \draw [] (0,0) node[below left] {$A_1$};
    \fill [] (yW1) circle (\nodesize) node[below] {$y(W_{1})$};
    \fill [] (yW2) circle (\nodesize) node[left] {$y(W_{2})$};
    \fill [blue] (W1) circle (\nodesize) node[below left, black] {$W_1$};
    \fill [blue] (W2) circle (\nodesize) node[above left, black] {$W_2$};
    \fill [blue] (V) circle (\nodesize) node[above right, black] {$V$};
    \fill [blue] (Q2) circle (\nodesize) node[left, black] {$Q_2$};
    \fill [blue] (Q1) circle (\nodesize) 
        node[label={[black, shift={(0.4, -0.2)}] $Q_1$}] {};
    \fill [blue] (Q) circle (\nodesize) node[above right, black] {$Q$};
\end{tikzpicture}
\caption{Case 5}
\label{Fig:case5}
\end{figure}

\paragraph{Case 5. $\St(p)$ consists of three line segments $[VQ_2]$, $[Q_2Q]$ and $[QQ_1]$ (up to $Q_1-Q_2$ symmetry).}
%Again we may write equations~\ref{thefirsteqDiff} to find $\gamma$ and $q$ as in the previous cases. 

By Proposition~\ref{diffproposition} 
\[
2\cos^2 \alpha = \cos \angle Qy(W_1)A_1,
\]
so 
\[
\angle Qy(W_1)A_1 = \arccos (2\cos^2 \alpha) > \frac{\pi}{4}
\]
for $p \in P_0$ (here we need a weaker bound than in Case 3b).

From the other hand $\angle A_1Qy(W_1) \geq \frac{\pi}{2}$, so $\angle QA_1y(W_1) \leq \frac{\pi}{4}$. 
Summing up we have that $\angle Qy(W_1)A_1 > \angle QA_1y(W_1)$, so $|Qy(W_1)| < |QA_1| = r$ which is a contradiction.

\section{Conclusion}
\label{discussion}

Now we briefly enlist related open problems. The most interesting is to find the set of minimizers for a circumference. Miranda, Paolini and Stepanov conjectured that an arbitrary minimizer for a circumference of radius $R > r$ is a horseshoe. It is also interesting to find a set of minimizers for long enough stadium.

It make sense to determine how strong Theorem~\ref{horseshoeT} can be. The condition $R > 5r$ seems excessive;  for $R < 1.75 r$ the statement is false.
For details see~\cite{cherkashin2020minimizers}. 

Here we focus on the uniqueness of a minimizer (up to symmetries) for a rectangle. Recall that the topology of the Case 3a leads to a minimizer (see Fig.~\ref{fig:TheOptimalCase}). Put $q = \frac{\angle Q_1Q_2V}{2}$, $\gamma = A_1y(W_2)Q_2$, $t = Q_1A_1y(W_1)$. Denote $Q_1Q_2$ by $a$ and $VW_1$ by $b$. Then one can obtain the following system
\[
\begin{cases}
\alpha + \beta = \frac{11\pi}{12}\\
\cos(q-\gamma) = 1 - \sin 2\alpha\\
q + \gamma = 2\alpha - \pi/6\\
4\cos^2 \alpha = \cot t \cdot \cos (q-\gamma) - \sin (q - \gamma)\\
2r\cos \frac{\alpha + t}{2} = a \cos \left (\frac{\pi}{3} + \frac{3\alpha + t}{2} - 2q \right ) + \frac{4}{\sqrt 3}r\cos \frac{\beta + \gamma}{2} \sin \frac{3\beta + \gamma}{2}\cos\left(\frac{3\alpha + t}{2} - \frac{2\pi}{3} \right ) - b\cos \frac{3\alpha+t}{2}\\
b = - 2r\cos \frac{\alpha + t}{2}\cos \frac{3\alpha + t}{2} + a\cos \left( 2q - \frac{\pi}{3}\right ) - \frac{2}{\sqrt 3}r\cos \frac{\beta + \gamma}{2}\sin \frac{3\beta+\gamma}{2}\\
x = 2r\cos t\\
y = r\cos \gamma + r\sin t + a\cos(q - \gamma).\\
\end{cases}
\]
The second and the forth equations are applications of Proposition~\ref{diffproposition} to points $y(W_2)$ and $y(W_1)$ respectively (we do not enlist case, corresponding to right-hand side of the fourth equation in Section~\ref{diff}, but it can be found in~\cite{cherkashin2020minimizers}). The fifth and the sixth equations are projection of quadrangle $Q_1Q_2VW_1$ on sides $Q_1W_1$ and $VW_1$.

The length equals to
\[
\phi = const - x - y - r\cos \beta - r\cos \alpha + a + b + \frac{4}{\sqrt{3}}r\cos\frac{\beta + \gamma}{2}\sin\frac{3\beta + \gamma}{2} + \frac{4}{\sqrt{3}} r\cos\frac{\beta+\gamma}{2} \sin\left ( \frac{3\beta+\gamma}{2} - \frac{2\pi}{3} \right).
\]

Computer simulation shows that the length has unique minimum for $p \in P_0$ with respect to the system. 
But a computer-assistant proof of this fact is too ugly to write it here.

\paragraph{Acknowledgments.} We thank Misha Basok for helping us to convince Fedor Petrov not to demand the proof of uniqueness of $\Sigma$.

The results of Sections 4.2 and 7 of Danila Cherkashin and Yana Teplitskaya are supported by <<Native towns>>, a social investment program of PJSC <<Gazprom Neft>>.
The research of Danila Cherkashin in Section 6 was supported by the Grant of the Government of the Russian Federation for the state support of scientific research supervised by leading scientists (agreement № 075-15-2019-1926).
The work of Yana Teplitskaya on Section 2 is supported by Russian Foundation for Basic Research, grant 17-01-00678.
The work of Alexey Gordeev on Section 5, as well as the implementation of computer programs~\cite{c0pymaster_minimizer}, are supported by Ministry of Science and Higher Education of the Russian Federation, agreement № 075–15–2019–1620. The work of Georgiy Strukov on Section 4.1 is supported by Ministry of Science and Higher Education of the Russian Federation, agreement № 075–15–2019–1619.

\bibliographystyle{plain}
\bibliography{main}

\end{document}